\documentclass[
    a4paper,
    11pt,
]{article}

\usepackage{amsmath,amssymb,amsfonts,amsthm}
\usepackage{mathtools,bm}
\usepackage{booktabs,enumitem,placeins}
\usepackage{color,graphicx}
\usepackage[margin=30mm,top=40mm]{geometry}
\usepackage[font=small]{caption}
\usepackage{cite}
\usepackage{url}
\usepackage{indentfirst}
\usepackage{hyperref}
\hypersetup{
  colorlinks=true,
  linkcolor=blue,
  citecolor=red,
  urlcolor=blue
}



\numberwithin{equation}{section}

\newtheorem{thm}{Theorem}[section]

\newtheorem{lem}[thm]{Lemma}
\newtheorem*{lem*}{Lemma}
\newtheorem{rem}[thm]{Remark}

\newcommand{\R}{\mathbb R}
\newcommand{\dd}{\,{\rm d}}
\newcommand{\id}{{\rm id}}
\newcommand{\nabs}{\nabla_{\!s}}
\newcommand{\Ds}{\Delta_s}
\newcommand{\divs}{\nabla_{\!s}\!\cdot}
\newcommand{\norm}[1]{\bigl\|#1\bigr\|}

\newcommand{\ipd}[1]{{\bigl(#1\bigr)}}

\newcommand{\E}{\mathcal E}
\newcommand{\Rays}{\mathcal R}
\newcommand{\Lag}{\mathcal L}

\newcommand{\B}{\mathcal B}
\newcommand{\Id}{{\rm Id}}
\newcommand{\D}{\mathcal D}
\newcommand{\vol}{\operatorname{vol}}
\newcommand{\bkap}{\overline{\varkappa}}

\newcommand{\ttau}{\Delta t}
\newcommand{\Gt}{{\Gamma(t)}}
\newcommand{\Gm}{{\Gamma^m}}
\newcommand{\rd}{{\rm d}}
\newcommand{\bR}{{\mathbb R}}
\newcommand{\dH}{\rd\mathcal{H}}
\newcommand{\bV}{{\mathbb V}}

\newcommand{\Example}[2]{%
  \par\vspace{0.5em}
  \noindent\textbf{Example #1: #2.}
  \par\vspace{0.5em}
}

\begin{document}

\title{A constrained Onsager variational framework for parametric approximations of Willmore and Helfrich flows}

\author{Quan Zhao\footnotemark[1]}

\renewcommand{\thefootnote}{\fnsymbol{footnote}}
\footnotetext[1]{School of Mathematical Sciences, University of Science and
Technology of China, 230026 Hefei, China\\
\texttt{\href{mailto:quanzhao@ustc.edu.cn}{quanzhao@ustc.edu.cn}}}

\date{}
\maketitle

\begin{abstract}
\noindent
We develop a constrained Onsager variational framework for parametric finite element approximations of Willmore and Helfrich flows.  The weak curvature relation is used as a PDE constraint to express the first variation of the bending energy in curvature-vector form.  Onsager's principle then combines this variation with dissipation based on the normal velocity, while an independent relaxed minimal-deformation-rate (relaxed-MDR) constraint selects the tangential velocity. The resulting mixed formulation admits a semidiscretization by continuous piecewise linear elements that satisfies an exact energy-dissipation law and exactly preserves area and volume when the corresponding constraints are imposed. {The framework accommodates constant spontaneous curvature} and applies to both closed surfaces and open surfaces with fixed boundaries under Navier or clamped conditions.  We also propose a linearly implicit fully discrete scheme, prove its unique solvability, and present numerical experiments demonstrating convergence, energy decay, constraint preservation, and effective mesh redistribution. 
\end{abstract}



\setlength\parskip{1ex}
\renewcommand{\thefootnote}{\arabic{footnote}}
\setcounter{equation}{0}
\setlength\parindent{24pt}

\section{Introduction}\label{sec:introduction}

The Willmore and Helfrich energies are fundamental curvature functionals in
differential geometry and in models of elastic interfaces and biological
membranes.  The Willmore energy measures bending through the squared mean
curvature \cite{Willmore93}, whereas the Helfrich model introduces spontaneous
curvature and is commonly supplemented with area and volume constraints to
describe lipid bilayers, vesicles, and red blood cells
\cite{Canham1970minimum,Helfrich73elastic,Seifert97}.  The associated
$L^2$-gradient dynamics give rise to Willmore flow and, when both area and
volume are constrained, Helfrich flow.  These fourth-order, highly nonlinear
geometric flows pose substantial analytical and numerical challenges.  For
analytical results, see
\cite{Kuwert01willmore,DziukKS02,KS02,Simonett05willmore,Blatt09singular,PalmurellaR22,Rupp23volume,Schlierf25spont}. Numerical approximations include level-set
methods \cite{Droske04level}, phase-field methods
\cite{Du04phase,FrankenRW13,Bretin15phase,RumpfSS25preprint}, and
front-tracking or parametric finite element methods
\cite{Mayer02numerical,ClarenzDDRR04,Rusu05,Dziuk08,BGN08willmore,BONITO2010,pwfade,pwfopen,OlischlagerR09,Deckelnick09error,BalzaniR12,KovacsLL21,Duan2021high,BaoL25,GNZ26}. 
Preserving the underlying gradient-flow structure, in particular the dissipation of bending energy, is a central objective in the design of such methods.

In pioneering work, Dziuk \cite{Dziuk08} proposed a variational parametric finite element formulation of Willmore flow whose continuous-in-time semidiscretization satisfies a discrete energy-dissipation law. Building on Dziuk's work, Barrett, Garcke, and N\"urnberg introduced the BGN tangential motion within a PDE-constrained variational formulation, establishing semidiscrete energy stability for Helfrich-type models with spontaneous curvature and for open-surface boundary-value problems \cite{pwfade,pwfopen}. At the fully discrete level, energy stability has subsequently been pursued through different approaches. Within the framework of Dziuk's formulation, the temporal collocation method of \cite{Duan2021high} yields high-order energy-diminishing schemes, but leads to highly nonlinear algebraic systems. Minimizing-movement schemes \cite{OlischlagerR09,BalzaniR12}, on the other hand, inherit energy decay naturally at the time-discrete level, at the expense of solving a highly nonlinear minimization problem at each time step. More recently, energy-stable parametric schemes have been developed by treating the curvature as an additional evolving unknown \cite{BaoL25,GNZ25willmore,GNZ26}. Such curvature-evolving formulations may require additional regularity and careful treatment of nonsmooth initial geometries and long-time computations, which can pose challenges in simulations of complex evolving surfaces.

For a smooth evolving surface, only the normal component of the velocity changes the geometry, whereas the tangential component changes the surface parametrization.  This tangential freedom becomes crucial after spatial discretization.  For example, a purely normal approximation, such as Dziuk's method \cite{Dziuk08}, may lead to severe mesh distortion.  In contrast, the BGN
tangential motion induced by the curvature equation yields asymptotically equidistributed curves and conformal polyhedral surfaces in three dimensions \cite{BGN07,BGN08parametric,Barrett20}.  The minimal-deformation-rate (MDR) principle has subsequently motivated formulations that reduce mesh deformation and mitigate the small-time-step mesh instabilities of BGN discretizations
\cite{Hu22evolving,Gao26,Gao26dualMDR}.
{Hu and Li \cite{Hu22evolving} introduced the MDR tangential velocity in their convergence analysis of evolving finite element methods for mean curvature and Willmore flows. More recently, Gao, Li and Tang \cite{Gao26dualMDR} combined dual formulations with MDR tangential motion to obtain linearly implicit, energy-stable schemes for mean curvature flow and surface diffusion. }
Other choices include DeTurck-type
reparametrizations and uniformization strategies \cite{DeTurck17,Duan24new,PAN26}.  More fundamentally, the distinct roles of the two velocity components motivate a normal--tangential velocity-splitting strategy, in which the normal velocity realizes the geometric gradient flow,
whereas the tangential velocity is selected independently for mesh
redistribution.  Such a separation has recently been used in an energy-stable
curvature-evolution formulation for Willmore flow \cite{GNZ25willmore,GNZ26} and in a
{minimizing-movement framework incorporating admissible tangential
motion developed by Liu and Zhao \cite{LZ26}.}

To realize this separation within an energy-stable weak formulation, we turn
to Onsager's variational principle.  Rooted in the reciprocal relations of
linear irreversible thermodynamics \cite{Onsager31a,Onsager31b}, the principle
characterizes dissipative evolution by minimizing, over admissible rates, a
Rayleighian that combines a dissipation potential with the variation of the free energy.  It has found successful applications in fluid
dynamics and soft-matter physics \cite{QianWS06,XuDD16,Doi11,Doi15}, as well as
in solid-state dewetting \cite{JiangZQSB19,ZhaoJWSQB24} governed by surface diffusion flow.  More recently, it has
been used to construct variational discretizations of mean curvature flow
\cite{LX25}.  For the fourth-order Willmore and Helfrich flows considered here,
however, a key issue remains: how to incorporate an independently selected
tangential motion into the curvature-vector PDE system while retaining the
energy-dissipation structure.

{To address this challenge, we develop a two-stage
curvature-constrained Onsager formulation for Willmore and Helfrich flows,
with tangential mesh redistribution imposed as a separate admissibility
condition. We adopt the existing relaxed-MDR condition \cite{LZ26} as a
linear, adjustable tangential rate constraint that interpolates between
purely normal and MDR motion; see \eqref{eq:rmdr-strong}.}
In the first stage, the weak curvature relation is
imposed as a PDE constraint to derive a curvature-vector representation of the
bending-energy variation.  Unlike the BGN construction \cite{pwfade}, this stage neither
closes the evolution system nor determines the tangential velocity.  In the
second stage, Onsager's principle determines the surface velocity by combining
this energy variation with dissipation associated with the normal velocity,
while the relaxed-MDR condition is imposed independently as a rate constraint.
The resulting mixed formulation admits a continuous piecewise linear
semidiscretization that satisfies an exact energy-dissipation law and exactly
preserves area and volume when the corresponding constraints are imposed.  {The
framework also accommodates open surfaces and constant spontaneous curvature.}

The rest of the paper is organized as follows.  In Section~\ref{sec:mathf}, we introduce the geometric setting and the gradient flow.  Next in Section~\ref{sec:variational}, we derive the mixed weak formulation based on the constrained Onsager principle.  In Section~\ref{sec:semidiscrete}, we develop the semidiscrete finite element approximation and establish its energy law.  In Section~\ref{sec:fully-discrete}, we introduce the fully discrete
scheme and prove its unique solvability.  In Section~\ref{sec:numerics}, we
report numerical experiments on convergence, energy decay, {constraint preservation}, and mesh quality.  We conclude in Section~\ref{sec:con}.

\section{Geometric setting and bending flows}\label{sec:mathf}

\subsection{Geometric notation and identities}\label{subsec:geometric-setting}

\begin{figure}[!htbp]
\centering
\includegraphics[width=0.5\textwidth,trim=0 150 0 0,clip]{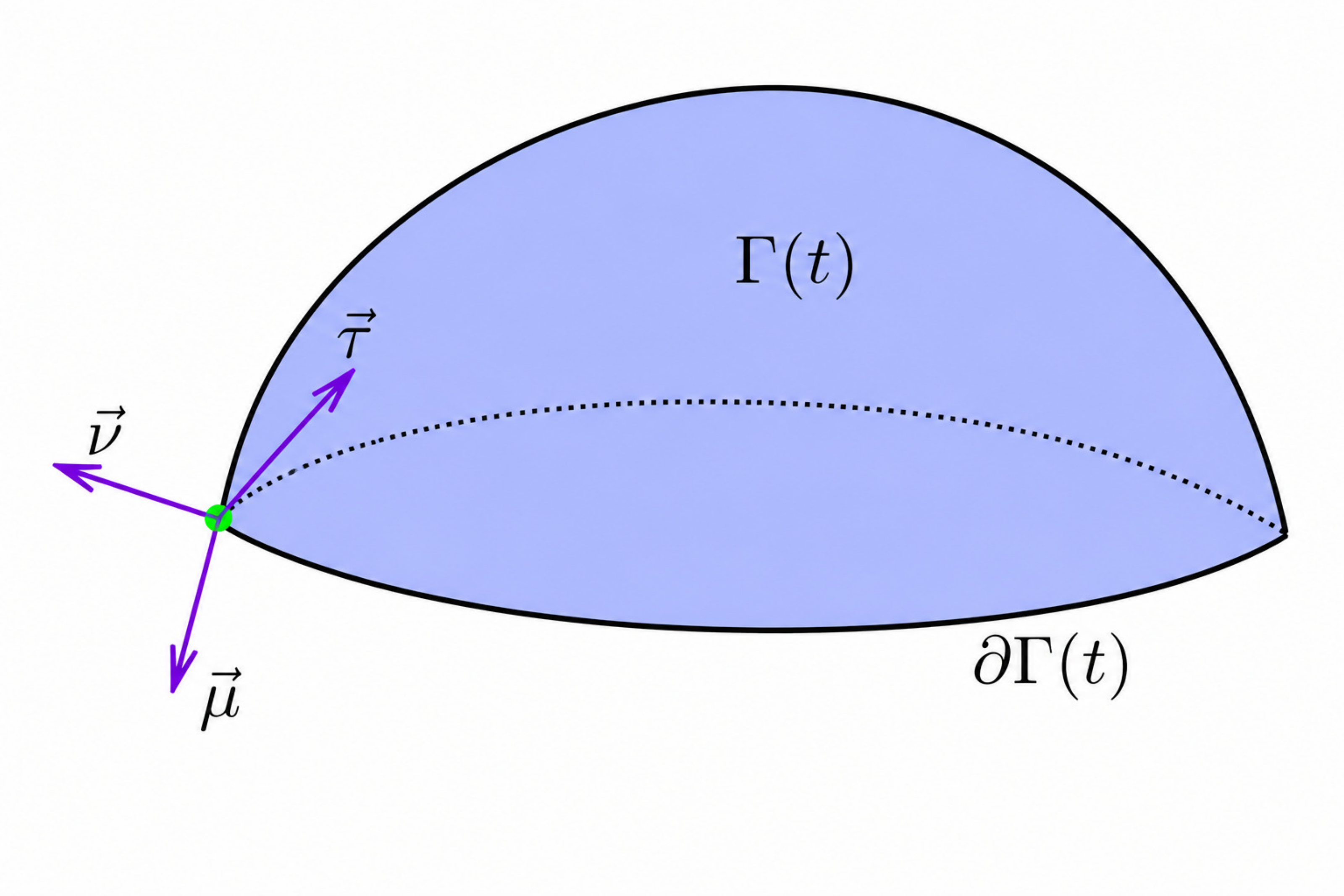}
\caption{Sketch of $\Gt$ with boundary $\partial\Gt$ and the unit vectors
$\vec\tau$, $\vec\mu$, and $\vec\nu$ along the boundary.}
\label{fig:open}
\end{figure}

Let $\{\Gamma(t)\}_{t\in[0,T]}$ be a family of smooth hypersurfaces in
$\bR^d$, $d\in\{2,3\}$, with or without boundary, parametrized by
\begin{equation}\label{eq:para}
  \vec u(\cdot,t):\Upsilon\to\bR^d,
  \qquad
  \Gamma(t)=\vec u(\Upsilon,t),
\end{equation}
where $\Upsilon$ is a fixed oriented reference manifold, possibly with
boundary.  The full velocity induced by this parametrization is
\begin{equation}\label{eq:parametric-velocity}
  \vec w(\vec u(\vec\rho,t),t)=
  \partial_t\vec u(\vec\rho,t),
  \qquad (\vec\rho,t)\in\Upsilon\times[0,T].
\end{equation}
Let $\vec\nu$ denote the unit normal to $\Gamma(t)$. The velocity can be decomposed
into its normal and tangential parts as
\[
\vec w
=
(\vec w\cdot\vec \nu)\vec \nu
+
(\Id-\vec\nu\otimes\vec\nu)\vec w
=
\mathcal V\,\vec\nu + P_\Gamma\vec w
\qquad\mbox{on}\quad\Gt,
\]
where $\mathcal V = \vec w\cdot\vec\nu$ is the normal velocity and
$P_\Gamma$ is the orthogonal projection onto the tangent space:
\begin{equation}
P_\Gamma= \Id - \vec\nu\otimes\vec\nu\qquad\mbox{on}\quad\Gt,
\end{equation}
with $\Id\in\R^{d\times d}$ being the identity matrix. 

We have the curvature identity
\begin{equation}\label{eq:mean-curvature-vector}
  \Ds\id|_{\Gamma(t)}
  =
  \vec\varkappa
  =
  \varkappa\,\vec\nu
  \qquad\hbox{on}\quad \Gamma(t),
\end{equation}
where $\id|_{\Gamma(t)}$ is the identity map on $\Gamma(t)$, $\Ds=\divs\nabs$ is
the Laplace--Beltrami operator, $\varkappa$ is the mean curvature, and
$\vec\varkappa$ is the mean-curvature vector.  Our sign convention is such
that $\varkappa=-\frac{(d-1)}{r}$ for a sphere of radius $r$ with the outer unit normal.

We define the evolving hypersurface by
\begin{equation}
\mathcal G_T=\bigcup\nolimits_{t\in[0,T]}(\Gamma(t)\times\{t\}).
\end{equation}
Given $f\in C^1(\mathcal G_T)$, its material derivative with respect to the
parametrization \eqref{eq:para} is defined by
\begin{equation}\label{eq:mattimeD}
	\partial_t^\circ f
  = \frac{\rd}{\rd t}f(\vec u(\vec\rho,t),t)
  \qquad\forall\,(\vec\rho,t)\in\Upsilon\times[0,T].
\end{equation}
The normal time derivative of $f$ is
\begin{equation}\label{eq:nortimeD}
\partial_t^\square f = \partial_t^\circ f - \vec w\cdot\nabs f\qquad\mbox{on}\quad\Gt,
\end{equation}
which measures the change of $f$ in the normal direction.  The following
lemma provides the geometric identities used later; see
\cite[Lemmas 37, 38, 39]{Barrett20}.

\begin{lem}
The following identities hold:
\begin{subequations}
\renewcommand{\theHequation}{\theparentequation.\alph{equation}}
\begin{alignat}{2}
\label{eq:dtsnu}
\partial_t^\square\vec\nu &= -\nabs\mathcal V\qquad &&\mbox{on}\quad\Gt, \\
\label{eq:vidennu}
\partial_t^\circ\vec\nu &= -[\nabs\vec w]^T\vec\nu\qquad &&\mbox{on}\quad\Gt,\\
\label{eq:dtscur}
\partial_t^\square\varkappa &= \Ds\mathcal V + \mathcal{V}\,|\nabs\vec\nu|^2\qquad &&\mbox{on}\quad\Gt,
\end{alignat}
where $\nabs\vec\nu$ is the Weingarten map and
$|A|^2={\rm tr}(A\,A^T)$ denotes the Frobenius norm.  Moreover, for a smooth
vector field $\vec\varphi:\mathcal G_T\to\bR^d$, the commutator rule reads
\begin{equation}
\partial_t^\circ(\nabs\cdot\vec\varphi)-\nabs\cdot(\partial_t^\circ\vec\varphi)
=
[\nabs\vec w - P_\Gamma\mathcal D(\vec w)P_\Gamma]:\nabs\vec\varphi
\quad\mbox{on}\quad\Gt, \label{eq:videncomu}
\end{equation}
\end{subequations}
where $\mathcal D(\vec w)$ is the symmetric tensor
\begin{equation*}    
\mathcal D (\vec w) = \nabs\vec w + (\nabs\vec w)^T.
\end{equation*}
\end{lem}

The next lemma collects the integration-by-parts formula and the transport
formula on an evolving hypersurface; see \cite[Lemma 32]{Barrett20}.
\begin{lem}
Let $\phi$ be a smooth scalar field on $\Gamma(t)$.  The integration-by-parts
formula on $\Gt$ reads
\begin{subequations}
\renewcommand{\theHequation}{\theparentequation.\alph{equation}}
\begin{equation}\label{eq:interpart}
\int_{\Gt}\nabs\phi\,\dH^{d-1}
=
-\int_{\Gt}\phi\,\varkappa\,\vec\nu\,\dH^{d-1}
+
\int_{\partial\Gt}\phi\,\vec\mu\,\dH^{d-2},
\end{equation}
where $\vec\mu$ is the outer unit co-normal to $\partial\Gt$ (see
Fig.~\ref{fig:open}), and $\dH^\ell$ denotes integration with respect to the
$\ell$-dimensional Hausdorff measure in $\bR^d$.  Moreover, given
$f\in C^1(\mathcal G_T)$, the Reynolds transport theorem reads
\begin{align}
\frac{\dd }{\dd t}\int_{\Gt}f\,\dH^{d-1}&=
\int_{\Gt}\left(\partial_t^\circ f + f\,\divs\vec w\right)\,\dH^{d-1}=
\int_{\Gt}\left(\partial_t^\circ f + f\,\nabs\vec\id:\nabs\vec w\right)\,\dH^{d-1}\nonumber\\
&=
\int_{\Gt}(\partial_t^\square f - f\,\mathcal V\,\varkappa)\,\dH^{d-1}
+
\int_{\partial\Gt}f\,\vec w\cdot\vec\mu\,\dH^{d-2},\label{eq:transport}
\end{align}
\end{subequations}
where the second equality uses the identity
\begin{equation}\label{eq:dividentity}
\divs\vec\varphi = \nabs\vec\id: \nabs\vec\varphi\qquad\mbox{on}\quad\Gt
\end{equation}
and the last equality follows from integration by parts in \eqref{eq:interpart} and \eqref{eq:nortimeD}.
\end{lem}

\subsection{The Willmore flow}
\label{subsec:formal-flow}

We consider the bending energy
\begin{equation}\label{eq:geometric-bending-energy}
  \E_{\bkap}
  =
  \frac12\int_{\Gt}|\varkappa-\bkap|^2\,\dH^{d-1}=\frac12\int_{\Gt} k^2\,\dH^{d-1}, 
\end{equation}
where $k=\varkappa-\bkap$ is the shifted scalar curvature, and
{the spontaneous curvature $\bkap\in\mathbb{R}$ is assumed
to be constant throughout this work, as in the classical
model~\cite{BGN08willmore,pwfade}. Models with spatially varying
spontaneous curvature have also been
studied~\cite{Liu12mesoscale,Rangamani21local,contri26}.} 

Applying the transport theorem \eqref{eq:transport} and integrating the
$\Ds\mathcal V$ term by parts gives
\begin{align}
  \frac{\dd}{\dd t}\E_{\bkap}
  ={}&\int_{\Gt} (k\,\partial_t^\square k-\frac12 k^2\,\mathcal V\,\varkappa)\,\dH^{d-1} + \frac12\int_{\partial\Gt}k^2\,\vec w\cdot\vec\mu\,\dH^{d-2} \nonumber\\
  ={}&\int_{\Gt}\left[
  \Delta_s\,k +k\,|\nabs\vec\nu|^2
  -
  \frac12 k^2\varkappa
  \right]\mathcal V\,\dH^{d-1}\nonumber\\
  &+
  \int_{\partial\Gt}
  \left[
  \frac12 k^2\,\vec w\cdot\vec\mu
  -
  (\vec\mu\cdot\nabs k)\,\mathcal V +
k\,
  (\vec\mu\cdot\nabs\mathcal V)
  \right]\dH^{d-2}.
  \label{eq:dtE}
\end{align}

Taking the $L^2$-gradient flow of \eqref{eq:geometric-bending-energy} gives the
formal Willmore flow
\begin{equation}\label{eq:willmore}
\mathcal V = -\Ds k - k\,|\nabs\vec\nu|^2+\frac12 k^2\varkappa\qquad\mbox{on}\quad\Gt.
\end{equation}
For open surfaces we consider the following fixed-boundary conditions:
\begin{subequations}\label{eqn:bd}
\renewcommand{\theHequation}{\theparentequation.\alph{equation}}
\begin{align}
\hbox{Navier:}\qquad
&\partial\Gt = \partial\Gamma(0),
\qquad
k = 0
\qquad\mbox{on}\quad\partial_1\Gt,
\label{eq:navier}
\\
\hbox{clamped:}\qquad
&\partial\Gt = \partial\Gamma(0),
\qquad
\vec\mu\cdot\nabs\mathcal V =0
\qquad\mbox{on}\quad\partial_2\Gt,
\label{eq:clamp}
\end{align}
\end{subequations}
where $\partial\Gt=\partial_1\Gt\cup\partial_2\Gt$ is a disjoint partition. The fixed boundary condition implies that
\begin{equation}\label{eq:fixbd}
 \vec w\cdot\vec\mu=0,\qquad \vec w\cdot\vec\nu=\mathcal V=0
 \quad\mbox{on}\quad\partial\Gt .
\end{equation}
On the clamped part, \eqref{eq:clamp} can be interpreted geometrically as
preserving the prescribed co-normal.  For example, when $d=3$, the fixed
boundary also gives $\partial_t^\square\vec\tau=\vec0$, and hence
 \begin{equation*}
    \partial_t^\square\vec\mu
    =
    \partial_t^\square(\vec\nu\times\vec\tau)
    =
    (\partial_t^\square\vec\nu)\times\vec\tau
    =
    -(\nabs\mathcal V)\times\vec\tau
    =
    \vec0
    \qquad\mbox{on} \quad\partial_2\Gamma(t),
  \end{equation*}
since $\nabs\mathcal{V}\in{\rm span}\{\vec\tau,\vec\mu\}$ and
$\vec\mu\cdot\nabs\mathcal V=0$; see \cite{GNZ26}. This implies that the co-normal vector does not change in time. 
With \eqref{eq:fixbd}, the first two boundary terms in \eqref{eq:dtE} vanish;
the remaining boundary term vanishes by \eqref{eq:navier} on $\partial_1\Gt$
and by \eqref{eq:clamp} on $\partial_2\Gt$. Hence, \eqref{eq:dtE} shows that
the geometric flow \eqref{eq:willmore}, together with the boundary conditions
\eqref{eqn:bd}, obeys the energy-dissipation law
\begin{equation}
	\frac{\dd}{\dd t} \E_{\bkap} = -\int_{\Gt}\mathcal{V}^2\,\dH^{d-1}\leq 0.
	\label{eq:energylaw}
\end{equation}

\section{Onsager variational formulation}\label{sec:variational}

In this section, we first recall Onsager’s variational principle and then develop a two-stage variational construction for the Willmore flow.

\subsection{A constrained Onsager principle}\label{subsec:onsager-principle}

We give a short review of Onsager's variational principle in an abstract
setting.  Let $\vec q(t)$ denote the state of an isothermal system and let
$\dot{\vec q}$ be its rate.  The free energy is denoted by $\E(\vec q)$.  For
a fixed state $\vec q$, the Rayleighian is
\begin{equation}\label{eq:abstract-Rays}
\Rays(\vec q;\dot{\vec q})
=
\Phi(\vec q;\dot{\vec q}) + D\E(\vec q)[\dot{\vec q}] .
\end{equation}
Here $\Phi$ is the dissipation potential.  In this work it is quadratic in the
rate and may depend on the current state:
\begin{equation}\label{eq:abstract-dissipation}
  \Phi(\vec q;\dot{\vec q})
  =
  \frac12\,\mathcal M_{\vec q}(\dot{\vec q},\dot{\vec q}),
\end{equation}
where $\mathcal M_{\vec q}(\cdot,\cdot)$ is a symmetric nonnegative bilinear
form.  The second term in the Rayleighian is the directional derivative of the
free energy,
\begin{equation}
  D\E(\vec q)[\vec\phi]
  =
  \frac{\dd}{\dd\varepsilon}
  \E(\vec q+\varepsilon\vec\phi)
  \bigg|_{\varepsilon=0}
  =
  \left\langle
  \frac{\delta\E(\vec q)}{\delta\vec q},\vec\phi
  \right\rangle .
\end{equation}

According to Onsager's variational principle, the physical rate minimizes the
Rayleighian over the admissible rate space $\mathcal A(\vec q)$
\cite{Onsager31a,Onsager31b}:
\begin{equation}\label{eq:abstract-minimization}
  \dot{\vec q}
  =
  \operatorname*{arg\,min}_{\vec v\in\mathcal A(\vec q)}
  \Rays(\vec q;\vec v)
  =
  \operatorname*{arg\,min}_{\vec v\in\mathcal A(\vec q)}
  \left\{
  \frac12\,\mathcal M_{\vec q}(\vec v,\vec v)
  +
  D\E(\vec q)[\vec v]
  \right\}.
\end{equation}
If the admissible rates form a linear space, the associated Euler--Lagrange
equation is
\begin{equation}\label{eq:abstract-gradient}
  \mathcal M_{\vec q}(\dot{\vec q},\vec\eta)
  =
  -D\E(\vec q)[\vec\eta]
  \qquad\forall\vec\eta\in\mathcal A(\vec q).
\end{equation}
Taking $\vec\eta=\dot{\vec q}$ in \eqref{eq:abstract-gradient} recovers the
energy-dissipation identity
\begin{equation}\label{eq:abstract-energy-dissipation}
  \frac{\dd}{\dd t}\E(\vec q(t))
  =
  D\E(\vec q)[\dot{\vec q}]
  =
  -\mathcal M_{\vec q}(\dot{\vec q},\dot{\vec q})
  =
  -2\Phi(\vec q;\dot{\vec q})
  \le 0 .
\end{equation}

\subsection{Curvature-constrained bending-energy variation}
\label{sec:bending-variation}

The first stage of the construction determines the bending-energy variation.
At a fixed
surface, we regard the full velocity $\vec w$ as an arbitrary admissible
variation.  In the closed case all boundary sets below are empty.  In the open cases
$\partial\Gamma=\partial_1\Gamma\cup\partial_2\Gamma$, where
$\partial_1\Gamma$ carries the Navier condition and $\partial_2\Gamma$ carries
the clamped co-normal condition.  In what follows, we strengthen the geometric
fixed-boundary condition \eqref{eq:fixbd} by fixing the parametrization
pointwise:
\[
  \vec w=\vec0
  \qquad\text{on}\quad\partial\Gt.
\]
This excludes any residual tangential reparametrization along the fixed
boundary. Therefore the clamped boundary condition implies that 
\begin{equation}\label{eq:clamp2}
\vec\mu(\cdot,t) = \vec\mu(\cdot,0)\qquad\mbox{on}\quad\partial_2\Gt.
\end{equation}
Accordingly, we set
\begin{equation}\label{eqn:fspace}
\begin{aligned}
  \bV_{\partial_1}(\Gt)
  &:=
  \{\chi\in H^1(\Gt):\chi=0\quad\mbox{on }\partial_1\Gt\},
  \\
  \bV_0(\Gt)
  &:=
  \{\chi\in H^1(\Gt):\chi=0\quad\mbox{on }\partial\Gt\}.
\end{aligned}
\end{equation}
The admissible velocity therefore satisfies
$\vec w\in[\bV_0(\Gt)]^d$.

On introducing the shifted curvature
$\vec k = \vec\varkappa - \bkap\vec\nu = k\vec\nu$, we rewrite the energy in
\eqref{eq:geometric-bending-energy} as
\begin{equation}\label{eq:bending-energy}
  \E_{\bkap}(\vec u)
  = \frac12\int_{\Gt}|\vec k|^2\,\dH^{d-1}.
\end{equation}
We denote by $\ipd{\cdot,\cdot}_{\Gamma}$ and
$\ipd{\cdot,\cdot}_{\partial_2\Gamma}$ the $L^2$ inner products over $\Gt$
and $\partial_2\Gt$, respectively.  Recalling
\eqref{eq:mean-curvature-vector} and the boundary conditions
\eqref{eqn:bd}, we define the shifted curvature weakly by finding
$\vec k\in[\bV_{\partial_1}(\Gamma(t))]^d$ such that
\begin{equation}\label{eq:weak-vector-curvature}
  \ipd{\vec k+\bkap\vec\nu,~\vec\eta}_{\Gamma}
  +
  \ipd{\nabs\vec\id, \nabs\vec\eta}_{\Gamma}
  =
  \ipd{\vec\mu,\vec\eta}_{\partial_2\Gamma},
  \qquad
  \forall\,\vec\eta\in[\bV_{\partial_1}(\Gamma(t))]^d.
\end{equation} 

{We define the bending-variation functional by 
\begin{equation}\label{eq:pbending-shape-derivative}
\begin{aligned}
\mathcal{B}_\Gamma(\vec k; \vec w)&={}-\ipd{\nabs\vec k,~\nabs\vec w}_{\Gamma}
-\ipd{\nabs\cdot\vec k,~\nabs\cdot\vec w}_{\Gamma}+\ipd{(\nabs\vec k)^T,~\mathcal D(\vec w)(\nabs\vec\id)^T}_{\Gamma}\\
 &\quad\;-\ipd{
\frac12|\vec k|^2+\bkap\,\vec\nu\cdot\vec k,~
\nabs\vec\id:\nabs\vec w 
}_{\Gamma}+\ipd{\bkap(\nabs\vec w)^T\vec\nu,~\vec k}_{\Gamma}.
\end{aligned}
\end{equation}}
Then the following lemma provides the directional derivative of the free energy in
\eqref{eq:bending-energy}; see \cite[Lemma~131]{Barrett20} for a simplified setting. 

\begin{lem}[curvature-constrained bending-energy variation]\label{lem:bending-shape-derivative}
Let $\mathcal{G}_T$ be a smooth evolving hypersurface with a smooth velocity
$\vec w\in[\bV_0(\Gamma(t))]^d$.  Assume that
$\vec k\in[\bV_{\partial_1}(\Gamma(t))]^d$ satisfies
\eqref{eq:weak-vector-curvature}.  Then the first variation of
\eqref{eq:bending-energy} in the direction $\vec w$ is
\begin{equation}\label{eq:bending-shape-derivative}
{D\E_{\bkap}(\vec u)[\vec w] = {} \mathcal B_\Gamma(\vec k; \vec w).}
\end{equation}
\end{lem}

\begin{proof}
For a fixed time, choose a transported test function
$\vec\eta\in[\bV_{\partial_1}(\Gamma(t))]^d$ with
$\partial_t^\circ\vec\eta=\vec0$.  Since $\vec w=\vec0$ on
$\partial\Gamma(t)$, the trace of $\vec\eta$ and the boundary measure are
time-independent.  On the clamped boundary $\partial_2\Gamma(t)$, the
prescribed co-normal is preserved by \eqref{eq:clamp2}.  Hence
\[
\frac{\dd}{\dd t}
\ipd{\vec\mu,\vec\eta}_{\partial_2\Gamma}=0,
\]
and differentiating \eqref{eq:weak-vector-curvature} gives
\begin{equation}\label{eq:cdc}
\frac{\dd}{\dd t}\ipd{\vec k,\vec\eta}_{\Gamma}
=
-
\frac{\dd}{\dd t}\ipd{\nabs\vec\id,\nabs\vec\eta}_{\Gamma}
-
\frac{\dd}{\dd t}\ipd{\bkap\vec\nu,\vec\eta}_{\Gamma}.
\end{equation}

Using the transport formula \eqref{eq:transport} and
$\partial_t^\circ\vec\eta=\vec0$, the left-hand side of \eqref{eq:cdc} can be
rewritten as
\begin{equation}\label{eq:clc}
\frac{\dd}{\dd t}\ipd{\vec k,\vec\eta}_{\Gamma}
=
\ipd{\partial_t^\circ\vec k,\vec\eta}_{\Gamma}
+
\ipd{\vec k\cdot\vec\eta,\nabs\cdot\vec w}_{\Gamma}.
\end{equation}
{We recall the identity $\nabs\vec\id:\nabs\vec\eta=\nabs\cdot\vec\eta$ in \eqref{eq:dividentity}.} For the first term on the right-hand side of \eqref{eq:cdc}, applying the
transport formula and the commutator identity \eqref{eq:videncomu}, together
with $\partial_t^\circ\vec\eta=\vec0$, gives
\begin{equation}\label{eq:continuous-right-divergence}
\begin{aligned}
\frac{\dd}{\dd t}\ipd{\nabs\vec\id,\nabs\vec\eta}_{\Gamma}
={}&
\frac{\dd}{\dd t}\ipd{\nabs\cdot\vec\eta,1}_{\Gamma}=
\ipd{\partial_t^\circ(\nabs\cdot\vec\eta),1}_{\Gamma}
+
\ipd{\nabs\cdot\vec\eta,\nabs\cdot\vec w}_{\Gamma}
\\
={}&
\ipd{
[\nabs\vec w-P_\Gamma\mathcal D(\vec w)P_\Gamma]
:\nabs\vec\eta,1
}_{\Gamma}
+
\ipd{\nabs\cdot\vec\eta,\nabs\cdot\vec w}_{\Gamma}
\\
={}&\ipd{\nabs\vec\eta,\nabs\vec w}_{\Gamma}
+
\ipd{\nabs\cdot\vec\eta,\nabs\cdot\vec w}_{\Gamma}
-
\ipd{(\nabs\vec\eta)^T,\mathcal D(\vec w)(\nabs\vec\id)^T}_{\Gamma},
\end{aligned}
\end{equation}
where the last equality uses the symmetry of $P_\Gamma$ and the fact that
$\nabs\vec\eta P_\Gamma=\nabs\vec\eta$.

For the last term in \eqref{eq:cdc}, 
\eqref{eq:vidennu} {implies}
\[
\partial_t^\circ(\bkap\,\vec\nu)
=-
\bkap(\nabs\vec w)^T\,\vec\nu.
\]
Another application of the transport formula gives
\begin{equation}\label{eq:crn}
\begin{aligned}
\frac{\dd}{\dd t}\ipd{\bkap\,\vec\nu,~\vec\eta}_{\Gamma}
={}&
\ipd{\partial_t^\circ(\bkap\,\vec\nu),~\vec\eta}_{\Gamma}
+
\ipd{\bkap\,\vec\nu\cdot\vec\eta,~\nabs\cdot\vec w}_{\Gamma}
\\
={}&-
\ipd{\bkap\,(\nabs\vec w)^T\vec\nu,~\vec\eta}_{\Gamma}
+
\ipd{\bkap\vec\nu\cdot\vec\eta,\nabs\cdot\vec w}_{\Gamma}.
\end{aligned}
\end{equation}

Combining \eqref{eq:cdc}--\eqref{eq:crn} and choosing the transported test
function whose value at the fixed time is $\vec\eta=\vec k$, we obtain
\begin{equation}\label{eq:cmck}
\begin{aligned}
\ipd{\partial_t^\circ\vec k,\vec k}_{\Gamma}
={}&
-\ipd{|\vec k|^2,\nabs\cdot\vec w}_{\Gamma}
-\ipd{\nabs\vec k,\nabs\vec w}_{\Gamma}
-\ipd{\nabs\cdot\vec k,\nabs\cdot\vec w}_{\Gamma}
\\
&+
\ipd{(\nabs\vec k)^T,\mathcal D(\vec w)(\nabs\vec\id)^T}_{\Gamma}+
\ipd{\bkap(\nabs\vec w)^T\vec\nu,\vec k}_{\Gamma}
-\ipd{\bkap\vec\nu\cdot\vec k,\nabs\cdot\vec w}_{\Gamma}.
\end{aligned}
\end{equation}

Finally, differentiating \eqref{eq:bending-energy} and applying the transport
formula again yields
\begin{equation}\label{eq:cet}
D\E_{\bkap}(\vec u)[\vec w]
=
\frac{\dd}{\dd t}\frac12\ipd{\vec k,\vec k}_{\Gamma}
=
\ipd{\partial_t^\circ\vec k,\vec k}_{\Gamma}
+
\frac12\ipd{|\vec k|^2,\nabs\cdot\vec w}_{\Gamma}.
\end{equation}
Substituting \eqref{eq:cmck} into \eqref{eq:cet}, we obtain
\[
\begin{aligned}
D\E_{\bkap}(\vec u)[\vec w]
={}&
-\ipd{\nabs\vec k,\nabs\vec w}_{\Gamma}
-\ipd{\nabs\cdot\vec k,\nabs\cdot\vec w}_{\Gamma}
+\ipd{(\nabs\vec k)^T,
\mathcal D(\vec w)(\nabs\vec\id)^T}_{\Gamma}
\\
&-
\ipd{
\frac12|\vec k|^2+\bkap\vec\nu\cdot\vec k,
\nabs\vec\id:\nabs\vec w
}_{\Gamma}
+\ipd{\bkap(\nabs\vec w)^T\vec\nu,\vec k}_{\Gamma},
\end{aligned}
\]
which is precisely \eqref{eq:pbending-shape-derivative}.  Since the fixed time
was arbitrary, the identity holds for every $t\in[0,T]$.
\end{proof}

\begin{rem}
In the special case $\bkap=0$ and $\partial\Gamma=\emptyset$,
\eqref{eq:bending-shape-derivative} reduces to the variation formula
\cite[(3.13)]{Dziuk08}, which was derived from the second variation of the area
functional.  After suppressing the area-difference elasticity contribution
in \cite{pwfade}, it also agrees with
\cite[(2.17)]{pwfade}, which was obtained by a formal PDE-constrained
optimization argument.
\end{rem}

\subsection{Constrained Onsager weak formulation}
\label{sec:vf}

We now use the bending-energy variation from
Lemma~\ref{lem:bending-shape-derivative} to determine the evolution through a
constrained Onsager principle.  At each time, the surface state is represented
by the parametrization $\vec u$, and its full rate is the velocity $\vec w$.

Since the geometric gradient-flow law determines only the normal velocity, we
choose the dissipation potential
\begin{equation}\label{eq:normal-dissipation}
  \Phi(\vec w)=
  \frac12\,\mathcal M_\Gamma(\vec w,\vec w)
  =\frac12\int_{\Gt}(\vec w\cdot\vec\nu)\,(\vec w\cdot\vec\nu)\,\dH^{d-1} .
\end{equation}
For a prescribed normal velocity, the minimal-deformation-rate (MDR)
principle selects the tangential component of $\vec w$ through the constraint
\cite{Hu22evolving,Gao26,Gao26dualMDR}
\[
\lambda_{\rm MDR}\vec\nu
=
\Ds\vec w
\qquad\mbox{on}\quad\Gt,
\]
or, equivalently, $P_\Gamma(\Ds\vec w)=\vec0$.  To interpolate between
purely normal motion and the MDR limit, we impose the scaled relaxed-MDR
condition \cite{LZ26}
\begin{equation}\label{eq:rmdr-strong}
P_\Gamma\vec w + \lambda\vec\nu
  =
  \alpha\,\Ds\vec w
  \qquad\hbox{on}\quad\Gt,
\end{equation}
where $\alpha\in\R_{\geq 0}$ is a tunable parameter and
$\lambda:\Gt\to\bR$ is an auxiliary variable.  At a fixed surface, this
constraint is linear in $(\vec w,\lambda)$.  For $\alpha=0$,
\eqref{eq:rmdr-strong} gives $P_\Gamma\vec w=\vec0$, so the velocity is purely
normal.  For $\alpha>0$, dividing \eqref{eq:rmdr-strong} by $\alpha$ and
setting $\widehat\lambda=\lambda/\alpha$ gives
\[
\frac1\alpha P_\Gamma\vec w
+
\widehat\lambda\vec\nu
=
\Ds\vec w.
\]
Thus, as $\alpha\to\infty$, the tangential relaxation term vanishes and the
classical MDR constraint is formally recovered.

{At a fixed time, write $\Gamma=\Gamma(t)$ and define the
tangential test space
\[
\mathbb{V}_\tau(\Gamma)
:=
\left\{
\vec\xi\in[\bV_0(\Gamma)]^d
\;\middle|\;
\vec\xi\cdot\vec\nu=0 \text{ a.e. on } \Gamma
\right\}.
\]
Testing \eqref{eq:rmdr-strong} with fields in
$\mathbb{V}_\tau(\Gamma)$ eliminates the auxiliary scalar. Integration by
parts then motivates the admissible velocity set, for $\alpha\ge0$,
\begin{equation}
\label{eq:relaxed-mdr-admissible}
\mathcal{A}_\alpha(\Gamma)
:=\bigl\{\vec v\in[\bV_0(\Gamma)]^d:\quad 
 (P_\Gamma\vec v,\vec\xi)_\Gamma
+\alpha(\nabs\vec v,\nabs\vec\xi)_\Gamma=0
\quad\forall\,\vec\xi\in\mathbb{V}_\tau(\Gamma)\bigr\}.
\end{equation}
}

Based on \eqref{eq:normal-dissipation} and
\eqref{eq:bending-shape-derivative}, we now apply Onsager's variational
principle \eqref{eq:abstract-minimization} to the Willmore flow.  Given an
initial parametrization $\vec u_0$ of
$\Gamma(0)$, for each
$t\in(0,T]$ we seek the evolving surface
$\Gamma(t)=\vec u(\Upsilon,t)$, whose velocity is related to $\vec u$ by
\eqref{eq:parametric-velocity}, together with the velocity $\vec w$ satisfying
{
\begin{equation}\label{eq:velocity-minimization}
\vec w(t)
\in
\operatorname*{argmin}_{\vec v\in\mathcal{A}_\alpha(\Gamma(t))}
\left\{
\frac12
\ipd{\vec v\cdot\vec\nu,\vec v\cdot\vec\nu}_{\Gamma(t)}
+
\B_{\Gamma(t)}(\vec k(t);\vec v)
\right\},
\end{equation}
}
where $\vec k(t)\in[\bV_{\partial_1}(\Gamma(t))]^d$ is characterized by
\eqref{eq:weak-vector-curvature}.

{For sufficiently smooth velocities, membership in
$\mathcal{A}_\alpha(\Gamma(t))$ is equivalent to the existence of an
auxiliary scalar $\lambda\in L^2(\Gamma(t))$ such that
\begin{equation}\label{eq:ctangv}
0=
\ipd{P_\Gamma\vec v,\vec\xi}_{\Gamma(t)}
+
\ipd{\lambda\,\vec\nu,\vec\xi}_{\Gamma(t)}
+
\alpha
\ipd{\nabs\vec v,\nabs\vec\xi}_{\Gamma(t)}
\qquad\forall\,\vec\xi\in[\bV_0(\Gamma(t))]^d,
\end{equation}
which is the weak form of \eqref{eq:rmdr-strong}.  For the mixed formulation, we introduce a Lagrange multiplier \(\vec p\in[\mathbb V_0(\Gamma(t))]^d\) to enforce \(\eqref{eq:ctangv}\).}  The corresponding
Lagrangian is
\begin{equation}\label{eq:velocity-lagrangian}
\begin{aligned}
\Lag_\Gamma^1(\vec w,\lambda,\vec p)
:={}&
\frac12\ipd{\vec w\cdot\vec\nu,\vec w\cdot\vec\nu}_{\Gamma}
+
\B_\Gamma(\vec k;\vec w)
\\
&+
\ipd{P_\Gamma\vec w,\vec p}_{\Gamma}
+
\ipd{\lambda\,\vec\nu,\vec p}_{\Gamma}
+
\alpha\ipd{\nabs\vec w,\nabs\vec p}_{\Gamma}.
\end{aligned}
\end{equation}
Taking variations with respect to $(\vec w,\lambda,\vec p)$ and supplementing
the result with the weak definition of $\vec k$ and the kinematic relation
\eqref{eq:parametric-velocity} yields the Euler--Lagrange system.  The complete
variational formulation is therefore to find a sufficiently smooth
parametrization $\vec u:\Upsilon\times[0,T]\to\bR^d$ with
$\vec u(\cdot,0)=\vec u_0$ and $\vec u(\Upsilon,t)=\Gamma(t)$, and
\[
(\vec w,\vec k,\lambda,\vec p)\in
[\bV_0(\Gamma(t))]^d\times[\bV_{\partial_1}(\Gamma(t))]^d
\times L^2(\Gamma(t))\times[\bV_0(\Gamma(t))]^d
\]
such that \eqref{eq:parametric-velocity} holds and
\begin{subequations}\label{eqn:weak}
\renewcommand{\theHequation}{\theparentequation.\alph{equation}}
\begin{align}
0={}&
\ipd{\vec w\cdot\vec\nu,\vec\chi\cdot\vec\nu}_{\Gamma}
+
\B_\Gamma(\vec k;\vec\chi)
+
\ipd{P_\Gamma\vec\chi,\vec p}_{\Gamma}\nonumber\\
&\qquad\qquad+
\alpha\ipd{\nabs\vec\chi,\nabs\vec p}_{\Gamma}
\qquad\forall\,\vec\chi\in[\bV_0(\Gamma(t))]^d,
\label{eq:weak1}\\[0.4em]
0={}&
\ipd{\vec k+\bkap\,\vec\nu,\vec\eta}_{\Gamma}
+
\ipd{\nabs\vec\id,\nabs\vec\eta}_{\Gamma}
-\ipd{\vec\mu,\vec\eta}_{\partial_2\Gamma}
\qquad \forall\,\vec\eta\in[\bV_{\partial_1}(\Gamma(t))]^d,
\label{eq:weak2}\\[0.4em]
0={}&
\ipd{\rho\,\vec\nu,\vec p}_{\Gamma}
\qquad \forall\,\rho\in L^2(\Gamma(t)),
\label{eq:weak3}\\[0.4em]
0={}&
\ipd{P_\Gamma\vec w,\vec\xi}_{\Gamma}
+
\ipd{\lambda\,\vec\nu,\vec\xi}_{\Gamma}
+
\alpha\ipd{\nabs\vec w,\nabs\vec\xi}_{\Gamma}
\qquad \forall\,\vec\xi\in[\bV_0(\Gamma(t))]^d .
\label{eq:weak4}
\end{align}
\end{subequations}

We next consider the Helfrich flow, namely, the area- and volume-preserving
gradient flow of \eqref{eq:geometric-bending-energy}.  We
denote the surface area of $\Gt$ by
{$\mathcal{H}^{d-1}(\Gt)$}.  In the closed case, $\vol(\Gt)$ denotes the volume enclosed by
$\Gt$. {For an open surface with fixed boundary, we choose a fixed
capping hypersurface $\Sigma$ such that
$\Gamma(t)\cup\Sigma=\partial\Omega(t)$ for a bounded region
$\Omega(t)\subset\mathbb R^d$. We take $\vec\nu$ and $\vec\nu_\Sigma$
to be the outward unit normals of $\Omega(t)$ on the two boundary pieces,
respectively, and write $\vol(\Gt):=|\Omega(t)|$.
The transport theorem gives
\begin{subequations}\label{eq:dtva}
\renewcommand{\theHequation}{\theparentequation.\alph{equation}}
\begin{align}
\frac{\dd}{\dd t}\vol(\Gt)
&=\int_\Gamma\vec w\cdot\vec\nu\,\dH^{d-1}+\int_\Sigma\vec 0\cdot\vec\nu_\Sigma\,\dH^{d-1}=\ipd{\vec w,\vec\nu}_{\Gamma},
\label{eq:dtv}\\
\frac{\dd}{\dd t}{\mathcal{H}^{d-1}(\Gt)}
&=-\ipd{\vec\varkappa,\vec w}_{\Gamma}
+\ipd{\vec\mu,\vec w}_{\partial\Gamma}
=-\ipd{\vec k+\bkap\vec\nu,\vec w}_{\Gamma}.
\label{eq:dta}
\end{align}
\end{subequations}
}
Here the fixed capping surface makes no contribution to \eqref{eq:dtv}, while
the boundary term in \eqref{eq:dta} vanishes because
$\vec w=\vec0$ on $\partial\Gamma$.

Accordingly, the Helfrich flow is obtained by
requiring the minimizer in \eqref{eq:velocity-minimization} to satisfy the two
additional constraints
\begin{subequations}\label{eq:helfrich-constraints}
\renewcommand{\theHequation}{\theparentequation.\alph{equation}}
\begin{align}
  0&=\ipd{\vec w,\vec\nu}_{\Gamma},
  \label{eq:helfrich-volume-constraint}
  \\
  0&=\ipd{\vec k+\bkap\,\vec\nu,\vec w}_{\Gamma}.
  \label{eq:helfrich-area-constraint}
\end{align}
\end{subequations}
If only \eqref{eq:helfrich-volume-constraint} is imposed and
\eqref{eq:helfrich-area-constraint} is omitted, the resulting evolution is
the volume-preserving Willmore flow.
Introducing scalar Lagrange multipliers
$\lambda_A,\lambda_V\in\R$ for the area and volume constraints, respectively,
we augment \eqref{eq:velocity-lagrangian} as follows:
\begin{equation}\label{eq:helfrich-lagrangian}
\begin{aligned}
\Lag_\Gamma^2(\vec w,\lambda,\vec p,\lambda_A,\lambda_V)
:={}&
\Lag_\Gamma^1(\vec w,\lambda,\vec p)
+
\lambda_A\ipd{(\vec k + \bkap\,\vec\nu),~\vec w}_{\Gamma}
+
\lambda_V\ipd{\vec w,\vec\nu}_{\Gamma}.
\end{aligned}
\end{equation}
The corresponding Euler--Lagrange system is to find a sufficiently smooth
parametrization $\vec u:\Upsilon\times[0,T]\to\bR^d$ with
$\vec u(\cdot,0)=\vec u_0$ and $\vec u(\Upsilon,t)=\Gamma(t)$, and
\[
(\vec w,\vec k,\lambda,\vec p,\lambda_A,\lambda_V)
\in
[\bV_0(\Gamma(t))]^d\times[\bV_{\partial_1}(\Gamma(t))]^d
\times L^2(\Gamma(t))
\times[\bV_0(\Gamma(t))]^d\times\R\times\R
\]
such that \eqref{eq:parametric-velocity} holds and
\begin{subequations}\label{eqn:hweak}
\renewcommand{\theHequation}{\theparentequation.\alph{equation}}
\begin{align}
0={}&
\ipd{\vec w\cdot\vec\nu,\vec\chi\cdot\vec\nu}_{\Gamma}
+
\B_\Gamma(\vec k;\vec\chi)
+
\ipd{P_\Gamma\vec\chi,\vec p}_{\Gamma}
+
\alpha\ipd{\nabs\vec\chi,\nabs\vec p}_{\Gamma}
\nonumber\\
&+
\lambda_A\ipd{\vec k+\bkap\,\vec\nu,\vec\chi}_{\Gamma}
+
\lambda_V\ipd{\vec\chi,\vec\nu}_{\Gamma}
\qquad \forall\,\vec\chi\in[\bV_0(\Gamma(t))]^d,
\label{eq:hweak1}\\[0.4em]
0={}&
\ipd{\vec k+\bkap\,\vec\nu,\vec\eta}_{\Gamma}
+
\ipd{\nabs\vec\id,\nabs\vec\eta}_{\Gamma}
-
\ipd{\vec\mu,\vec\eta}_{\partial_2\Gamma}
\qquad \forall\,\vec\eta\in[\bV_{\partial_1}(\Gamma(t))]^d,
\label{eq:hweak2}\\[0.4em]
0={}&
\ipd{\rho\,\vec\nu,\vec p}_{\Gamma}
\qquad \forall\,\rho\in L^2(\Gamma(t)),
\label{eq:hweak3}\\[0.4em]
0={}&
\ipd{P_\Gamma\vec w,\vec\xi}_{\Gamma}
+
\ipd{\lambda\,\vec\nu,\vec\xi}_{\Gamma}
+
\alpha\ipd{\nabs\vec w,\nabs\vec\xi}_{\Gamma}
\qquad \forall\,\vec\xi\in[\bV_0(\Gamma(t))]^d,
\label{eq:hweak4}\\[0.4em]
0={}&
\ipd{\vec w,\vec\nu}_{\Gamma},
\label{eq:hweak5}\\[0.4em]
0={}&
\ipd{\vec k+\bkap\,\vec\nu,\vec w}_{\Gamma},
\label{eq:hweak6}
\end{align}
\end{subequations}
The last two equations are precisely the volume and area constraints in
\eqref{eq:helfrich-constraints}.  The system is completed by the kinematic relation
\eqref{eq:parametric-velocity}, and the initial conditions
$\vec u(\cdot,0)=\vec u_0$.

\begin{thm}[continuous conservation and energy dissipation]
\label{thm:continuous-energy}
Let a sufficiently smooth solution of \eqref{eqn:hweak}, and
\eqref{eq:parametric-velocity} be given.
Then
\begin{equation}\label{eq:continuous-conservation}
  \frac{\dd}{\dd t}{\mathcal{H}^{d-1}(\Gt)}=0,
  \qquad
  \frac{\dd}{\dd t}\vol(\Gamma(t))=0,
\end{equation}
and
\begin{equation}\label{eq:continuous-mixed-energy-law}
  \frac{\dd}{\dd t}\E_{\bkap}(\Gamma(t))
  +
  \ipd{\vec w\cdot\vec\nu,\vec w\cdot\vec\nu}_{\Gamma}
  =0.
\end{equation}
For the unconstrained system \eqref{eqn:weak}, the same energy identity holds
without the conservation statements in \eqref{eq:continuous-conservation}.
\end{thm}

\begin{proof}
The conservation laws follow directly from \eqref{eq:dtva},
\eqref{eq:hweak5}, and \eqref{eq:hweak6}.  To prove the energy identity, take
$\vec\chi=\vec w$ in \eqref{eq:hweak1}, $\vec\xi=\vec p$ in
\eqref{eq:hweak4}, and $\rho=\lambda$ in \eqref{eq:hweak3}.  The
relaxed-MDR multiplier terms cancel, while the area and volume multiplier
terms vanish by \eqref{eq:hweak5} and \eqref{eq:hweak6}.  Hence
\[
  \ipd{\vec w\cdot\vec\nu,\vec w\cdot\vec\nu}_{\Gamma}
  +\B_\Gamma(\vec k;\vec w)=0.
\]
Lemma~\ref{lem:bending-shape-derivative} identifies the second term with
$\frac{\dd}{\dd t}\E_{\bkap}(\Gamma(t))$, which proves
\eqref{eq:continuous-mixed-energy-law}.  The unconstrained case follows by
the same argument after omitting the area and volume multiplier terms.
\end{proof}

\section{Semidiscrete finite element approximations}\label{sec:semidiscrete}

For each $t\in[0,T]$, let $\Gamma^h(t)$ be a polyhedral approximation of
$\Gt$ with
\[
  \Gamma^h(t)=\bigcup\nolimits_{j=1}^J\overline{\sigma_j^h(t)},
  \quad\mbox{with}\quad\mathcal T^h(t)=\{\sigma_j^h(t)\}_{j=1}^J,\quad
  \mathcal Q^h(t)=\{\vec q^h_k(t)\}_{k=1}^K,
\]
where $\mathcal T^h(t)$ is a collection of mutually disjoint
$(d-1)$-simplices in $\bR^d$, and $\mathcal Q^h(t)$ is the set of globally
labeled vertices.  We assume that the connectivity of
$\mathcal T^h(t)$ stays unchanged in time, that the triangulation is
consistently oriented, and that all simplices remain nondegenerate for
$t\in[0,T]$.  We also assume that
$\partial_i\Gamma^h(t)=\partial_i\Gamma^h(0)$ approximates the fixed boundary
$\partial_i\Gamma(0)$, $i=1,2$.

We introduce the finite element spaces
\begin{subequations}
\begin{align}
  \mathbb V^h(\Gamma^h(t))
  &:=
  \{\chi\in C(\Gamma^h(t))\;:\;\chi|_{\sigma}\ \hbox{is affine for all }
  \sigma\in\mathcal T^h(t)\}.\\
  \mathbb V_0^h(\Gamma^h(t))
  &:=
  \{\chi\in \bV^h(\Gamma^h(t))\;:\;\chi = 0\quad\mbox{on}\quad\partial\Gamma^h(t)\}.\\
  \mathbb V_{\partial_1}^h(\Gamma^h(t))
  &:=
  \{\chi\in \bV^h(\Gamma^h(t))\;:\;\chi = 0\quad\mbox{on}\quad\partial_1\Gamma^h(t)\}.
\end{align}
\end{subequations}
We also introduce the finite element trace spaces on the
two boundary parts by
\begin{equation}\label{eq:boundary-fe-space}
\mathbb V^h(\partial_i\Gamma^h(t))
:=
\left\{
\xi^h|_{\partial_i\Gamma^h(t)}:
\xi^h\in\mathbb V^h(\Gamma^h(t))
\right\},
\qquad i=1,2.
\end{equation}
Let $\vec\mu_{\partial}^h
\in[\mathbb V^h(\partial_2\Gamma^h(t))]^d$
be a time-independent approximation of the prescribed co-normal $\vec\mu$ on
$\partial_2\Gt$.

We denote by $\vec\nu^h$ the elementwise constant unit normal to
$\Gamma^h(t)$.  We then set the discrete projection operator
  \[
  P_{\Gamma^h}:=\Id-\vec\nu^h\otimes\vec\nu^h
  =\nabs\vec\id
  \]
on each element.  For functions $f$ and $g$ that are piecewise continuous with possible jumps across interelement boundaries,
we introduce the $L^2$-inner product and the mass-lumped inner product over
$\Gamma^h(t)$ by
\begin{subequations}
\label{eq:surface-inner-products}
\renewcommand{\theHequation}{\theparentequation.\alph{equation}}
\begin{align}
  \ipd{f,~g}_{\Gamma^h}
  &:={}
  \sum_{j=1}^J
  \int_{\sigma_j^h(t)}f\,g\,\dH^{d-1},
  \label{eq:l2mass}
  \\
  \ipd{f,~g}_{\Gamma^h}^h
  &:={}
  \frac{1}{d}\sum_{j=1}^J\mathcal{H}^{d-1}(\sigma_j^h(t))
  \sum_{k=0}^{d-1}
  \lim_{\sigma_j^h(t)\ni\vec x\to\vec q^h_{j_k}}
  f(\vec x)\, g(\vec x),\label{eq:mass-lumping}
\end{align}
\end{subequations}
where $\{\vec q^h_{j_k}\}_{k=0}^{d-1}$ are the vertices of
$\sigma_j^h(t)$, and
$\mathcal{H}^{d-1}(\sigma_j^h(t))=\int_{\sigma_j^h(t)}1\,\dH^{d-1}$ is its
surface area.  Both inner products in \eqref{eq:surface-inner-products}
extend naturally to vector- and tensor-valued functions.  We define
$\ipd{\cdot,\cdot}_{\partial_2\Gamma^h}^h$ analogously as the mass-lumped
inner product over $\partial_2\Gamma^h(t)$.

\subsection{Discrete material derivatives and transport identities}

In what follows, we adopt the notation and conventions of
\cite[Sections~3.4 and~3.5]{Barrett20}.  Let
$\{\phi_k^h(\cdot,t)\}_{k=1}^K$ denote the finite element basis of
$\bV^h(\Gamma^h(t))$. We first introduce the parametrization of
$\Gamma^h(t)$ as
\begin{equation}\label{eq:sd-parametrization}
  \vec u^h(\vec q,t)
  :=
  \sum_{k=1}^K\vec q_k^h(t)\phi_k^h(\vec q,0),
  \qquad \vec q\in\Gamma^h(0),
\end{equation}
so that $\vec u^h(\cdot,t)\in[\bV^h(\Gamma^h(0))]^d$ and
$\Gamma^h(t)=\vec u^h(\Gamma^h(0),t)$.
We define
\begin{equation}
\mathcal G_T^h
=
\bigcup\nolimits_{t\in[0,T]}
\bigl(\Gamma^h(t)\times\{t\}\bigr),
\end{equation}
as the evolving polyhedral surface.  We assume that each
$\vec q_k^h(t)$, $k=1,\ldots,K$, is a $C^1$ function in time.  The discrete analogue
of the full velocity $\vec w$ is then defined as
\begin{equation}\label{eq:sd-induced-velocity}
\vec w^h(\vec z,t)
=
\sum_{k=1}^K
\left[\frac{\rd}{\rd t}\vec q_k^h(t)\right]\phi^h_k(\vec z,t)
\qquad\forall\,(\vec z,t)\in\mathcal G_T^h.
\end{equation}
Equivalently,
\begin{equation}\label{eq:sd-parametric-velocity}
  \vec w^h(\vec u^h(\vec q,t),t)
  =
  \partial_t\vec u^h(\vec q,t)
  \qquad\forall\,\vec q\in\Gamma^h(0).
\end{equation}

For $j=1,\ldots,J$, we set
\begin{equation}
\mathcal S_{j,T}^h:=
\bigcup\nolimits_{t\in[0,T]}
\bigl(\sigma_j^h(t)\times\{t\}\bigr),
\end{equation}
which is a $C^1$ evolving hypersurface.  Given
$f\in L^\infty(\mathcal{G}_T^h)$ with
$f|_{\mathcal S_{j,T}^h}\in C^1(\mathcal S_{j,T}^h)$, we define its discrete
material derivative elementwise by
\begin{equation}
\bigl(\partial_t^{\circ,h}f\bigr)(\vec z,t):=\partial_t^\circ f(\vec z,t)\qquad\forall(\vec z,t)\in\mathcal S_{j,T}^h.
\end{equation}
Then
\begin{equation}\label{eq:transported-basis}
  \partial_t^{\circ,h}\phi_k^h(\cdot,t)=0,
  \qquad k=1,\ldots,K.
\end{equation}
For $f(\cdot,t)=\sum_{k=1}^K f_k(t)\,\phi_k^h(\cdot,t)$,
\eqref{eq:transported-basis} gives
\begin{equation}
\partial_t^{\circ,h} f(\vec z,t)
=
\sum_{k=1}^K
\left[\frac{\rd}{\rd t} f_k(t)\right]\phi_k^h(\vec z,t)
\qquad\forall\,(\vec z,t)\in\mathcal G_T^h.
\end{equation}
In particular,
\begin{equation}\label{eq:sd-velocity-material-id}
  \partial_t^{\circ,h}\vec\id=\vec w^h
  \qquad\mbox{on }\mathcal G_T^h.
\end{equation}

We introduce the space--time finite element spaces
\begin{align*}
\mathbb V^h(\mathcal G_T^h)
&:=
\bigl\{
\chi\in C(\mathcal G_T^h):
\chi(\cdot,t)\in\mathbb V^h(\Gamma^h(t))
\quad\forall\,t\in[0,T]
\bigr\},
\\
\mathbb V_T^h(\mathcal G_T^h)
&:=
\bigl\{
\chi\in\mathbb V^h(\mathcal G_T^h):
\partial_t^{\circ,h}\chi\in C(\mathcal G_T^h)
\bigr\}.
\end{align*}

The following lemma is the discrete analogue of the transport theorem.
\begin{lem}[discrete transport identities]\label{lem:sd-transport}
Let $\eta^h,\zeta^h\in\mathbb V_T^h(\mathcal G_T^h)$. Then
\begin{equation}\label{eq:sd-transport}
\frac{\dd}{\dd t}\ipd{\eta^h,~\zeta^h}_{\Gamma^h}^{\star}
=
\ipd{\partial_t^{\circ,h}\eta^h,~\zeta^h}_{\Gamma^h}^{\star}
+
\ipd{\eta^h,~\partial_t^{\circ,h}\zeta^h}_{\Gamma^h}^{\star}
+
\ipd{\eta^h\zeta^h,~\nabs\cdot\vec w^h}_{\Gamma^h}^{\star},
\end{equation}
where $\star$ denotes either the standard or the mass-lumped inner product in
\eqref{eq:surface-inner-products}.  Moreover, \eqref{eq:sd-transport} remains
valid if $\eta^h$ and $\zeta^h$ are $C^1$ on each $\mathcal S_{j,T}^h$.
\end{lem}

\begin{proof}
For functions in $\mathbb V_T^h(\mathcal G_T^h)$, the result is
\cite[Theorem~70]{Barrett20}.  For piecewise $C^1$ functions, we apply the
transport identity on each evolving simplex $\sigma_j^h(t)$ and sum over
$j=1,\ldots,J$.
\end{proof}

For later use, we have the discrete analogues of \eqref{eq:vidennu} and \eqref{eq:videncomu}; see \cite[Lemmas~72 and~73]{Barrett20}:
\begin{subequations}\label{eqn:dviden}
\renewcommand{\theHequation}{\theparentequation.\alph{equation}}
\begin{align}
\partial_t^{\circ,h}\vec\nu^h
&=-[\nabs\vec w^h]^T\vec\nu^h,
\label{eq:dvidennu}\\
\partial_t^{\circ,h}(\nabs\cdot\vec\varphi^h)
-\nabs\cdot(\partial_t^{\circ,h}\vec\varphi^h)
&=
\bigl[
\nabs\vec w^h-P_{\Gamma^h}\mathcal D(\vec w^h)P_{\Gamma^h}
\bigr]:\nabs\vec\varphi^h,
\label{eq:dvidencomu}
\end{align}
\end{subequations}
almost everywhere on $\Gamma^h(t)$, for
$\vec\varphi^h\in[\mathbb V_T^h(\mathcal G_T^h)]^d$.

Analogously to \eqref{eq:dtva}, the following discrete identities hold:
\begin{subequations}\label{eq:sd-area-volume-derivatives}
\renewcommand{\theHequation}{\theparentequation.\alph{equation}}
\begin{align}
  \frac{\dd}{\dd t}{\mathcal{H}^{d-1}(\Gamma^h(t))}
  &=
  \ipd{1,\nabs\cdot\vec w^h}_{\Gamma^h}
  =
  \ipd{\nabs\vec\id,\nabs\vec w^h}_{\Gamma^h},
  \label{eq:sd-aread}
  \\
  \frac{\dd}{\dd t}\vol(\Gamma^h(t))
  &=
  \ipd{\vec w^h,\vec\nu^h}_{\Gamma^h}
  =
  \ipd{\vec w^h,\vec\nu^h}_{\Gamma^h}^h,
  \label{eq:sd-volumed}
\end{align}
\end{subequations}
Here \eqref{eq:sd-aread} follows from \eqref{eq:sd-transport}, whereas
\eqref{eq:sd-volumed} follows from \cite[Theorem~71]{Barrett20}.  In the open
case, $\vol(\Gamma^h(t))$ denotes the volume enclosed by the union of
$\Gamma^h(t)$ and the fixed cap.

\subsection{The semidiscrete scheme}

We now consider the semidiscrete approximations of the variational formulations
in Section~\ref{sec:vf}.  The discrete analogue of
$\B_\Gamma(\vec k;\vec\chi)$ in
\eqref{eq:bending-shape-derivative} is defined as
\begin{align}\label{eq:sd-bending-variation}
\B_h(\vec k^h;\vec\chi^h)
:={}&
-
\ipd{\nabs\vec k^h,\nabs\vec\chi^h}_{\Gamma^h}
-
\ipd{\nabs\cdot\vec k^h,\nabs\cdot\vec\chi^h}_{\Gamma^h}
+
\ipd{
(\nabs\vec k^h)^T,
\mathcal D(\vec\chi^h)(\nabs\vec\id)^T
}_{\Gamma^h}^h\nonumber
\\
&-
\ipd{
\frac12|\vec k^h|^2+\bkap\,\vec\nu^h\cdot\vec k^h,
\nabs\vec\id:\nabs\vec\chi^h
}_{\Gamma^h}^h+
\ipd{\bkap(\nabs\vec\chi^h)^T\vec\nu^h,\vec k^h}_{\Gamma^h}^h.
\end{align}
Then the semidiscrete scheme for the Helfrich system
\eqref{eqn:hweak} is as follows.  Given the initial polyhedral surface
$\Gamma^h(0)$, for all $t\in(0,T]$ we find
$\Gamma^h(t)=\vec u^h(\Gamma^h(0),t)$ with
$\vec u^h\in[\mathbb V^h(\Gamma^h(0))]^d$, and
\[
(\vec w^h,\vec k^h,\lambda^h,\vec p^h,\lambda_A^h,\lambda_V^h)
\in[\bV_0^h(\Gamma^h(t))]^d
\times[\bV_{\partial_1}^h(\Gamma^h(t))]^d
\times\bV_0^h(\Gamma^h(t))
\times[\bV_0^h(\Gamma^h(t))]^d\times\bR\times\bR
\]
such that the kinematic relation \eqref{eq:sd-parametric-velocity} holds and
\begin{subequations}\label{eq:sd-helfrich-system}
\renewcommand{\theHequation}{\theparentequation.\alph{equation}}
\begin{align}
0={}&
\ipd{\vec w^h\cdot\vec\nu^h,\vec\chi^h\cdot\vec\nu^h}_{\Gamma^h}^h
+
\B_h(\vec k^h;\vec\chi^h) +
\ipd{P_{\Gamma^h}\vec\chi^h,\vec p^h}_{\Gamma^h}^h\nonumber\\
		  &\quad
+\alpha\ipd{\nabs\vec\chi^h,\nabs\vec p^h}_{\Gamma^h}+
\lambda_A^h\ipd{\vec k^h+\bkap\vec\nu^h,\vec\chi^h}_{\Gamma^h}^h
+
\lambda_V^h\ipd{\vec\chi^h,\vec\nu^h}_{\Gamma^h}^h,
  \label{eq:sdh-u}\\[0.4em]
0={}&\ipd{\vec k^h+{\bkap}\vec\nu^h,\vec\eta^h}_{\Gamma^h}^h
  +
  \ipd{\nabs\vec\id,\nabs\vec\eta^h}_{\Gamma^h}
  -\ipd{\vec\mu_\partial^h,\vec\eta^h}_{\partial_2\Gamma^h},
  \label{eq:sdh-k}\\[0.4em]
0={}&
\ipd{\rho^h\,\vec\nu^h,\vec p^h}_{\Gamma^h}^h,
\label{eq:sdh-pnormal}\\[0.4em]
0={}&
\ipd{P_{\Gamma^h}\vec w^h,\vec\xi^h}_{\Gamma^h}^h
+
\ipd{\lambda^h\vec\nu^h,\vec\xi^h}_{\Gamma^h}^h
+
\alpha\ipd{\nabs\vec w^h,\nabs\vec\xi^h}_{\Gamma^h},
\label{eq:sdh-mdr}\\[0.4em]
0={}&
\ipd{\vec w^h,\vec\nu^h}_{\Gamma^h}^h,
\label{eq:sdh-volume}\\[0.4em]
0={}&
\ipd{\vec k^h+\bkap\vec\nu^h,\vec w^h}_{\Gamma^h}^h,
\label{eq:sdh-area}
\end{align}
\end{subequations}
for all $(\vec\chi^h,\vec\eta^h,\rho^h,\vec\xi^h)$ in
\[
[\mathbb V_0^h(\Gamma^h(t))]^d
\times[\mathbb V_{\partial_1}^h(\Gamma^h(t))]^d
\times\mathbb V_0^h(\Gamma^h(t))
\times[\mathbb V_0^h(\Gamma^h(t))]^d.
\]

 The semidiscrete volume-preserving Willmore flow is
obtained by setting $\lambda_A^h=0$ and omitting \eqref{eq:sdh-area}.
Similarly, the semidiscrete approximation of the Willmore flow corresponding
to \eqref{eqn:weak} is
obtained by setting
$\lambda_A^h=\lambda_V^h=0$ and omitting
\eqref{eq:sdh-volume}--\eqref{eq:sdh-area}.

For $(\Gamma^h(t),\vec k^h(t))$, we define the semidiscrete bending
energy by
\begin{equation}\label{eq:sd-bending-energy}
  \E_{\bkap}^h(\Gamma^h(t))
  :=
  \frac12\ipd{|\vec k^h|^2,1}_{\Gamma^h}^h .
\end{equation}
The following lemma gives the variation of the discrete bending energy.
\begin{lem}[semidiscrete bending variation]\label{lem:sd-bending-variation}
Let
$(\vec w^h,\vec k^h,\lambda^h,\vec p^h,
\lambda_A^h,\lambda_V^h)$ satisfy
\eqref{eq:sd-helfrich-system}, with $\vec k^h\in[\mathbb V_T^h(\mathcal G_T^h)]^d$.
Then 
\begin{equation}\label{eq:sd-energy-variation}
  \frac{\dd}{\dd t}\E_{\bkap}^h(\Gamma^h(t))
  =
  \B_h(\vec k^h;\vec w^h).
\end{equation}
\end{lem}

\begin{proof}
The argument follows that of
Lemma~\ref{lem:bending-shape-derivative} and is analogous to  the proof of \cite[Theorem 3.1]{pwfade}.  Fix a time $t$ and choose a
transported test function
$\vec\eta^h\in[\mathbb V_{\partial_1}^h(\Gamma^h(t))]^d$ with
$\partial_t^{\circ,h}\vec\eta^h=\vec0$.  The boundary term is time-independent;
hence
\[
  \frac{\dd}{\dd t}
  \ipd{\vec\mu_\partial^h,\vec\eta^h}_{\partial_2\Gamma^h}=0.
\]
Differentiating \eqref{eq:sdh-k} and applying \eqref{eq:sd-transport} with
the mass-lumped inner product gives
\begin{equation}\label{eq:sd-diff-curvature}
\ipd{\partial_t^{\circ,h}\vec k^h,\vec\eta^h}_{\Gamma^h}^h
+
\ipd{\vec k^h\cdot\vec\eta^h,\nabs\cdot\vec w^h}_{\Gamma^h}^h
=
-\frac{\dd}{\dd t}
\ipd{\nabs\vec\id,\nabs\vec\eta^h}_{\Gamma^h}
-\frac{\dd}{\dd t}
\ipd{\bkap\vec\nu^h,\vec\eta^h}_{\Gamma^h}^h .
\end{equation}
We recall
$\nabs\vec\id:\nabs\vec\eta^h=\nabs\cdot\vec\eta^h$ on each simplex.
Using the standard transport formula and \eqref{eq:dvidencomu} then yields
\begin{equation}\label{eq:sd-right-divergence}
\begin{aligned}
\frac{\dd}{\dd t}
\ipd{\nabs\vec\id,\nabs\vec\eta^h}_{\Gamma^h}
={}&
\ipd{\nabs\vec w^h,\nabs\vec\eta^h}_{\Gamma^h}
+
\ipd{\nabs\cdot\vec\eta^h,\nabs\cdot\vec w^h}_{\Gamma^h}
\\
&-
\ipd{
(\nabs\vec\eta^h)^T,
\mathcal D(\vec w^h)(\nabs\vec\id)^T
}_{\Gamma^h}^h .
\end{aligned}
\end{equation}
Similarly, 
\eqref{eq:dvidennu} {gives}
\begin{equation}\label{eq:sd-right-normal}
\begin{aligned}
\frac{\dd}{\dd t}
\ipd{\bkap\vec\nu^h,\vec\eta^h}_{\Gamma^h}^h
={}&-
\ipd{
\bkap(\nabs\vec w^h)^T\vec\nu^h,
\vec\eta^h
}_{\Gamma^h}^h+
\ipd{
\bkap\vec\nu^h\cdot\vec\eta^h,
\nabs\cdot\vec w^h
}_{\Gamma^h}^h .
\end{aligned}
\end{equation}
At the fixed time $t$, choose the transported test function above such that
$\vec\eta^h(t)=\vec k^h(t)$.  Combining
\eqref{eq:sd-diff-curvature}--\eqref{eq:sd-right-normal}, and applying the
mass-lumped transport formula to \eqref{eq:sd-bending-energy}, we obtain
\eqref{eq:sd-energy-variation}.
\end{proof}

The following theorem gives the energy dissipation law on semidiscrete level. 

\begin{thm}[semidiscrete conservation and energy stability]
\label{prop:semidiscrete-energy}
Let
$(\vec w^h,\vec k^h,\lambda^h,\vec p^h,
\lambda_A^h,\lambda_V^h)$ satisfy
\eqref{eq:sd-helfrich-system}  with $\vec k^h\in[\mathbb V_T^h(\mathcal G_T^h)]^d$.  Then
\begin{equation}
  \frac{\dd}{\dd t}{\mathcal{H}^{d-1}(\Gamma^h(t))}=0,
  \qquad
  \frac{\dd}{\dd t}\vol(\Gamma^h(t))=0,
\end{equation}
and the semidiscrete bending energy satisfies
\begin{equation}\label{eq:sd-energy-law}
  \frac{\dd}{\dd t}\E_{\bkap}^h(\Gamma^h(t))
  +
  \ipd{
  \vec w^h\cdot\vec\nu^h,
  \vec w^h\cdot\vec\nu^h
  }_{\Gamma^h}^h
  =
  0 .
\end{equation}   
\end{thm}

\begin{proof}
The volume conservation follows immediately from
\eqref{eq:sd-area-volume-derivatives} and
\eqref{eq:sdh-volume}.  For the surface area, we choose
$\vec\eta^h=\vec w^h$ in \eqref{eq:sdh-k} and obtain
\[
  \ipd{\nabs\vec\id,\nabs\vec w^h}_{\Gamma^h}
  =
  -
  \ipd{\vec k^h+\bkap\vec\nu^h,\vec w^h}_{\Gamma^h}^h
  +
  \ipd{\vec\mu_\partial^h,\vec w^h}_{\partial_2\Gamma^h}
  =
  0,
\]
where the boundary term vanishes because $\vec w^h=\vec0$ on
$\partial\Gamma^h$, and the last equality follows from \eqref{eq:sdh-area}.
Hence
{$\frac{\dd}{\dd t}\mathcal H^{d-1}(\Gamma^h(t))=0$.}

It remains to prove the energy identity.  Taking $\vec\chi^h=\vec w^h$ in
\eqref{eq:sdh-u}, $\vec\xi^h=\vec p^h$ in
\eqref{eq:sdh-mdr}, and $\rho^h=\lambda^h$ in
\eqref{eq:sdh-pnormal}, the relaxed-MDR multiplier terms cancel:
\[
  \ipd{P_{\Gamma^h}\vec w^h,\vec p^h}_{\Gamma^h}^h
  +
  \alpha\ipd{\nabs\vec w^h,\nabs\vec p^h}_{\Gamma^h}
  =
  -\ipd{\lambda^h\vec\nu^h,\vec p^h}_{\Gamma^h}^h
  =
  0 .
\]
The area and volume multiplier terms vanish by
\eqref{eq:sdh-area} and \eqref{eq:sdh-volume}.  Therefore
\eqref{eq:sdh-u} gives
\[
  \ipd{\vec w^h\cdot\vec\nu^h,\vec w^h\cdot\vec\nu^h}_{\Gamma^h}^h
  +
  \B_h(\vec k^h;\vec w^h)
  =
  0.
\]
Using Lemma~\ref{lem:sd-bending-variation} with
\eqref{eq:sd-energy-variation} yields \eqref{eq:sd-energy-law}. 
\end{proof}

Under the same regularity assumptions, for the semidiscrete Willmore system
obtained from
\eqref{eq:sd-helfrich-system} by omitting
\eqref{eq:sdh-volume}--\eqref{eq:sdh-area} and setting
$\lambda_A^h=\lambda_V^h=0$, the same argument gives
\[
  \frac{\dd}{\dd t}\E_{\bkap}^h(\Gamma^h(t))
  +
  \ipd{
  \vec w^h\cdot\vec\nu^h,
  \vec w^h\cdot\vec\nu^h
  }_{\Gamma^h}^h
  =
  0 .
\]

\section{Fully discrete scheme}\label{sec:fully-discrete}

We employ the uniform partition of the time interval
\[
[0,T]=\bigcup\nolimits_{i=1}^M[t_{i-1},t_i],
\qquad \ttau=T/M,
\qquad t_m=m\ttau.
\]
Let $\Gamma^0$ be a polyhedral approximation of the initial surface
$\Gamma(0)$:
\[
  \Gamma^0=\bigcup\nolimits_{j=1}^J\overline{\sigma_j^0},
  \quad
  \mathcal T^0=\{\sigma_j^0\}_{j=1}^J,
  \quad
  \mathcal Q^0=\{\vec q_k^0\}_{k=1}^K,
\]
where $\mathcal T^0$ is a collection of mutually disjoint $(d-1)$-simplices
in $\bR^d$, and $\mathcal Q^0$ is the set of globally labeled vertices.  We
introduce the scalar reference finite element space
\[
  \mathbb V^0
  :=
  \bigl\{
  \chi^h\in C(\Gamma^0):
  \chi^h|_{\sigma}\ \hbox{is affine for all }
  \sigma\in\mathcal T^0
  \bigr\}.
\]

For $m\geq 1$, let
$\vec u^m\in[\mathbb V^0]^d$ be the approximation of
$\vec u(\cdot,t_m)$.  This defines the
polyhedral surface
$\Gamma^m=\vec u^m(\Gamma^0)$, which has the same connectivity as $\Gamma^0$:
\[
  \Gamma^m=\bigcup\nolimits_{j=1}^J\overline{\sigma_j^m},
  \quad
  \mathcal T^m=\{\sigma_j^m\}_{j=1}^J,
  \quad
  \mathcal Q^m=\{\vec q_k^m\}_{k=1}^K,
\]
where $\mathcal T^m$ and $\mathcal Q^m$ are the collections of simplices and
vertices of $\Gm$, respectively.  For the open-surface case, we impose
$\partial_i\Gamma^m=\partial_i\Gamma^0$ as an approximation of the fixed
boundary $\partial_i\Gamma(0)$, $i=1,2$.

Let $\{\vec q_{j_k}^m\}_{k=0}^{d-1}$ be the consistently oriented vertices of
$\sigma_j^m$.  The geometric unit normal is defined elementwise by
\begin{align}\label{eq:geonormal}
\vec\nu_j^m:=\vec\nu^m|_{\sigma_j^m}
=
\frac{\vec N(\sigma_j^m)}{|\vec N(\sigma_j^m)|},
\qquad 
\vec N(\sigma_j^m)
=
(\vec q_{j_1}^m-\vec q_{j_0}^m)\wedge\cdots\wedge
(\vec q_{j_{d-1}}^m-\vec q_{j_0}^m),
\end{align}
where $\wedge$ denotes the generalized wedge product producing the oriented
normal.  The discrete tangential projection on $\Gamma^m$ is
\[
  P_{\Gamma^m}:=\Id-\vec\nu^m\otimes\vec\nu^m .
\]

On $\Gamma^m$ we introduce the finite element spaces
\begin{subequations}
\renewcommand{\theHequation}{\theparentequation.\alph{equation}}
\begin{align*}
\mathbb V^m
&:=
\bigl\{
\chi^h\in C(\Gamma^m):
\chi^h|_{\sigma}\ \hbox{is affine for all }
\sigma\in\mathcal T^m
\bigr\}.\\
\mathbb V_0^m
&:=
\mathbb V^m\cap\bV_0(\Gamma^m),
\qquad
\mathbb V_1^m
:=
\mathbb V^m\cap\bV_{\partial_1}(\Gamma^m).
\end{align*}
\end{subequations}
Let $I_h^m:C(\Gamma^m)\to\mathbb V^m$ denote the standard interpolation
operator.  We denote by $\ipd{\cdot,\cdot}_{\Gamma^m}$ the $L^2$ inner
product over $\Gamma^m$, and by $\ipd{\cdot,\cdot}_{\Gamma^m}^h$ the corresponding
mass-lumped inner product, defined as in \eqref{eq:mass-lumping} with
$\Gamma^h(t)$ replaced by $\Gamma^m$.  We write
\begin{equation}
 \norm{\eta}_{L^2(\Gm)}^2:=\ipd{\eta,\eta}_{\Gm},
 \qquad
 \norm{\eta}_{L_h^2(\Gm)}^2:=\ipd{\eta,\eta}_{\Gm}^h .
\end{equation}

\subsection{Fully discrete scheme for Willmore flow}

For a candidate parametrization
$\vec u^{m+1}\in[\mathbb V^0]^d$ at the next time level, we define the fully
discrete velocity on $\Gamma^m$ by
\begin{equation}
  \vec w^{m+1}
  :=
  \frac{\vec u^{m+1}\circ(\vec u^m)^{-1} -\vec\id|_{\Gm}}{\ttau}\qquad\mbox{on}\qquad\Gm.
\end{equation}   

Let $\vec k^m$ approximate $\vec k(\cdot,t_m)$. We introduce
\begin{align}
  \B_h^m(\vec k;\vec\eta)
  :={}&
  -
  \ipd{\nabs\vec k,\nabs\vec\eta}_{\Gm}
  +
  \mathcal R_h^m(\vec\eta),
\label{eq:Bmd}
\end{align}
where the lagged lower-order terms are collected in
\begin{align}
  \mathcal R_h^m(\vec\eta)
  :={}&
  -
  \ipd{\divs\vec k^m,\divs\vec\eta}_{\Gm}
  +
  \ipd{(\nabs\vec k^m)^T,\D(\vec\eta)(\nabs\vec\id)^T}_{\Gm}^h
  \nonumber\\
  &-
  \ipd{
  {\frac12 |\vec k^m|^2+\bkap\vec\nu^m\cdot\vec k^m},
  \nabs\vec\id:\nabs\vec\eta
  }_{\Gm}^h
  \nonumber\\
  &+
  \ipd{\bkap\vec k^m,(\nabs\vec\eta)^T\vec\nu^m}^h_{\Gm}.
\end{align}
The sign convention in \eqref{eq:Bmd} is chosen so that
$\B_h^m(\vec k^{m+1};\cdot)$ is the linearly implicit approximation
of the semidiscrete bending variation
$\B_h(\vec k^h;\cdot)$ in \eqref{eq:sd-bending-variation}.

We now present the fully discrete parametric approximation.  {Given the initial polyhedral surface $\Gamma^0$, we initialize the shifted curvature as follows.}
If an analytic expression for the
initial shifted curvature is available, we set
$\vec k^0=I_h^0\vec k_0$.  Otherwise, we compute
$\vec k^0\in[\mathbb V_1^0]^d$ from
\begin{equation}\label{eq:fd-initial-curvature}
0=
\ipd{\vec k^0+{\bkap}\vec\nu^0,\vec\eta^h}_{\Gamma^0}^h
+
\ipd{\nabs\vec\id,\nabs\vec\eta^h}_{\Gamma^0}
-\ipd{\vec\mu_\partial^h,\vec\eta^h}_{\partial_2\Gamma^0}
\qquad\forall\,\vec\eta^h\in[\mathbb V_1^0]^d.
\end{equation}

For $m=0,\ldots,M-1$, given $(\Gamma^m,\vec k^m)$, we find
\[
  (\vec w^{m+1},\vec k^{m+1},\lambda^{m+1},\vec p^{m+1})
  \in
  [\mathbb V_0^m]^d\times[\mathbb V_1^m]^d
  \times\mathbb V_0^m\times[\mathbb V_0^m]^d
\]
with
$\vec u^{m+1}=(\vec\id|_{\Gm}+\ttau\,\vec w^{m+1})\circ\vec u^m$ such that
\begin{subequations}\label{eqn:fd-willmore}
\renewcommand{\theHequation}{\theparentequation.\alph{equation}}
\begin{align}
0={}&
  \ipd{
  \vec w^{m+1}\cdot\vec\nu^m,
  \vec\chi^h\cdot\vec\nu^m
  }_{\Gm}^h
  +
  \B_h^m(\vec k^{m+1};\vec\chi^h)
  \nonumber\\
  &\quad+
  \ipd{P_{\Gamma^m}\vec\chi^h,\vec p^{m+1}}_{\Gm}^h
  +
  \alpha\ipd{\nabs\vec\chi^h,\nabs\vec p^{m+1}}_{\Gm},
\label{eq:fdw-u}\\[0.4em]
0={}&
  \ipd{\vec k^{m+1}+{\bkap}\vec\nu^m,\vec\eta^h}_{\Gm}^h
  +
  \ipd{
  \nabs[\vec\id|_{\Gm}+\ttau\,\vec w^{m+1}],
  \nabs\vec\eta^h
  }_{\Gm}  -\ipd{\vec\mu_\partial^h,\vec\eta^h}_{\partial_2\Gamma^m},
\label{eq:fdw-curvature}\\[0.4em]
0={}&
  \ipd{\rho^h\vec\nu^m,\vec p^{m+1}}_{\Gm}^h,
\label{eq:fdw-pnormal}\\[0.4em]
0={}&
  \ipd{P_{\Gamma^m}\vec w^{m+1},\vec\xi^h}_{\Gm}^h
  +
  \ipd{\lambda^{m+1}\vec\nu^m,\vec\xi^h}_{\Gm}^h
  +
  \alpha\ipd{\nabs\vec w^{m+1},\nabs\vec\xi^h}_{\Gm},
\label{eq:fdw-mdr}
\end{align}
\end{subequations}
for all $(\vec\chi^h,\vec\eta^h,\rho^h,\vec\xi^h)
\in[\mathbb V_0^m]^d\times[\mathbb V_1^m]^d
\times\mathbb V_0^m\times[\mathbb V_0^m]^d$.  We then set
$\Gamma^{m+1}=\vec u^{m+1}(\Gamma^0)$.  Let
\[
F^m:=\vec\id|_{\Gamma^m}+\ttau\vec w^{m+1}:
\Gamma^m\to\Gamma^{m+1}.
\]

 $\vec k^{m+1}$ is therefore defined on $\Gamma^m$. For use at the next time level, {it is pushed-forward} under $F^m$ according to
\[
\vec k_{\Gamma^{m+1}}^{m+1}\circ F^m
=
\vec k_{\Gamma^m}^{m+1},
\]
where the subscripts only indicate the supporting surface.  With a slight abuse of notation, we still use the same symbol
$\vec k^{m+1}$ for its representations on $\Gamma^{m+1}$.

For the solvability analysis, we follow~\cite{BGN08parametric} and introduce
the vertex normal vector $\vec\nu_{h,p}^{m}\in[\mathbb V^m]^d$.  For each
vertex $\vec q_k^m\in\mathcal Q^m$, let
$\mathcal I_k^m:=\{j:\vec q_k^m\in\overline{\sigma_j^m}\}$ denote the
index set of the simplices sharing $\vec q_k^m$.
We assume that all simplices in $\mathcal T^m$ are nondegenerate and define
\begin{equation}\label{eq:vertex-normal-explicit}
\vec\nu_{h,p}^m(\vec q_k^m)
:=
\frac{\displaystyle\sum_{j\in\mathcal I_k^m}\mathcal{H}^{d-1}(\sigma_j^m)\,\vec\nu_j^m}
{\displaystyle\sum_{j\in\mathcal I_k^m}\mathcal{H}^{d-1}(\sigma_j^m)}
\qquad k=1,\ldots,K.
\end{equation}
Equivalently,
$\vec\nu_{h,p}^{m}$ can be regarded as the mass-lumped $L^2$--projection of
$\vec\nu^m$ onto $[\mathbb V^m]^d$, i.e.,
\begin{equation} \label{eq:nuhomegah}
\ipd{\vec\nu_{h,p}^m, \vec\chi^h}_{\Gm}^h
= \ipd{\vec\nu^m,~\vec\chi^h}_{\Gm}^h
= \ipd{\vec\nu^m,~\vec\chi^h}_{\Gm}
\quad\forall\vec\chi^h\in[\mathbb V^m]^d,
\end{equation}
where the last equality follows from the exactness of the mass-lumped rule for
elementwise affine integrands.  It follows that
\begin{equation}\label{eq:vpiden}
\ipd{\chi\,\vec\nu_{h,p}^m,~\vec\eta^h}_{\Gm}^h
= \ipd{\chi\,\vec\nu^m,~\vec\eta^h}_{\Gm}^h
\qquad\forall\chi\in \mathbb V^m,\quad\vec\eta^h\in[\mathbb V^m]^d.
\end{equation}

Let
\begin{equation}
\mathcal Q_{\rm f}^m
:=
\{\vec q\in\mathcal Q^m:\vec q\notin\partial\Gamma^m\},
\end{equation}
denote the set of free vertices of the mesh $\mathcal T^m$.  We use the following
mesh assumptions in the solvability results:
\begin{enumerate}[label=$(\mathbf{A\arabic*})$, ref=$\mathbf{A\arabic*}$]
\item \label{assumpI} $\mathcal H^{d-1}(\sigma) > 0$ for all $\sigma \in \mathcal T^m$;
\item \label{assumpII} $\dim\operatorname{span}\left(
\bigl\{\vec\nu_{h,p}^{m}(\vec q)\bigr\}_{\vec q\in \mathcal Q_{\rm f}^m}\right)=d$;
\item \label{assumpIII} $\vec\nu_{h,p}^{m}(\vec q)\neq \vec 0$ for all
$\vec q\in\mathcal Q_{\rm f}^m$.
\end{enumerate}

\begin{lem}\label{lem:a0}
Suppose that assumptions~\ref{assumpI} and \ref{assumpIII} hold.  If
$\vec a^h\in[\mathbb V_0^m]^d$ satisfies
\begin{equation}\label{eq:projcon}
\ipd{P_{\Gamma^m}\vec a^h,\vec a^h}_{\Gm}^h=0,
\end{equation}
and
\begin{equation}\label{eq:norcon}
\ipd{\rho^h\vec\nu^m,\vec a^h}_{\Gm}^h=0
\qquad\forall\,\rho^h\in\mathbb V_0^m,
\end{equation}
then $\vec a^h=\vec0\in[\mathbb V_0^m]^d$.
\end{lem}

\begin{proof}
Let $P_j^m:=\Id-\vec\nu_j^m\otimes\vec\nu_j^m$ on $\sigma_j^m$.  Since
$P_j^m$ is symmetric and idempotent, expanding \eqref{eq:projcon} gives
\[
0
=
\frac{1}{d}\sum_{j=1}^J\mathcal{H}^{d-1}(\sigma_j^m)
\sum_{\ell=0}^{d-1}
\left|P_j^m\vec a^h(\vec q_{j_\ell}^m)\right|^2.
\]
By assumption~\ref{assumpI}, all weights are positive, and hence
\begin{equation}\label{eq:pjakm}
0=P_j^m\vec a^h(\vec q_k^m)=P_j^m\vec a_k^m
\qquad\forall\,j\in\mathcal I_k^m,\quad k=1,\ldots,K,
\end{equation}
where we set $\vec a_k^m:=\vec a^h(\vec q_k^m)$. 
Suppose that $\vec a_k^m\neq\vec0$ at a fixed free vertex
$\vec q_k^m\in\mathcal Q_{\rm f}^m$.  Equation~\eqref{eq:pjakm} then implies
\[
\vec a_k^m=(\vec a_k^m\cdot\vec\nu_j^m)\vec\nu_j^m
\qquad\forall\,j\in\mathcal I_k^m.
\]
Since $|\vec\nu_j^m|=1$, this implies
\[
\vec\nu_j^m
=
\frac{\vec\nu_j^m\cdot\vec a_k^m}{|\vec a_k^m|^2}\,\vec a_k^m
\qquad\forall\,j\in\mathcal I_k^m.
\]
Using the definition \eqref{eq:vertex-normal-explicit}, we thus obtain
\begin{equation}\label{eq:vhpnew}
\vec\nu_{h,p}^m(\vec q_k^m)
=
\frac{
\displaystyle\sum_{j\in\mathcal I_k^m}
\mathcal H^{d-1}(\sigma_j^m)\,(\vec\nu_j^m\cdot\vec a_k^m)}
{|\vec a_k^m|^2\displaystyle\sum_{j\in\mathcal I_k^m}\mathcal H^{d-1}(\sigma_j^m)}
\,\vec a_k^m.
\end{equation}

Choosing $\rho_k^h\in\mathbb V_0^m$ with
$\rho_k^h(\vec q_i^m)=\delta_{ki}$ in \eqref{eq:norcon} gives
\[
0
=
\ipd{\rho_k^h\vec\nu^m,\vec a^h}_{\Gm}^h
=
\frac1d\sum_{j\in\mathcal I_k^m}
\mathcal{H}^{d-1}(\sigma_j^m)\,\vec\nu_j^m\cdot\vec a_k^m,
\]
which, on recalling \eqref{eq:vhpnew}, implies that
$\vec\nu_{h,p}^m(\vec q_k^m)=\vec0$.  This contradicts
assumption~\ref{assumpIII}.  Hence $\vec a_k^m=\vec0$.  Since the vertex
$\vec q_k^m$ was arbitrary and the boundary values already vanish,
$\vec a^h=\vec0$ in $[\mathbb V_0^m]^d$.
\end{proof}

The following theorem establishes unique solvability of the introduced scheme.
\begin{thm}[existence and uniqueness]\label{thm:fd-willmore-solvability}
Assume that $\Gamma^m$ is connected, assumptions~\ref{assumpI}--\ref{assumpIII}
hold, and $\alpha\ge0$.
Then the fully discrete scheme \eqref{eqn:fd-willmore} has a unique solution.
\end{thm}

\begin{proof}
The system is linear, square, and finite dimensional, so it is enough to show
that its homogeneous system has only the zero solution.  Dropping all known
terms, we seek
\[
(\vec w,\vec k,\lambda,\vec p)
\in
[\mathbb V_0^m]^d\times[\mathbb V_1^m]^d
\times\mathbb V_0^m\times[\mathbb V_0^m]^d
\]
such that
\begin{subequations}\label{eqn:fdw-homogeneous}
\renewcommand{\theHequation}{\theparentequation.\alph{equation}}
\begin{align}
0={}&
\ipd{\vec w\cdot\vec\nu^m,\vec\chi^h\cdot\vec\nu^m}_{\Gm}^h
-
\ipd{\nabs\vec k,\nabs\vec\chi^h}_{\Gm}
+
\ipd{P_{\Gamma^m}\vec\chi^h,\vec p}_{\Gm}^h
+
\alpha\ipd{\nabs\vec\chi^h,\nabs\vec p}_{\Gm},
\label{eq:fdw-hom-u}\\
0={}&
\ipd{\vec k,\vec\eta^h}_{\Gm}^h
+
\ttau\ipd{\nabs\vec w,\nabs\vec\eta^h}_{\Gm},
\label{eq:fdw-hom-curvature}\\
0={}&
\ipd{\rho^h\vec\nu^m,\vec p}_{\Gm}^h,
\label{eq:fdw-hom-pnormal}\\
0={}&
\ipd{P_{\Gamma^m}\vec w,\vec\xi^h}_{\Gm}^h
+
\ipd{\lambda\vec\nu^m,\vec\xi^h}_{\Gm}^h
+
\alpha\ipd{\nabs\vec w,\nabs\vec\xi^h}_{\Gm},
\label{eq:fdw-hom-mdr}
\end{align}
\end{subequations}
for all $(\vec\chi^h,\vec\eta^h,\rho^h,\vec\xi^h)$ in the corresponding
test spaces.

Taking $\vec\xi^h=\vec p$ in \eqref{eq:fdw-hom-mdr} and
$\rho^h=\lambda$ in \eqref{eq:fdw-hom-pnormal} gives
\begin{equation}\label{eq:fdw-hom-cross}
\ipd{P_{\Gamma^m}\vec w,\vec p}_{\Gm}^h
+
\alpha\ipd{\nabs\vec w,\nabs\vec p}_{\Gm}
=0.
\end{equation}
Next, take $\vec\chi^h=\vec w$ in \eqref{eq:fdw-hom-u} and
$\vec\eta^h=\vec k$ in \eqref{eq:fdw-hom-curvature}.  Using
\eqref{eq:fdw-hom-cross}, we obtain
\[
\norm{\vec w\cdot\vec\nu^m}_{L_h^2(\Gm)}^2
+
\frac1{\ttau}\norm{\vec k}_{L_h^2(\Gm)}^2
=0.
\]
Hence $\vec k=\vec0$ and $\vec w\cdot\vec\nu^m=0$ in the mass-lumped norm.
Taking $\vec\eta^h=\vec w$ in \eqref{eq:fdw-hom-curvature} then gives
$\norm{\nabs\vec w}_{L^2(\Gm)}=0$, so $\vec w$ is constant on the connected
surface.  If $\partial\Gamma^m\neq\emptyset$, the boundary condition
$\vec w\in[\mathbb V_0^m]^d$ immediately yields $\vec w=\vec0$.  If
$\partial\Gamma^m=\emptyset$, the identity
$\vec w\cdot\vec\nu^m=0$ in the mass-lumped norm implies
\[
\vec w\cdot\vec\nu_{h,p}^m(\vec q)=0
\qquad\forall\,\vec q\in\mathcal Q_{\rm f}^m.
\]
Assumption~\ref{assumpII} then also yields $\vec w=\vec0$.

With $\vec w=\vec0$, equation \eqref{eq:fdw-hom-mdr} reduces to
\[
\ipd{\lambda\vec\nu^m,\vec\xi^h}_{\Gm}^h=0
\qquad\forall\,\vec\xi^h\in[\mathbb V_0^m]^d.
\]
Choose $\vec\xi^h$ such that
\begin{equation}
\vec\xi^h(\vec q)=\left\{\begin{array}{ll}
\lambda(\vec q)\vec\nu_{h,p}^m(\vec q)\qquad &\mbox{if}\quad\vec q\in\mathcal Q_{\rm f}^m;\\[0.4em]
\vec 0\qquad &\mbox{otherwise}
\end{array}\right.
\end{equation} 
Using
\eqref{eq:vpiden}, the resulting identity is
\[
\sum_{\vec q_k^m\in\mathcal Q_{\rm f}^m}
m_k^m\lambda(\vec q_k^m)^2
|\vec\nu_{h,p}^m(\vec q_k^m)|^2=0,
\qquad
m_k^m:=\frac1d\sum_{j\in\mathcal I_k^m}\mathcal H^{d-1}(\sigma_j^m).
\]
Assumptions~\ref{assumpI} and \ref{assumpIII} therefore imply
$\lambda=0$.

Finally, taking $\vec\chi^h=\vec p$ in \eqref{eq:fdw-hom-u} yields
\[
\ipd{P_{\Gamma^m}\vec p,\vec p}_{\Gm}^h
+
\alpha\norm{\nabs\vec p}_{L^2(\Gm)}^2
=0.
\]
Since $P_{\Gamma^m}$ is symmetric and idempotent and $\alpha\ge0$, both
terms are nonnegative; in particular,
$\ipd{P_{\Gamma^m}\vec p,\vec p}_{\Gm}^h=0$.
Together with \eqref{eq:fdw-hom-pnormal}, Lemma~\ref{lem:a0} gives
$\vec p=\vec0$.  Hence the homogeneous system has only the trivial solution,
which proves existence and uniqueness.
\end{proof}

\subsection{Fully discrete scheme for Helfrich flow}

We define
\begin{equation}\label{eq:fd-area-coefficient}
\theta_A^m
:=
I_h^m(\vec k^m\cdot\vec\nu_{h,p}^m+\bkap)
\in\mathbb V^m.
\end{equation}
The direction $\theta_A^m\vec\nu^m$ is a nodal normal-projection
approximation of the semidiscrete area-gradient direction
$\vec k^m+\bkap\vec\nu^m$.  These two directions are not algebraically
identical on a polyhedral surface.  The projected form is used in the fully
discrete scheme primarily to obtain a unique-solvability condition for the two global constraint multipliers. 

The fully discrete scheme for the Helfrich flow is as follows.
For $m=0,\ldots,M-1$, given $(\Gamma^m,\vec k^m)$, we find
\[
  (\vec w^{m+1},\vec k^{m+1},\lambda^{m+1},\vec p^{m+1},
  \lambda_A^{m+1},\lambda_V^{m+1})
  \in
  [\mathbb V_0^m]^d\times[\mathbb V_1^m]^d
  \times\mathbb V_0^m\times[\mathbb V_0^m]^d
  \times\bR\times\bR
\]
with
$\vec u^{m+1}=(\vec\id|_{\Gm}+\ttau\,\vec w^{m+1})\circ\vec u^m$
such that
\begin{subequations}\label{eqn:fd-helfrich}
\renewcommand{\theHequation}{\theparentequation.\alph{equation}}
\begin{align}
0={}&
  \ipd{
  \vec w^{m+1}\cdot\vec\nu^m,
  \vec\chi^h\cdot\vec\nu^m
  }_{\Gm}^h
  +
  \B_h^m(\vec k^{m+1};\vec\chi^h)
  \nonumber\\
  &\quad+
  \ipd{P_{\Gamma^m}\vec\chi^h,\vec p^{m+1}}_{\Gm}^h
  +
  \alpha\ipd{\nabs\vec\chi^h,\nabs\vec p^{m+1}}_{\Gm}
  \nonumber\\
  &\quad+
  \lambda_A^{m+1}
  \ipd{\theta_A^m\vec\nu^m,\vec\chi^h}_{\Gm}^h
  +
  \lambda_V^{m+1}
  \ipd{\vec\nu^m,\vec\chi^h}_{\Gm}^h,
\label{eq:fdh-u}\\
0={}&
  \ipd{\vec k^{m+1}+{\bkap}\vec\nu^m,\vec\eta^h}_{\Gm}^h
  +
  \ipd{
  \nabs(\vec\id|_{\Gm}+\ttau\,\vec w^{m+1}),
  \nabs\vec\eta^h
  }_{\Gm} -\ipd{\vec\mu_\partial^h,~\vec\eta^h}_{\partial_2\Gamma^m},
\label{eq:fdh-curvature}\\
0={}&
  \ipd{\rho^h\vec\nu^m,\vec p^{m+1}}_{\Gm}^h,
\label{eq:fdh-pnormal}\\
0={}&
  \ipd{P_{\Gamma^m}\vec w^{m+1},\vec\xi^h}_{\Gm}^h
  +
  \ipd{\lambda^{m+1}\vec\nu^m,\vec\xi^h}_{\Gm}^h
  +
  \alpha\ipd{\nabs\vec w^{m+1},\nabs\vec\xi^h}_{\Gm},
\label{eq:fdh-mdr}\\
0={}&
  \ipd{\vec w^{m+1},\vec\nu^m}_{\Gm}^h,
\label{eq:fdh-volume}\\
0={}&
  \ipd{\theta_A^m\vec\nu^m,\vec w^{m+1}}_{\Gm}^h {,}\label{eq:fdh-area}
\end{align}
\end{subequations}
for all $(\vec\chi^h,\vec\eta^h,\rho^h,\vec\xi^h)
\in[\mathbb V_0^m]^d\times[\mathbb V_1^m]^d
\times\mathbb V_0^m\times[\mathbb V_0^m]^d$.
We then set $\Gamma^{m+1}=\vec u^{m+1}(\Gamma^0)$, and identify
$\vec k^{m+1}$ on $\Gamma^{m+1}$ by the push-forward convention under $F^m$ described above.

\begin{thm}[existence and uniqueness]\label{thm:fd-helfrich-solvability}
Assume that $\Gamma^m$ is connected, assumptions~\ref{assumpI}--\ref{assumpIII}
hold, and $\alpha\ge0$.
Assume, in addition, that
\begin{enumerate}[label=$(\mathbf{B\arabic*})$, ref=$\mathbf{B\arabic*}$]
\item \label{assumpB} there exist
$\vec q_i^m,\vec q_j^m\in\mathcal Q_{\rm f}^m$ such that
\[
\theta_A^m(\vec q_i^m)\neq\theta_A^m(\vec q_j^m).
\]
\end{enumerate}
Then the fully discrete scheme \eqref{eqn:fd-helfrich} has a unique solution.
\end{thm}

\begin{proof}
It suffices to consider the corresponding homogeneous system: find
\[
(\vec w,\vec k,\lambda,\vec p,\lambda_A,\lambda_V)
\in[\mathbb V_0^m]^d\times[\mathbb V_1^m]^d
\times\mathbb V_0^m\times[\mathbb V_0^m]^d\times\bR\times\bR
\]
such that
\begin{subequations}
\renewcommand{\theHequation}{\theparentequation.\alph{equation}}
\begin{align}
\label{eq:homof1}
0={}&
\ipd{\vec w\cdot\vec\nu^m,\vec\chi^h\cdot\vec\nu^m}_{\Gm}^h
-
\ipd{\nabs\vec k,\nabs\vec\chi^h}_{\Gm}
+
\ipd{P_{\Gamma^m}\vec\chi^h,\vec p}_{\Gm}^h
+
\alpha\ipd{\nabs\vec\chi^h,\nabs\vec p}_{\Gm}\nonumber\\
&+\lambda_A\ipd{\theta_A^m\vec\nu^m,~\vec\chi^h}_{\Gm}^h + \lambda_V\ipd{\vec\nu^m,~\vec\chi^h}_{\Gm}^h.\\
\label{eq:homofvc}
0={}&\ipd{\vec w,~\vec\nu^m}_{\Gm}^h.\\
0={}&\ipd{\theta_A^m\vec\nu^m,~\vec w}_{\Gm}^h.
\label{eq:homofac}
\end{align}
\end{subequations}    
together with \eqref{eq:fdw-hom-curvature}, \eqref{eq:fdw-hom-pnormal}, and
\eqref{eq:fdw-hom-mdr}.  Choosing $\vec\chi^h=\vec w$ in
\eqref{eq:homof1}, $\vec\eta^h=\vec k$ in
\eqref{eq:fdw-hom-curvature}, $\rho^h=\lambda$ in
\eqref{eq:fdw-hom-pnormal}, and $\vec\xi^h=\vec p$ in
\eqref{eq:fdw-hom-mdr} leads to
 \[
\norm{\vec w\cdot\vec\nu^m}_{L_h^2(\Gm)}^2
+
\frac1{\ttau}\norm{\vec k}_{L_h^2(\Gm)}^2
=0.
\]
Here the terms involving $\lambda_A$ and $\lambda_V$ vanish by
\eqref{eq:homofac} and \eqref{eq:homofvc}, respectively.  As in the proof of
Theorem~\ref{thm:fd-willmore-solvability}, we therefore obtain
\[
\vec k=\vec0,\qquad \vec w=\vec0,\qquad \lambda=0.
\]

Next, taking
$\vec\chi^h=\vec p$ in \eqref{eq:homof1} with
$\vec w=\vec0$ and $\vec k=\vec0$ gives
\begin{equation}
0=\ipd{P_{\Gamma^m}\vec p,\vec p}_{\Gm}^h
+
\alpha\ipd{\nabs\vec p,\nabs\vec p}_{\Gm}+\lambda_A\ipd{\theta_A^m\vec\nu^m,~\vec p}_{\Gm}^h + \lambda_V\ipd{\vec\nu^m,~\vec p}_{\Gm}^h.
\end{equation}
Let $\rho_A^h,\rho_V^h\in\mathbb V_0^m$ be defined by their nodal values as
\[
\rho_A^h(\vec q)=\theta_A^m(\vec q),\qquad
\rho_V^h(\vec q)=1
\quad\text{for }\vec q\in\mathcal Q_{\rm f}^m,
\]
with both functions set to zero at the boundary vertices.  Since $\vec p$
also vanishes at the boundary vertices, choosing $\rho^h=\rho_A^h$ and
$\rho^h=\rho_V^h$ in \eqref{eq:fdw-hom-pnormal} cancels the area and volume
multiplier terms, respectively.  Hence
\[
\ipd{P_{\Gamma^m}\vec p,\vec p}_{\Gm}^h
+
\alpha\norm{\nabs\vec p}_{L^2(\Gm)}^2
=0.
\]
By the same nonnegativity argument as in the proof of
Theorem~\ref{thm:fd-willmore-solvability}, Lemma~\ref{lem:a0} then gives
$\vec p=\vec0$.

It remains to determine $\lambda_A$ and $\lambda_V$, which satisfy 
\[
\ipd{(\lambda_A\theta_A^m+\lambda_V)\vec\nu^m,\vec\chi^h}_{\Gm}^h=0
\qquad\forall\,\vec\chi^h\in[\mathbb V_0^m]^d.
\]
Set $\zeta^h=\lambda_A\theta_A^m+\lambda_V$ and choose $\vec\chi^h$ such that
\begin{equation}
\vec\chi^h(\vec q)=\left\{\begin{array}{ll}
\zeta^h(\vec q)\vec\nu_{h,p}^m(\vec q)\qquad &\mbox{if}\quad\vec q\in\mathcal Q_{\rm f}^m;\\[0.4em]
\vec 0\qquad &\mbox{otherwise}
\end{array}\right.
\end{equation}
By \eqref{eq:vpiden} and assumptions~\ref{assumpI} and \ref{assumpIII}, this
implies
\[
\lambda_A\theta_A^m(\vec q)+\lambda_V=0
\qquad\forall\,\vec q\in\mathcal Q_{\rm f}^m,
\]
which gives $\lambda_A=0$, and hence $\lambda_V=0$.  Thus the homogeneous system
has only the trivial solution.
\end{proof}
\begin{rem}
If assumption~\ref{assumpB} is violated, i.e., no such pair of free vertices
exists, then the nodal vectors
$\bigl(\theta_A^m(\vec q)\bigr)_{\vec q\in\mathcal Q_{\rm f}^m}$ and
$\bigl(1\bigr)_{\vec q\in\mathcal Q_{\rm f}^m}$ are linearly dependent.
Consequently, the discrete area and volume constraints are linearly
dependent.  We may then set either $\lambda_A^{m+1}$ or
$\lambda_V^{m+1}$ to zero.
\end{rem}

For later reference, we denote the fully discrete bending energy at time
$t_m$ by
\begin{equation}\label{eq:fd-bending-energy}
\E_h^m
:=
\frac12\ipd{|\vec k^m|^2,1}_{\Gamma^m}^h.
\end{equation}

{Unlike the semidiscrete result in
Theorem~\ref{prop:semidiscrete-energy}, an unconditional fully discrete
energy estimate is not established here. For the Helfrich
flow, exact preservation of the enclosed volume and surface area is
likewise not guaranteed at the fully discrete level.}

\section{Numerical results}\label{sec:numerics}

In this section, we first describe the solution strategies for the two linear systems arising from the fully discrete schemes \eqref{eqn:fd-willmore} for Willmore flow and \eqref{eqn:fd-helfrich} for Helfrich flow. We then present a variety of numerical experiments. {The relaxed-MDR parameter $\alpha$ has the dimension of length squared.
For the comparable geometric scales considered here, we use $\alpha=10$
unless stated otherwise. This choice is also motivated by the parameter study
in Fig.~\ref{fig:ellipsoidalpha}, where $\alpha=10$ provides substantial
mesh-quality improvement, while further increasing $\alpha$ yields
only modest additional gains.
}  All simulations are carried out without any heuristic remeshing strategy.

\subsection{Solution strategies}\label{subsec:solution-strategies}

We begin by describing the algebraic structure of the linear systems that
arise at each time step.  Let $\{\vec\phi_i\}_{i=1}^{N_0}$,
$\{\vec\psi_i\}_{i=1}^{N_1}$, and $\{\varphi_i\}_{i=1}^{n_0}$ be bases of
$[\mathbb V_0^m]^d$, $[\mathbb V_1^m]^d$, and $\mathbb V_0^m$,
respectively.  We introduce the matrices
\begin{align*}
(M_\nu)_{ij}
&:=
\ipd{
\vec\phi_j\cdot\vec\nu^m,
\vec\phi_i\cdot\vec\nu^m
}_{\Gamma^m}^h,\qquad C_{ij}
:=
\ipd{P_{\Gamma^m}\vec\phi_j,\vec\phi_i}_{\Gamma^m}^h
+\alpha\ipd{\nabs\vec\phi_j,\nabs\vec\phi_i}_{\Gamma^m},
\\
(M_k)_{ij}
&:=
\ipd{\vec\psi_j,\vec\psi_i}_{\Gamma^m}^h,\qquad 
B_{ij}
:=
\ipd{\nabs\vec\psi_j,\nabs\vec\phi_i}_{\Gamma^m},
\qquad 
D_{ij}
:=
\ipd{\varphi_i\vec\nu^m,\vec\phi_j}_{\Gamma^m}^h.
\end{align*}
The scheme \eqref{eqn:fd-willmore} can then be written in the matrix form
$K_W\mathbf x=\mathbf f$:
\begin{equation}\label{eq:willmore-block-system} 
\begin{pmatrix}
M_\nu &-B&0&C\\
-B^T&-\frac{1}{\ttau}M_k&0&0\\
0&0&0&D\\
C&0&D^T&0
\end{pmatrix}\begin{pmatrix}
\mathbf w\\ \mathbf k\\ \boldsymbol\lambda\\ \mathbf p
\end{pmatrix}=
\begin{pmatrix}
\mathbf f_w\\ \mathbf f_k\\ \mathbf0\\ \mathbf0
\end{pmatrix},
\end{equation}
where $\mathbf w$, $\mathbf k$, $\boldsymbol\lambda$, and $\mathbf p$ are the
coefficient vectors with respect to the chosen bases of the corresponding
finite element spaces.  Multiplying \eqref{eq:fdw-curvature} by
$-1/\ttau$ yields the symmetric matrix $K_W$.  By
Theorem~\ref{thm:fd-willmore-solvability}, this matrix is nonsingular under
the stated geometric assumptions.  In practice, we solve
\eqref{eq:willmore-block-system} using a sparse direct factorization.

For the Helfrich scheme \eqref{eqn:fd-helfrich}, the two additional
Lagrange multipliers lead to the block form
\begin{equation}\label{eq:helfrich-block-system}
\begin{pmatrix}
K_W&Q\\
Q^T&0
\end{pmatrix}
\begin{pmatrix}
\mathbf x\\ \boldsymbol\Lambda
\end{pmatrix}
=
\begin{pmatrix}
\mathbf f\\ \mathbf0
\end{pmatrix},
\qquad
\boldsymbol\Lambda
:=
\begin{pmatrix}\lambda_A^{m+1}\\ \lambda_V^{m+1}\end{pmatrix},
\end{equation}
where $Q$ contains the contributions from the two constraints.  Taking the
Schur complement of \eqref{eq:helfrich-block-system} yields
\begin{align*}
\bigl(Q^{T}K_W^{-1}Q\bigr)\boldsymbol\Lambda
= {}& Q^{T}K_W^{-1}\mathbf f,
\\
\mathbf x
= {}& K_W^{-1}(\mathbf f-Q\boldsymbol\Lambda),
\end{align*}
which can be solved efficiently after factorizing $K_W$.

\subsection{Willmore flow}
\label{subsec:numerical-willmore}

\Example{1}{{convergence of expanding spheres}}

{
We test convergence using expanding spheres
$\Gamma(t)=\{\vec x\in\mathbb R^d:|\vec x|=r(t)\}$ in
$\mathbb R^d$, $d\in\{2,3\}$, with $r(0)=r_0>0$. With the outward unit normal $\vec\nu=\vec x/r(t)$, the exact scalar curvature and shifted curvature vector are
\[
\varkappa(t)=-\frac{d-1}{r(t)},
\qquad
\vec k(\vec x,t)
=-\left(\frac{d-1}{r(t)}+\bkap\right)\vec\nu.
\]

For $d=2$, a closed circle with radius $r(t)$ satisfying 
\begin{equation}\label{eq:circle-radius}
r'(t)=\frac{1}{2r(t)}
\left(\frac{1}{r(t)}-\bkap\right)
\left(\frac{1}{r(t)}+\bkap\right)
\end{equation}
is an exact solution of \eqref{eq:willmore}\cite[Example~1]{GNZ26}.  For $\bkap\neq0$ and $|\bkap r_0|\neq 1$, the radius satisfies
\begin{equation}
\bkap^4t+\bkap^2\bigl(r^2(t)-r_0^2\bigr)
+\ln\left(\frac{1-\bkap^2r^2(t)}{1-\bkap^2r_0^2}\right)=0.\nonumber
\end{equation}
For $d=3$, the corresponding radius equation is
\cite[(5.1)]{pwfade}
\begin{equation}\label{eq:Rtime}
r'(t)=-\frac{\bkap}{r(t)}
\left(\frac{2}{r(t)}+\bkap\right).
\end{equation}
For $\bkap\neq0$, set $z(t)=r(t)+2/\bkap$ and
$z_0=r_0+2/\bkap$. If $z_0\neq0$, then $z(t)$ is determined implicitly by
\begin{equation}
\frac12\bigl(z^2(t)-z_0^2\bigr)
-\frac{4}{\bkap}\bigl(z(t)-z_0\bigr)
+\frac{4}{\bkap^2}\ln\frac{z(t)}{z_0}
+\bkap^2t=0.\nonumber
\end{equation}
If $z_0=0$, the sphere is stationary and $r(t)=r_0$.

In both tests, we initialize the shifted curvature vector by its
analytic nodal interpolant,
\[
\vec k^0=I_h^0\vec k(\cdot,0)
=-\left(\frac{d-1}{r_0}+\bkap\right)
\frac{\vec\id|_{\Gamma^0}}{r_0}.
\]
At a vertex $\vec q_k^m\in\mathcal Q^m$, let
$\vec n_k^m=\vec q_k^m/|\vec q_k^m|$.  The position and shifted
curvature-vector errors are defined uniformly by
\begin{equation}\label{eq:radial-node-errors}
e_{\vec u}^m(\vec q_k^m)=\vec q_k^m-r(t_m)\vec n_k^m,
\qquad
e_{\vec k}^m(\vec q_k^m)=\vec k^m(\vec q_k^m)
+\left(\frac{d-1}{r(t_m)}+\bkap\right)\vec n_k^m.
\end{equation}
We measure their nodal $L^\infty$ and mass-lumped $L^2$ norms over
$[0,T]$:
\begin{subequations}\label{eq:spherical-errors}
\renewcommand{\theHequation}{\theparentequation.\alph{equation}}
\begin{align}
\norm{e_*}_{L^\infty}
&:=\max_{1\leq m\leq M}\max_{1\leq k\leq K}
\left|e_*^m(\vec q_k^m)\right|,
\label{eq:spherical-linf}\\
\norm{e_*}_{L^2}
&:=\max_{1\leq m\leq M}
\norm{e_*^m}_{L_h^2(\Gamma^m)},
\label{eq:spherical-l2}
\end{align}
\end{subequations}
for $*=\vec u,\vec k$.  We take $r_0=1$ and $T=1$ in both cases. For $d=2$, we set $\bkap=-0.5$, and use the nonuniform
initial node distribution from \cite[Example~1]{GNZ26}, with
counterclockwise indexing,
\[
\vec q_j^0=\bigl(\cos g(j/N),\sin g(j/N)\bigr)^T,
\qquad g(\rho)=2\pi\rho+0.1\sin(2\pi\rho),
\qquad j=0,\ldots,N-1,
\]
where $N=J=K$. For $d=3$, we take $\bkap=-1$, so that the unit sphere will expand towards a sphere of radius $2$. 

Tables~\ref{tab:circle-errors} and \ref{tab:errors} report the results
for $d=2$ and $d=3$, respectively.
Under the coupled space--time refinement $\ttau\propto h^2$, all four
errors exhibit approximately second-order convergence for $d=2$.
For $d=3$, the position errors converge at second order in both norms,
and the shifted curvature-vector error is approximately second order
in the mass-lumped $L^2$ norm, whereas its nodal $L^\infty$ error
converges at a lower rate on the finer meshes. 
}

\begin{table*}[!htbp]
\centering
\caption{{Errors and experimental orders of convergence for an
expanding unit circle under Willmore flow with $\bkap=-0.5$, and $T=1$.}}
\label{tab:circle-errors}
{\resizebox{\textwidth}{!}{%
\begin{tabular}{cc@{\qquad}rr@{\quad}rr@{\qquad}rr@{\quad}rr}
\toprule
&
& \multicolumn{4}{c}{Position error}
& \multicolumn{4}{c}{Shifted curvature-vector error} \\
\cmidrule(lr){3-6}\cmidrule(lr){7-10}
$J=K=N$ & $\ttau$
& $\norm{e_{\vec u}}_{L^\infty}$ & EOC
& $\norm{e_{\vec u}}_{L^2}$ & EOC
& $\norm{e_{\vec k}}_{L^\infty}$ & EOC
& $\norm{e_{\vec k}}_{L^2}$ & EOC \\
\midrule
32 & $4.000\mathrm{E}{-2}$
& $1.831\mathrm{E}{-2}$ & -- & $5.070\mathrm{E}{-2}$ & --
& $3.198\mathrm{E}{-2}$ & -- & $8.016\mathrm{E}{-2}$ & -- \\
64 & $1.000\mathrm{E}{-2}$
& $4.414\mathrm{E}{-3}$ & 2.06 & $1.216\mathrm{E}{-2}$ & 2.07
& $8.166\mathrm{E}{-3}$ & 1.97 & $2.025\mathrm{E}{-2}$ & 1.99 \\
128 & $2.500\mathrm{E}{-3}$
& $1.094\mathrm{E}{-3}$ & 2.01 & $3.012\mathrm{E}{-3}$ & 2.02
& $2.053\mathrm{E}{-3}$ & 1.99 & $5.078\mathrm{E}{-3}$ & 2.00 \\
256 & $6.250\mathrm{E}{-4}$
& $2.729\mathrm{E}{-4}$ & 2.00 & $7.512\mathrm{E}{-4}$ & 2.00
& $5.141\mathrm{E}{-4}$ & 2.00 & $1.271\mathrm{E}{-3}$ & 2.00 \\
\bottomrule
\end{tabular}%
}}

\par\medskip
\centering
\caption{Errors and experimental orders of convergence for an expanding
sphere under Willmore flow with $\bkap=-1$, and $T=1$.}
\label{tab:errors}
\resizebox{\textwidth}{!}{%
\begin{tabular}{cc@{\qquad}rr@{\quad}rr@{\qquad}rr@{\quad}rr}
\toprule
&
& \multicolumn{4}{c}{Position error}
& \multicolumn{4}{c}{Shifted curvature-vector error} \\
\cmidrule(lr){3-6}\cmidrule(lr){7-10}
$(J,K)$ & $\ttau$
& $\norm{e_{\vec u}}_{L^\infty}$ & EOC
& $\norm{e_{\vec u}}_{L^2}$ & EOC
& $\norm{e_{\vec k}}_{L^\infty}$ & EOC
& $\norm{e_{\vec k}}_{L^2}$ & EOC \\
\midrule
$(80,42)$ & $4.00\mathrm{E}{-2}$
& $1.19\mathrm{E}{-1}$ & -- & $6.12\mathrm{E}{-1}$ & --
& $2.39\mathrm{E}{-1}$ & -- & $8.72\mathrm{E}{-1}$ & -- \\
$(320,162)$ & $1.00\mathrm{E}{-2}$
& $2.68\mathrm{E}{-2}$ & 2.15 & $1.35\mathrm{E}{-1}$ & 2.18
& $6.50\mathrm{E}{-2}$ & 1.88 & $2.14\mathrm{E}{-1}$ & 2.02 \\
$(1280,642)$ & $2.50\mathrm{E}{-3}$
& $6.61\mathrm{E}{-3}$ & 2.02 & $3.29\mathrm{E}{-2}$ & 2.04
& $2.15\mathrm{E}{-2}$ & 1.60 & $5.43\mathrm{E}{-2}$ & 1.98 \\
$(5120,2562)$ & $6.25\mathrm{E}{-4}$
& $1.65\mathrm{E}{-3}$ & 2.00 & $8.17\mathrm{E}{-3}$ & 2.01
& $8.80\mathrm{E}{-3}$ & 1.29 & $1.51\mathrm{E}{-2}$ & 1.85 \\
\bottomrule
\end{tabular}
}
\end{table*}
\FloatBarrier

\Example{2}{tangential motion}

We next investigate how the relaxed-MDR-selected tangential velocity affects
the mesh distribution.
We monitor the triangulation using the mesh-quality indicator {\cite{Bank96}}
\[
q_{\rm min}^m=\min_{\sigma\in\mathcal T^m} q_\sigma\qquad\mbox{with}\quad q_\sigma = \frac{4\sqrt3\mathcal H^2(\sigma)}
{|\ell_1|^2+|\ell_2|^2+|\ell_3|^2},
\qquad 0<q_\sigma\le1,
\]
where $\ell_1$, $\ell_2$, and $\ell_3$ are the edge vectors of the surface
triangle $\sigma$.  Thus, $q_\sigma=1$ for an equilateral triangle, and
$q_\sigma$ decreases as the triangle degenerates.

We first consider an ellipsoid of dimensions $4\times4\times1$, discretized by
$J=5120$ triangles and $K=2562$ vertices.  We set $\bkap=0$,
$\ttau=10^{-3}$, and $T=1$ and choose $\alpha=0$ and $\alpha=10$ in \eqref{eqn:fd-willmore}. We {compare} the results with that from the BGN scheme with the same initial mesh and time step.  {Here the BGN comparison uses scheme~(4.1) of \cite{pwfade} with $\theta=0$ and no area or volume constraints.} As shown in
Fig.~\ref{fig:ellipsoid-tangential}, all three computations approach a
nearly spherical surface and produce similar energy histories.  Their
mesh behavior, however, differs markedly.  Without tangential motion, the
minimum element quality decreases, whereas the relaxed-MDR and BGN
tangential motions redistribute the vertices and improve the surface mesh.

\begin{figure}[!htbp]
\centering
\includegraphics[width=0.24\textwidth]{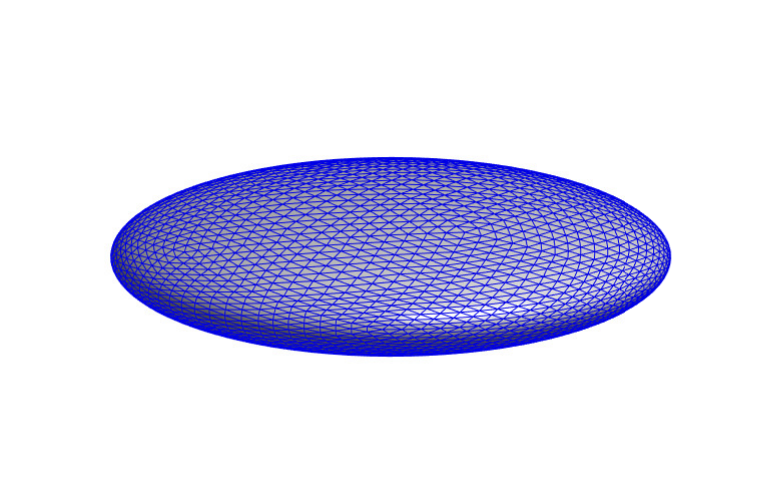}
\includegraphics[width=0.24\textwidth]{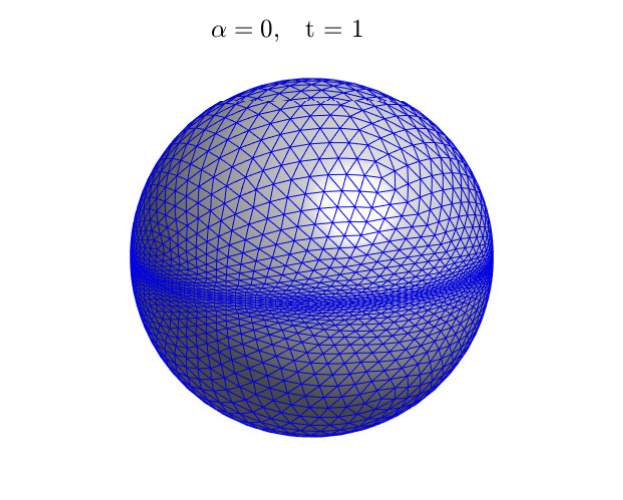}
\includegraphics[width=0.24\textwidth]{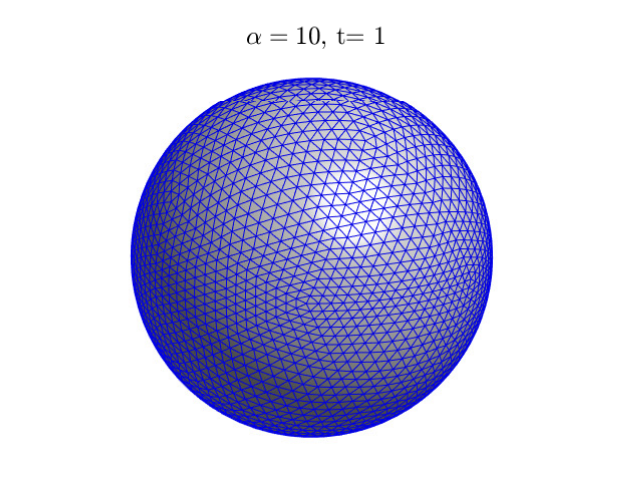}
\includegraphics[width=0.24\textwidth]{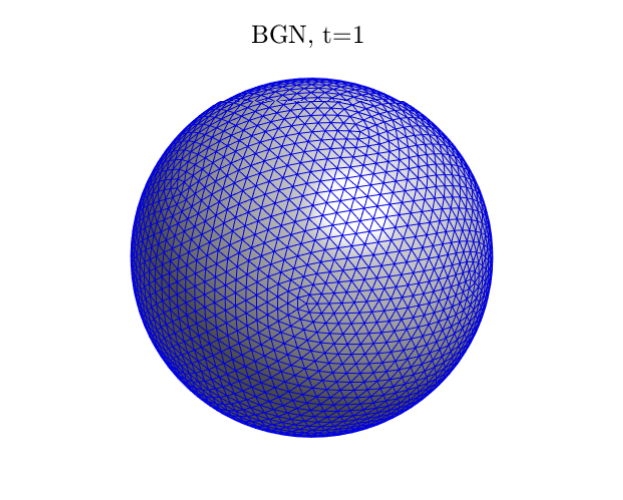}
\includegraphics[width=0.95\textwidth]{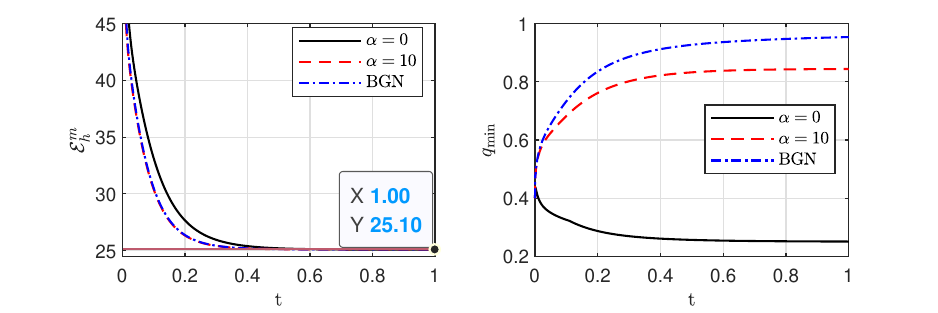}
\caption{Evolution for a $4\times4\times1$ ellipsoid under
Willmore flow with $\bkap=0$, $\ttau=10^{-3}$, and $T=1$.
The upper panels show, from left to right, $\Gamma^0$ and $\Gamma^M$ obtained with $\alpha=0$, $\alpha=10$, and the BGN scheme.  The lower
panels show the discrete bending energy $\E_h^m$ and the minimum mesh quality
$q_{\rm min}^m$.}
\label{fig:ellipsoid-tangential}
\end{figure}

{To further assess the effect of the relaxation parameter, we repeat
the ellipsoid experiment for $\alpha\in\{0,0.1,0.5,2,10,100\}$ (see
Fig.~\ref{fig:ellipsoidalpha}). The energy decreases for all tested
values of $\alpha$ and reaches similar levels at long times. For
$\alpha\ge 2$, the energy curves nearly overlap at the scale shown.
At $T=1$, increasing $\alpha$ from $0$ to $2$ raises the minimum mesh
quality $q_{\rm min}^M$ from approximately $0.25$ to $0.81$. Further
increases in $\alpha$ yield diminishing gains, with $q_{\rm min}^M$
reaching approximately $0.84$ for $\alpha=10$ and $0.85$ for $\alpha=100$.}

\begin{figure}
\centering
\includegraphics[width=0.8\textwidth]{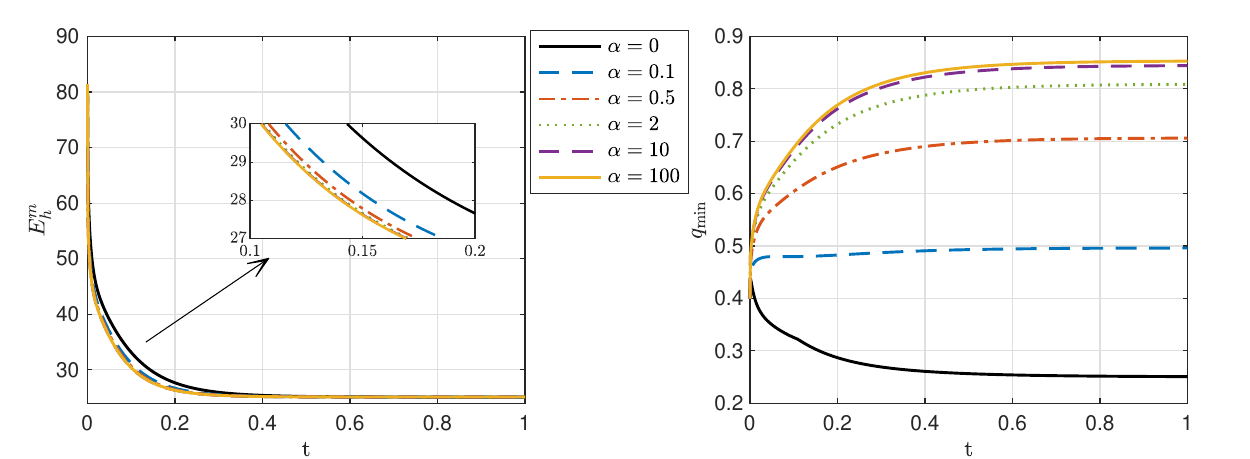}
\caption{{Discrete bending energy $\E_h^m$ (left) and minimum mesh
quality $q_{\rm min}^m$ (right) for the $4\times4\times1$ ellipsoid under
Willmore flow with different values of $\alpha$, using the same initial
mesh, $\bkap=0$, $\ttau=10^{-3}$, and $T=1$ as in
Fig.~\ref{fig:ellipsoid-tangential}. The inset enlarges the energy curves
over $0.1\le t\le0.2$.}}
\label{fig:ellipsoidalpha}
\end{figure}

We repeat this comparison for a closed cigar of total length $4$ with two
hemispherical caps of radius $0.5$, so that its full dimensions are
$4\times1\times1$.  The initial mesh consists of $J=5818$ triangles and
$K=2911$ vertices, and we again take $\bkap=0$, $\ttau=10^{-3}$, and $T=1$.
Figure~\ref{fig:cigar-tangential} compares $\alpha=0$ with $\alpha=10$.
The two energy curves are nearly indistinguishable and approach the spherical
Willmore energy $8\pi$.  By contrast, the relaxed-MDR tangential motion
maintains a visibly more uniform mesh and a larger value of
$q_{\rm min}^m$.

\begin{figure}[!htbp]
\centering
\includegraphics[width=0.3\textwidth]{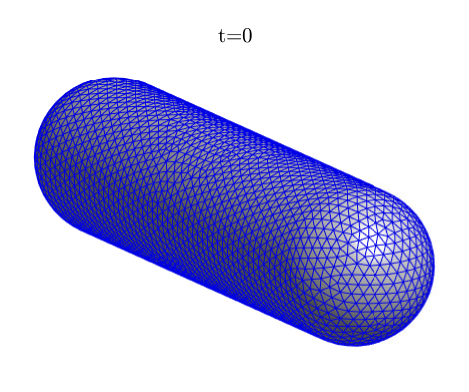}
\includegraphics[width=0.3\textwidth]{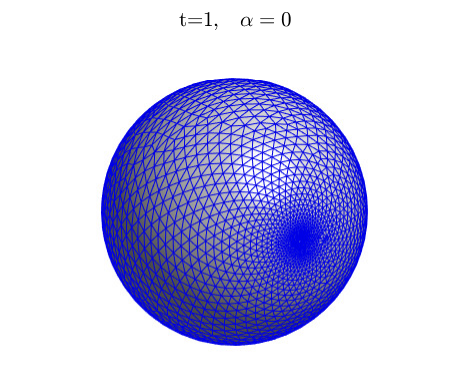}
\includegraphics[width=0.3\textwidth]{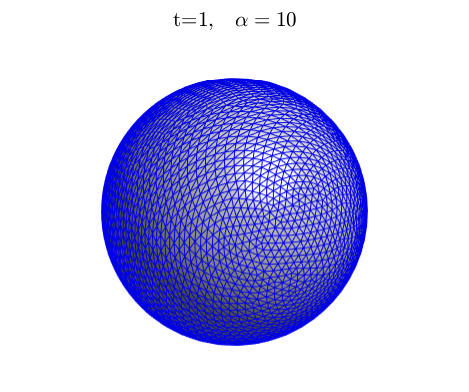}
\includegraphics[width=0.95\textwidth]{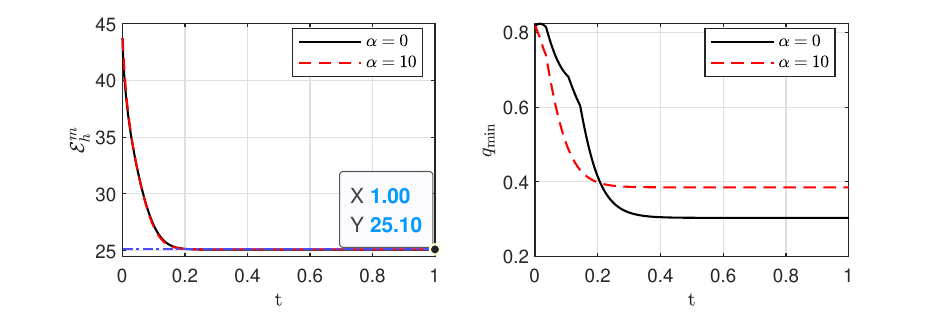}
\caption{Evolution for a $4\times1\times1$ cigar under
Willmore flow with $\bkap=0$, $\ttau=10^{-3}$, and $T=1$.
The upper panels show $\Gamma^0$ (left) and $\Gamma^M$ for
$\alpha=0$ (middle) and $\alpha=10$ (right).  The lower panels show
$\E_h^m$ and $q_{\rm min}^m$; the horizontal reference line in the energy
plot is $8\pi$.}
\label{fig:cigar-tangential}
\end{figure}

We next consider the genus-two experiment studied in
\cite[Fig.~16]{pwfade} and \cite[Fig.~9]{BGN08willmore}, {both of which use additional mesh redistribution/smoothing in this experiment}. The initial surface is the boundary of a
$7\times4\times1$ cuboid after removing two $2\times2\times1$ through-holes.
The initial triangulation contains $J=7512$ triangles and $K=3754$ vertices.
We use $\bkap=0$ and $\ttau=10^{-3}$, and run the relaxed-MDR computation
until $T=5$.  The sharp edges are rapidly smoothed while the two holes are
preserved.  On their common time interval, the relaxed-MDR and BGN {(using scheme~(4.1) of \cite{pwfade} with $\theta=0$ and no additional mesh smoothing)} schemes
produce comparable energy decay.  The BGN mesh, however, deteriorates rapidly,
so that this computation can only be continued to approximately $t=1.2$.
The relaxed-MDR mesh remains nondegenerate up to $T=5$, as shown in
Fig.~\ref{fig:genus-two}.  Here, the final discrete energy is $\E_h^M=43.98$,
remarkably close to the value $44.006$ obtained in \cite{BGN08willmore} on a
much finer mesh with $(J,K)=(79872,39934)$; see \cite{Hsu92,Joshi07} for
related results on genus-two Willmore surfaces.

\begin{figure}[!htbp]
\centering
\includegraphics[width=0.32\textwidth]{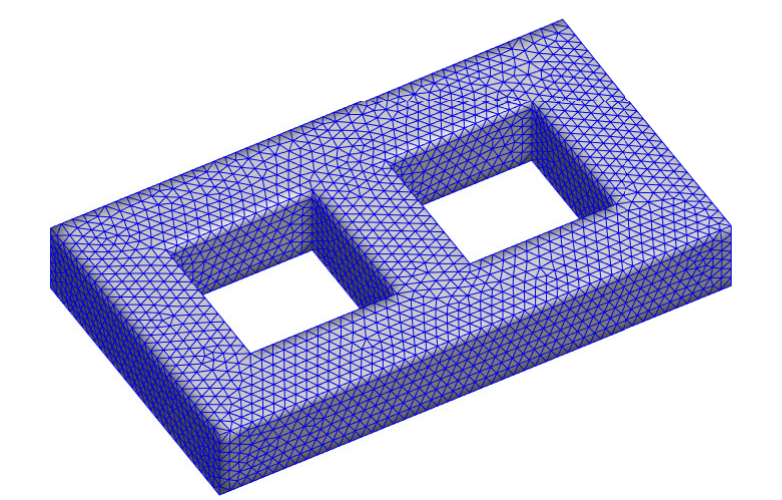}
\includegraphics[width=0.32\textwidth]{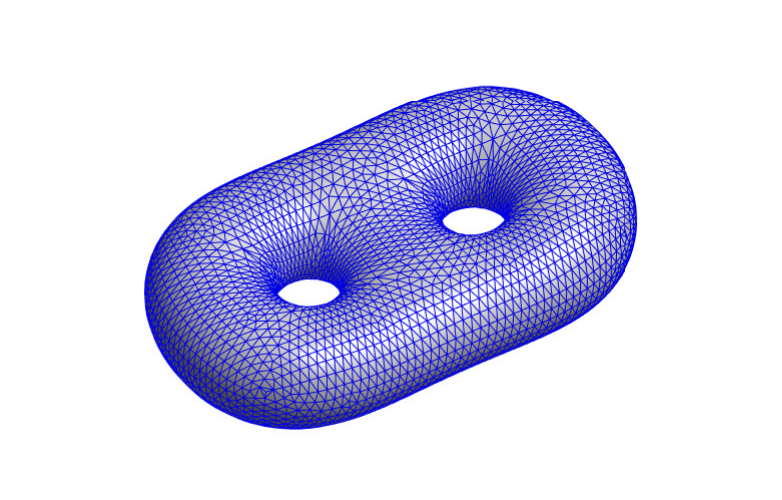}
\includegraphics[width=0.32\textwidth]{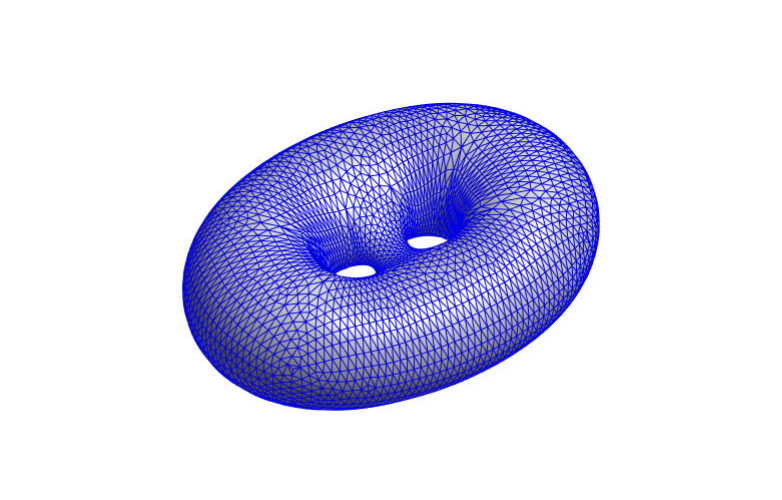}
\includegraphics[width=0.32\textwidth]{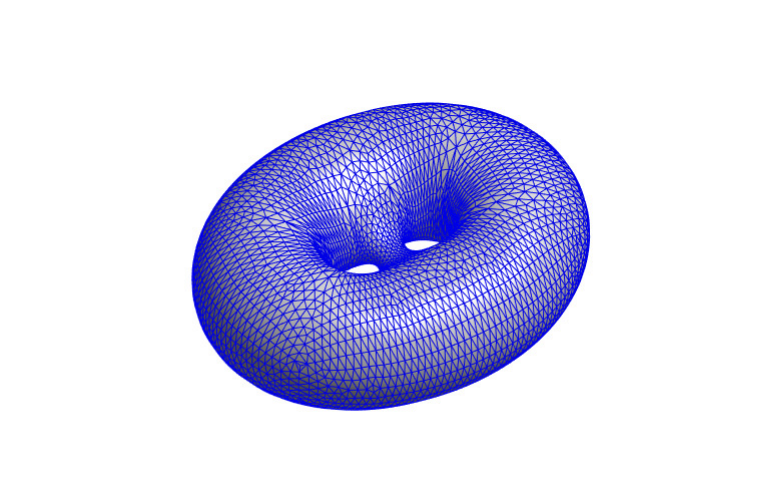}
\includegraphics[width=0.32\textwidth]{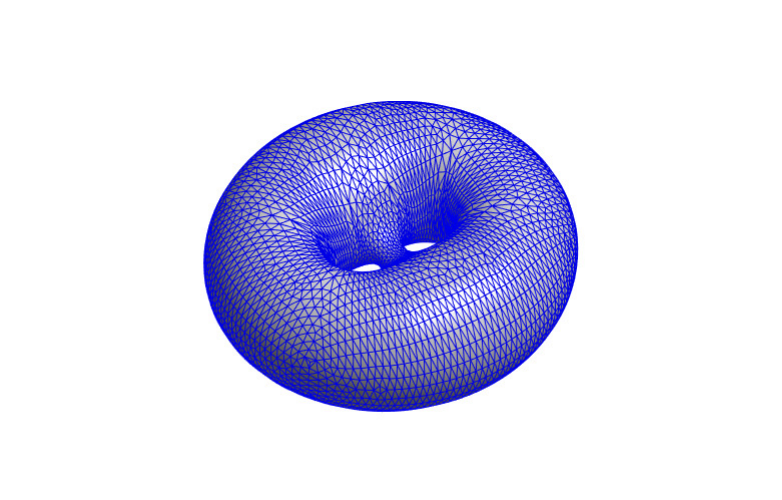}
\includegraphics[width=0.32\textwidth]{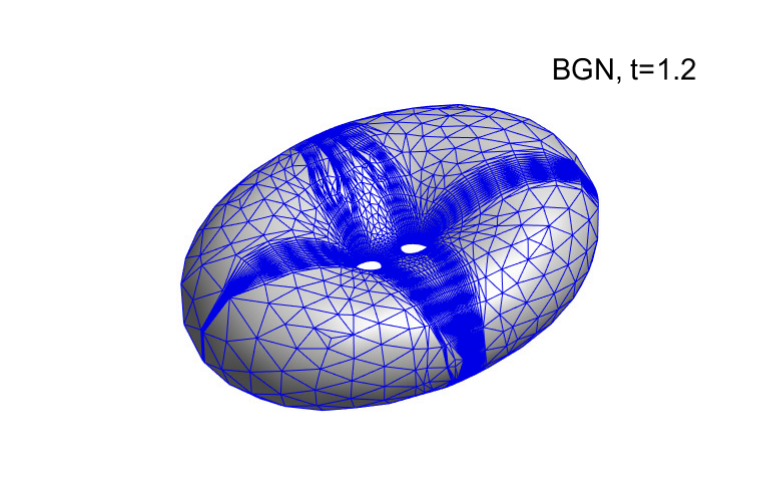}
\includegraphics[width=0.95\textwidth]{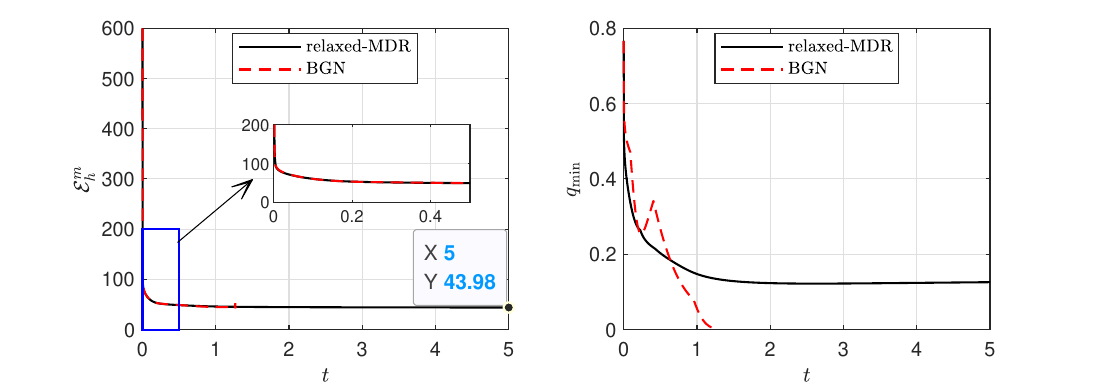}
\caption{Evolution of a genus-two surface under Willmore flow with
$\bkap=0$ and $\ttau=10^{-3}$.  Reading from left to right and top to bottom,
the first five surface panels show the relaxed-MDR solution at
$t=0,0.2,1.2,2$, and $5$; the sixth panel shows the solution at $t=1.2$
computed with the BGN scheme of \cite{pwfade} without additional mesh
regularization.
The lower panels compare $\E_h^m$ and $q_{\rm min}^m$ for the two schemes.}
\label{fig:genus-two}
\end{figure}

As a genus-one example, we consider the torus
\[
\vec u_0(\theta,\varphi)
=
\bigl((2+\cos\varphi)\cos\theta,
      (2+\cos\varphi)\sin\theta,
      \sin\varphi\bigr),
\qquad (\theta,\varphi)\in[0,2\pi)^2.
\]

\begin{figure}[!htbp]
\centering
\includegraphics[width=0.4\textwidth]{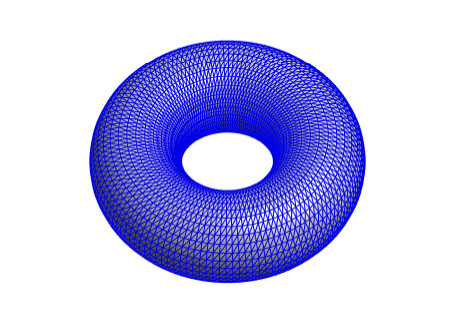}
\includegraphics[width=0.4\textwidth]{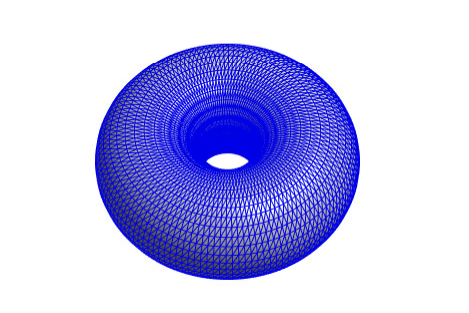}
\includegraphics[width=0.65\textwidth]{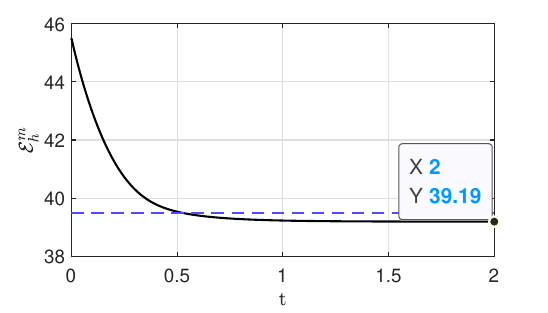}
\caption{Evolution of a torus under Willmore flow, with $\bkap=0$, $\ttau=10^{-3}$, and $T=2$.  The upper panels
show $\Gamma^m$ at $t=0$ and $t=2$, and the lower panel shows $\E_h^m$; the
dashed line marks $4\pi^2$.}
\label{fig:torus}
\end{figure}
The mesh contains $J=8192$ triangles and $K=4096$ vertices.  With
$\bkap=0$, $\ttau=10^{-3}$, and $T=2$, the surface evolves toward a
Clifford torus, while the discrete energy approaches the reference value
$4\pi^2$; see Fig.~\ref{fig:torus}.

\Example{3}{spontaneous curvature effects}

We first examine the effect of a nonzero constant spontaneous curvature by
repeating the cigar experiment from {\bf Example~2} on the same $4\times1\times1$ surface mesh.  We take
$\bkap=-2$,
$\ttau=2\times10^{-4}$, and $T=1$.  The results are reported in Fig.~\ref{fig:cigar-bkap2}.  We then consider the experiment from \cite[Fig.~12]{pwfade}.  The
initial surface is a nonuniform stretched tube of total length $6$, with left
and right spherical-cap radii $1$ and $0.5$, respectively.  We take
$\bkap=-3$, $\ttau=5\times10^{-4}$, and $T=0.3$.  The initial mesh consists
of $J=6020$ triangles and $K=3012$ vertices.  The results are shown in
Fig.~\ref{fig:stretched-bkap3}.  In both examples, the nonzero
spontaneous curvature induces pronounced neck formation and produces several
connected lobes as the discrete bending energy decreases.

\begin{figure}[!htbp]
\centering
\includegraphics[width=0.32\textwidth]{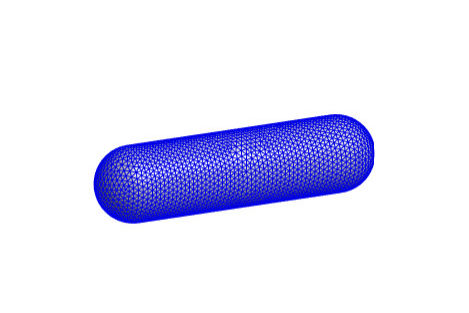}
\includegraphics[width=0.32\textwidth]{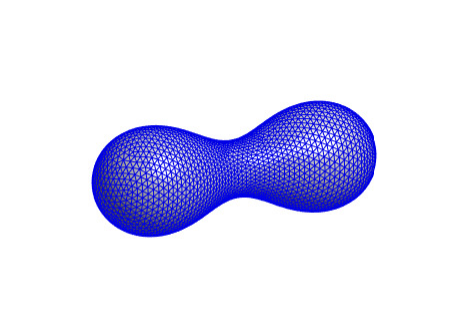}
\includegraphics[width=0.32\textwidth]{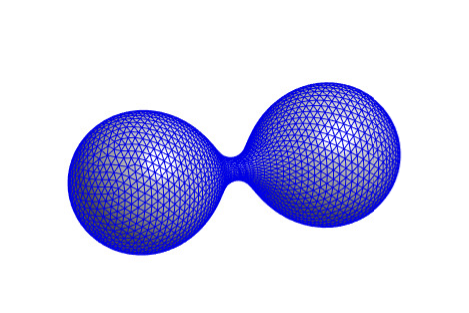}\\
\includegraphics[width=0.7\textwidth]{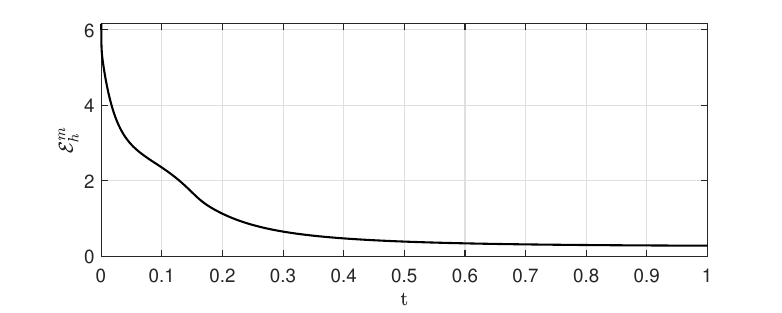}
\caption{Evolution of a $4\times1\times1$ cigar with $\bkap=-2$, $\ttau=2\times10^{-4}$, and $T=1$.  The
surface panels show $\Gamma^m$ at $t=0,0.1$, and $1$, and the lower panel
shows the discrete bending energy $\E_h^m$.}
\label{fig:cigar-bkap2}
\end{figure}

\begin{figure}[!htbp]
\centering
\includegraphics[width=0.3\textwidth]{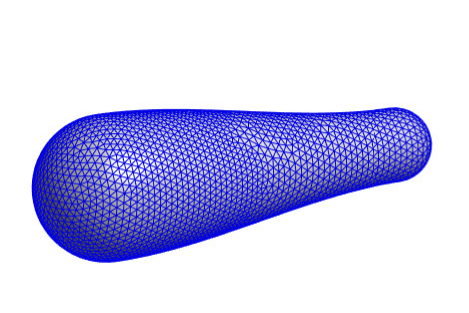}
\includegraphics[width=0.3\textwidth]{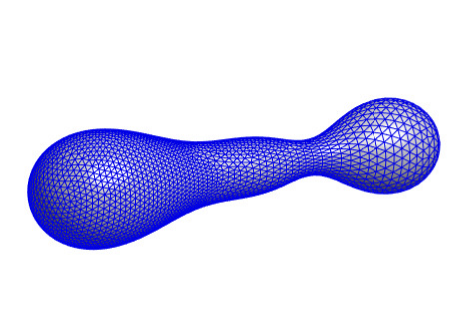}
\includegraphics[width=0.3\textwidth]{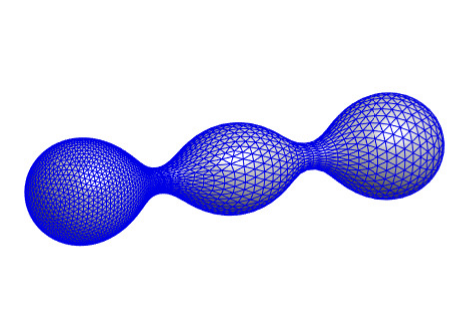}\\
\includegraphics[width=0.6\textwidth]{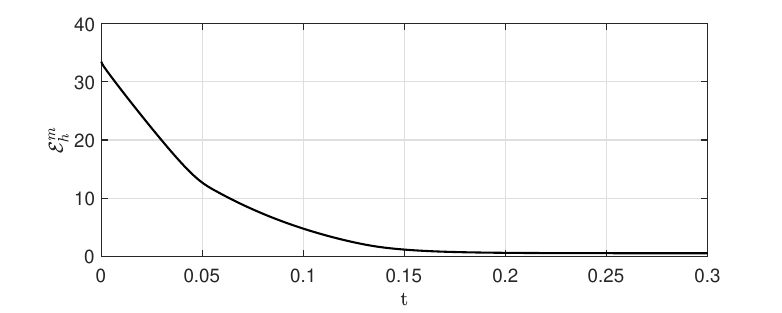}
\caption{Evolution of the nonuniform stretched tube with $\bkap=-3$, $\ttau=5\times10^{-4}$, and $T=0.3$.
The surface panels show $\Gamma^m$ at $t=0,0.05$, and $0.3$, and the lower
panel shows $\E_h^m$.}
\label{fig:stretched-bkap3}
\end{figure}

\Example{4}{surface with open boundaries}

We first consider the upper unit hemisphere
\[
\Gamma^0=\{(x,y,z)\in\mathbb R^3:x^2+y^2+z^2=1,\ z\geq0\},
\]
whose boundary is held fixed.  The initial triangulation has
$J=4096$ triangles and $K=2113$ vertices, and we use $\ttau=10^{-3}$.

Under the Navier condition, we take $\bkap=-4$ with $T=0.5$ and
$\bkap=-1$ with $T=1$.  Both surfaces flatten, with the latter approaching a
nearly planar disk with almost zero bending energy; see
Fig.~\ref{fig:open-navier}.  Under the clamped condition, the boundary
co-normal is fixed at its initial value $\vec\mu=(0,0,-1)$.  We take
$\bkap=-4$ with $T=0.5$ and $\bkap=-1.5$ with $T=5$; the corresponding
evolutions are shown in Fig.~\ref{fig:open-clamped}.

\begin{figure}[!htbp]
\centering
\includegraphics[width=0.99\textwidth]{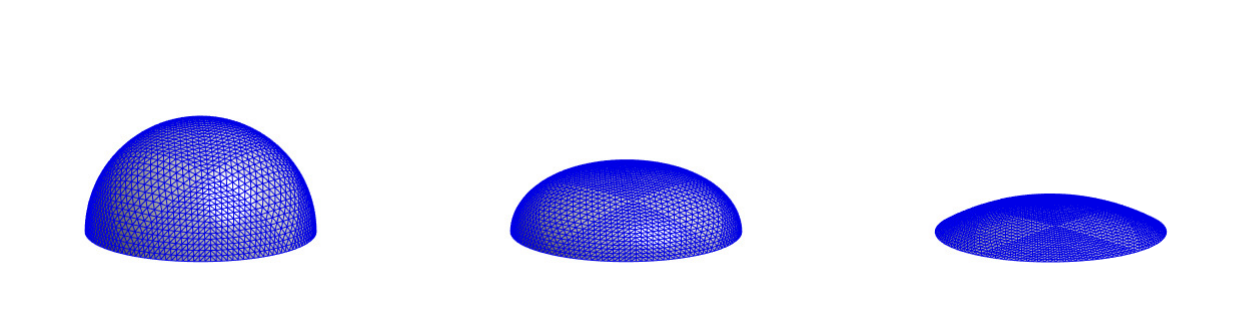}
\includegraphics[width=0.9\textwidth]{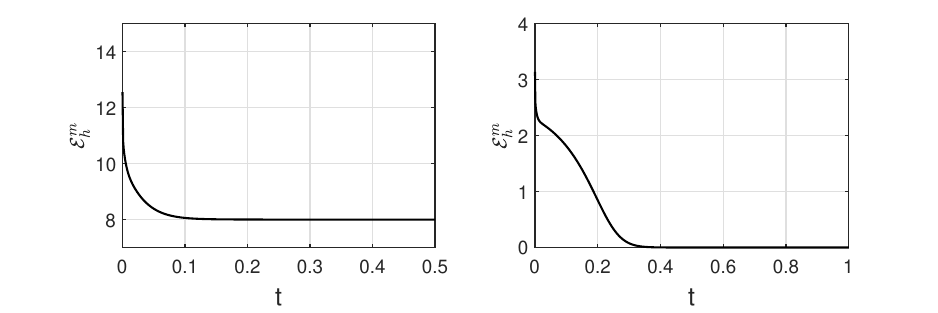}
\caption{Willmore flow of the unit hemisphere with fixed Navier boundary condition.  The upper panels show the
initial surface (left), the solution at $t=0.5$ for $\bkap=-4$ (middle), and
the solution at $t=1$ for $\bkap=-1$ (right).  The lower panels show the
corresponding energy histories.  In both cases $\ttau=10^{-3}$.}
\label{fig:open-navier}
\end{figure}

\begin{figure}[!htbp]
\centering
\includegraphics[width=0.99\textwidth]{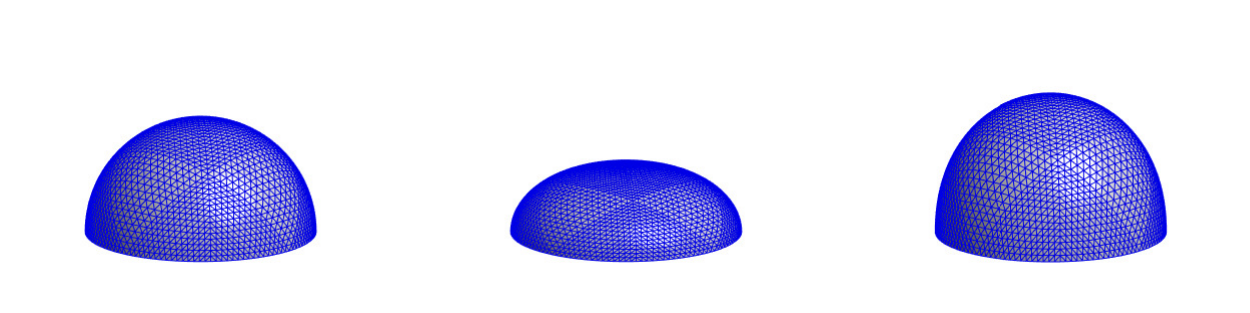}
\includegraphics[width=0.9\textwidth]{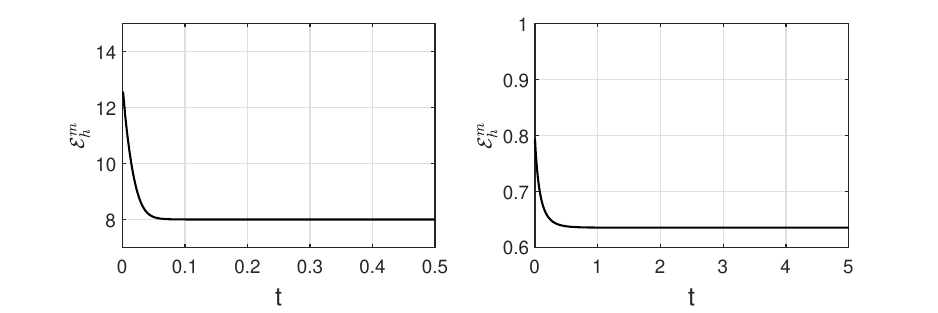}
\caption{Willmore flow of the unit hemisphere with fixed clamped boundary condition.  The upper panels show the
initial surface (left), the solution at $t=0.5$ for $\bkap=-4$ (middle), and
the solution at $t=5$ for $\bkap=-1.5$ (right).  The lower panels show the
corresponding energy histories.  In both cases $\ttau=10^{-3}$.}
\label{fig:open-clamped}
\end{figure}

\subsection{Volume-preserving Willmore and Helfrich flows}
\label{subsec:numerical-helfrich}

To quantify the preservation of the geometric constraints, we define the
relative area and volume losses
\[
e_A^m
:=
{\frac{\mathcal{H}^{d-1}(\Gamma^0)-\mathcal{H}^{d-1}(\Gamma^m)}{\mathcal{H}^{d-1}(\Gamma^0)}},
\qquad
e_V^m
:=
\frac{\vol(\Gamma^0)-\vol(\Gamma^m)}{\vol(\Gamma^0)}.
\]
With this convention, positive values indicate a loss of area or volume,
whereas negative values indicate a gain.

\Example{5}{volume-preserving Willmore flow}

We first impose only the volume constraint on the nonuniform stretched tube
from {\bf Example~3}.  We take $\bkap=-3$,
$\ttau=2\times10^{-4}$, and {$T=1.5$}. Figure~\ref{fig:stretched-volume}
shows that neck formation persists under the volume constraint, while the
relative volume loss remains small throughout the evolution.

\begin{figure}[!htbp]
\centering
\includegraphics[width=0.3\textwidth]{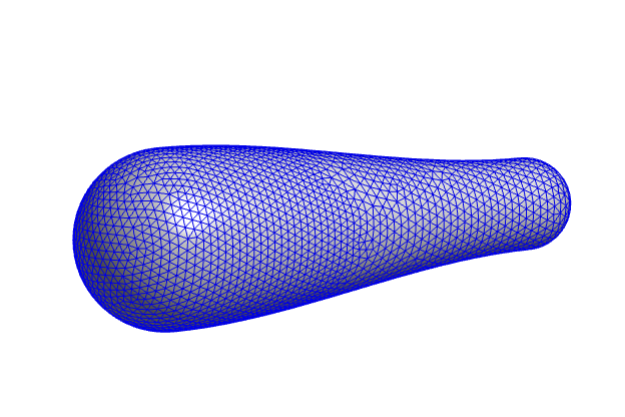}
\includegraphics[width=0.3\textwidth]{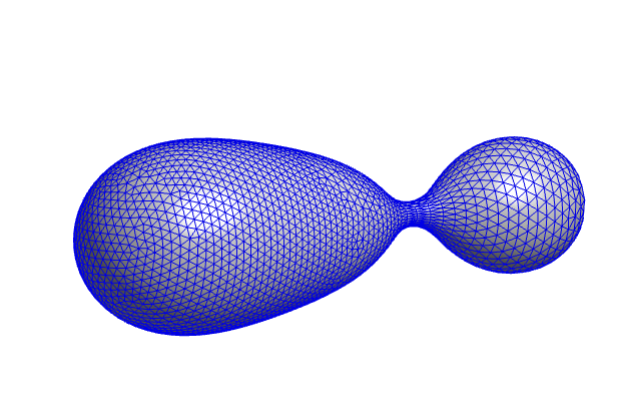}
\includegraphics[width=0.3\textwidth]{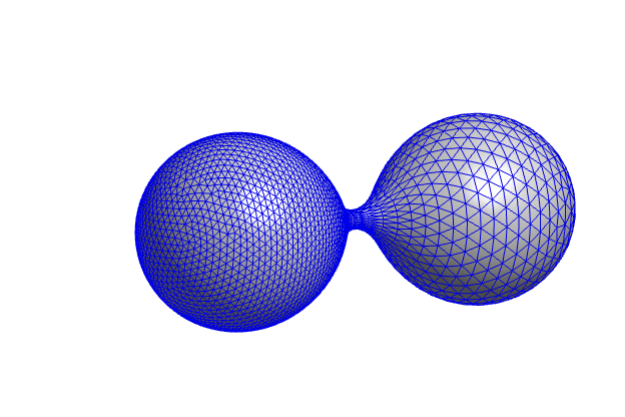}
\includegraphics[width=0.9\textwidth]{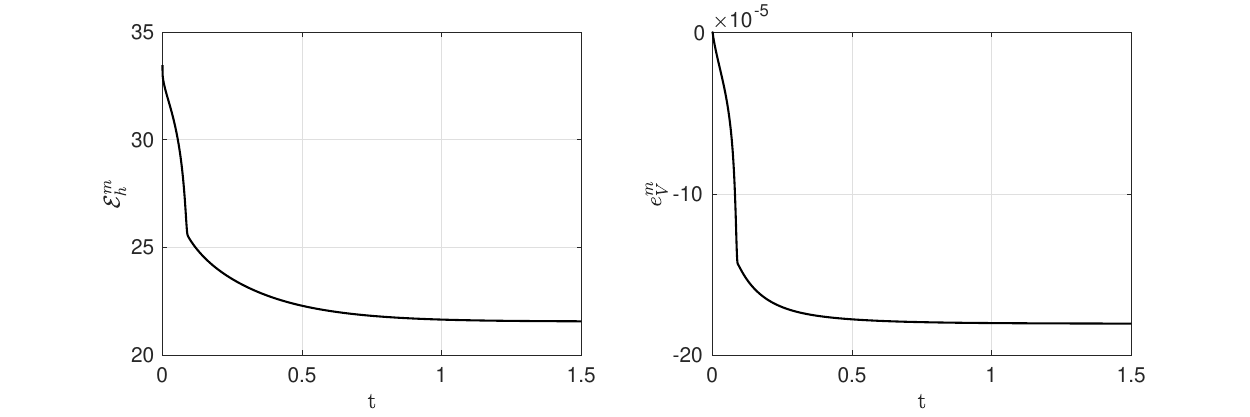}
\caption{Volume-preserving Willmore flow of the nonuniform stretched tube
with $\bkap=-3$, $\ttau=2\times10^{-4}$, and $T=1.5$.  The upper panels show
$\Gamma^m$ at $t=0,0.1$, and $1.5$.  The lower panels show $\E_h^m$ and the
relative volume loss $e_V^m$.}
\label{fig:stretched-volume}
\end{figure}

\Example{6}{Helfrich convergence on a Clifford torus}

We next test the convergence of the Helfrich scheme using the Clifford torus
\[
\vec u_0(\theta,\varphi)
=
\bigl((\sqrt2+\cos\varphi)\cos\theta,
      (\sqrt2+\cos\varphi)\sin\theta,
      \sin\varphi\bigr),
\qquad (\theta,\varphi)\in[0,2\pi)^2.
\]
With $\bkap=0$ and both the area and volume constraints imposed, this surface
is stationary.  Its exact area, enclosed volume, and bending energy are
\[
{\mathcal{H}^{d-1}(\Gamma)}=4\sqrt2\pi^2,
\qquad
\vol(\Gamma)=2\sqrt2\pi^2,
\qquad
\E=4\pi^2.
\]
For a vertex $\vec q=(q_x,q_y,q_z)^T\in\mathcal Q^m$,
we define the geometric and energy errors by
\begin{equation*}
e_\Gamma^m
:=
\max_{\vec q\in\mathcal Q^m}
\left|
\sqrt{
\left(\sqrt{(q_x)^2+(q_y)^2}-\sqrt2\right)^2
+(q_z)^2
}-1
\right|,
\qquad 
e_E^m
:=\frac{|\E_h^m-4\pi^2|}{4\pi^2}.
\end{equation*}
The numerical errors and convergence orders are reported in Table~\ref{tab:clifford-errors}. Both the geometric errors and the relative energy errors exhibit approximately second-order convergence, while the relative violations of the two geometric constraints decrease as the spatial and temporal discretization parameters are refined.

\begin{table*}[!htbp]
\centering
\caption{Errors and experimental orders of convergence for the stationary
Clifford torus under Helfrich flow with $\bkap=0$ and $T=1$.}
\label{tab:clifford-errors}
\resizebox{\textwidth}{!}{%
\begin{tabular}{cc@{\qquad}rr@{\qquad}rr@{\qquad}rr}
\toprule
&
& \multicolumn{2}{c}{Geometric error}
& \multicolumn{2}{c}{Relative energy error}
& \multicolumn{2}{c}{Constraint errors} \\
\cmidrule(lr){3-4}\cmidrule(lr){5-6}\cmidrule(lr){7-8}
$(J,K)$ & $\ttau$
& $\max_m e_\Gamma^m$ & EOC
& $\max_m e_E^m$ & EOC
& $\max_m|e_A^m|$ & $\max_m|e_V^m|$ \\
\midrule
$(256,128)$ & $4.00\mathrm{E}{-2}$
& $5.12\mathrm{E}{-2}$ & --
& $9.02\mathrm{E}{-2}$ & --
& $4.54\mathrm{E}{-5}$ & $7.83\mathrm{E}{-5}$ \\
$(1024,512)$ & $1.00\mathrm{E}{-2}$
& $1.45\mathrm{E}{-2}$ & 1.81
& $2.59\mathrm{E}{-2}$ & 1.80
& $1.35\mathrm{E}{-6}$ & $2.68\mathrm{E}{-6}$ \\
$(4096,2048)$ & $2.50\mathrm{E}{-3}$
& $3.60\mathrm{E}{-3}$ & 2.02
& $6.68\mathrm{E}{-3}$ & 1.95
& $2.59\mathrm{E}{-8}$ & $5.11\mathrm{E}{-8}$ \\
\bottomrule
\end{tabular}
}
\end{table*}

\Example{7}{Helfrich flow on an ellipsoid}

We next consider an
oblate ellipsoid of dimensions $4\times4\times1$, whose
initial mesh contains $J=5120$ triangles and $K=2562$ vertices. We take $\bkap=0$, $\ttau=10^{-3}$ and $T=0.2$.  The
ellipsoid evolves toward a biconcave equilibrium while the relative area and
volume losses remain small; see Fig.~\ref{fig:helfrich-ellipsoid}.

\begin{figure}[!htbp]
\centering
\includegraphics[width=0.3\textwidth]{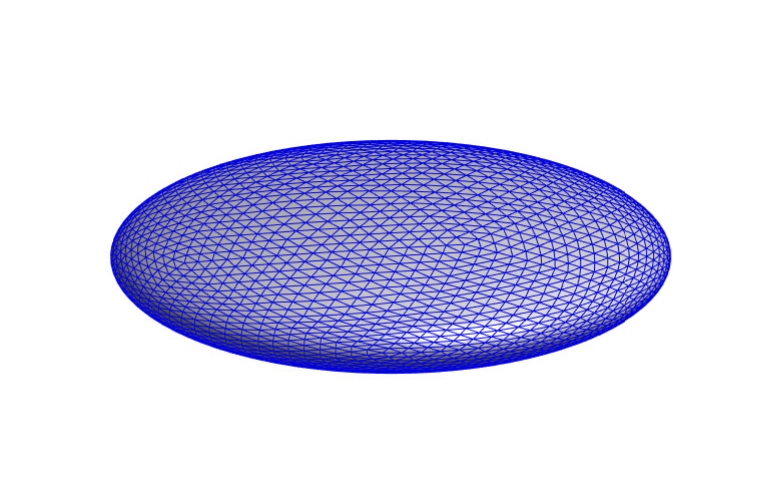}
\includegraphics[width=0.3\textwidth]{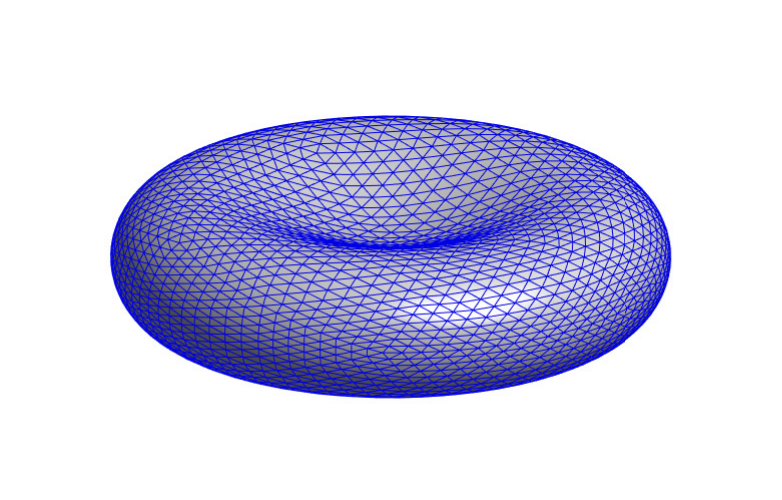}
\includegraphics[width=0.3\textwidth]{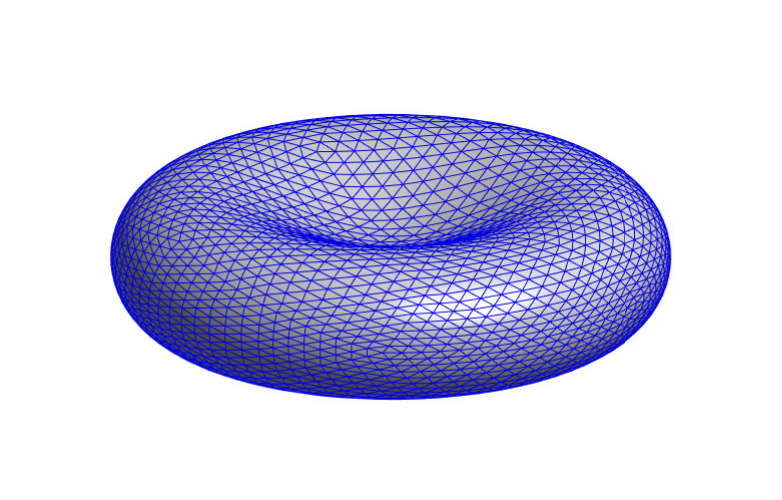}
\includegraphics[width=0.9\textwidth]{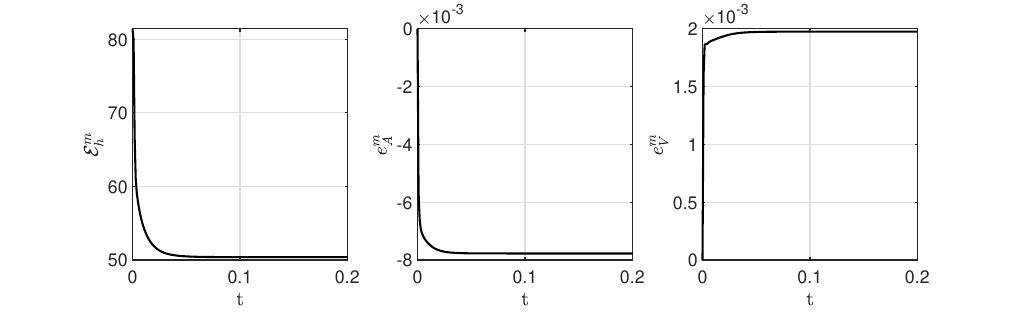}
\caption{Helfrich flow of a $4\times4\times1$ ellipsoid with $\bkap=0$,
$\ttau=10^{-3}$, and $T=0.2$.
The upper panels show $\Gamma^m$ at $t=0,0.05$, and $0.2$.  The lower panels
show $\E_h^m$, $e_A^m$, and $e_V^m$.}
\label{fig:helfrich-ellipsoid}
\end{figure}

{We further consider threefold and sevenfold azimuthal radial
perturbations of the initial ellipsoid with $\bkap=-2$, where the equatorial radius varies between $2\pm0.04$. In
Fig.~\ref{fig:helfrich-ellipsoid3fold}, the threefold perturbation grows
into three pronounced protrusions. The energy first drops rapidly, then
varies more slowly before undergoing a second substantial decrease as
the protrusions elongate. The sevenfold perturbation in
Fig.~\ref{fig:helfrich-ellipsoid7fold} instead becomes less pronounced
during the initial relaxation, and the surface develops an approximately
axisymmetric biconcave shape. Its energy
decreases rapidly towards a much flatter regime, without the second large
decrease observed in the threefold case. These computations show that a
prescribed initial radial mode need not produce the same number of
persistent protrusions under constant spontaneous curvature. In both
runs, the relative area and volume deviations remain of order $10^{-3}$
and $10^{-4}$, respectively.}

\begin{figure}[!htbp]
\centering
\includegraphics[width=0.24\textwidth]{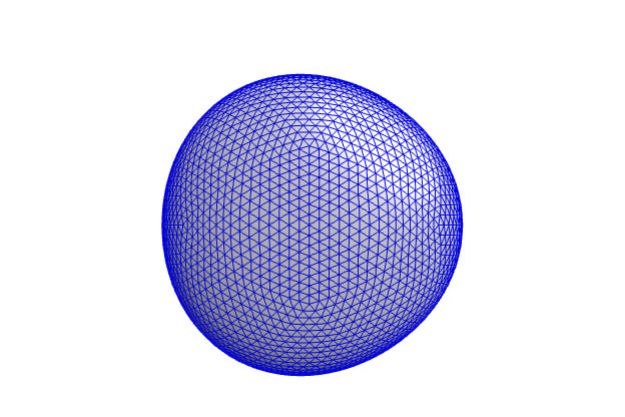}
\includegraphics[width=0.24\textwidth]{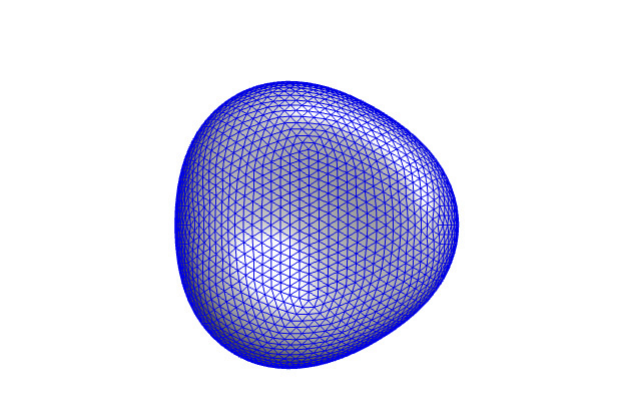}
\includegraphics[width=0.24\textwidth]{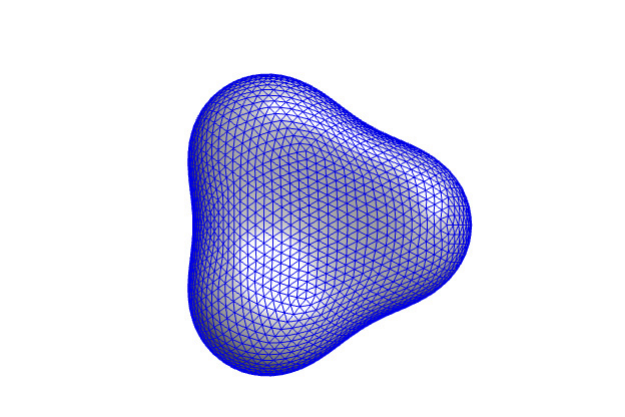}
\includegraphics[width=0.24\textwidth]{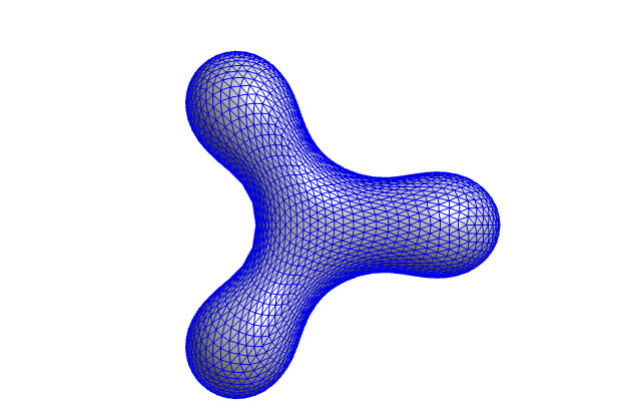}
\includegraphics[width=0.8\textwidth]{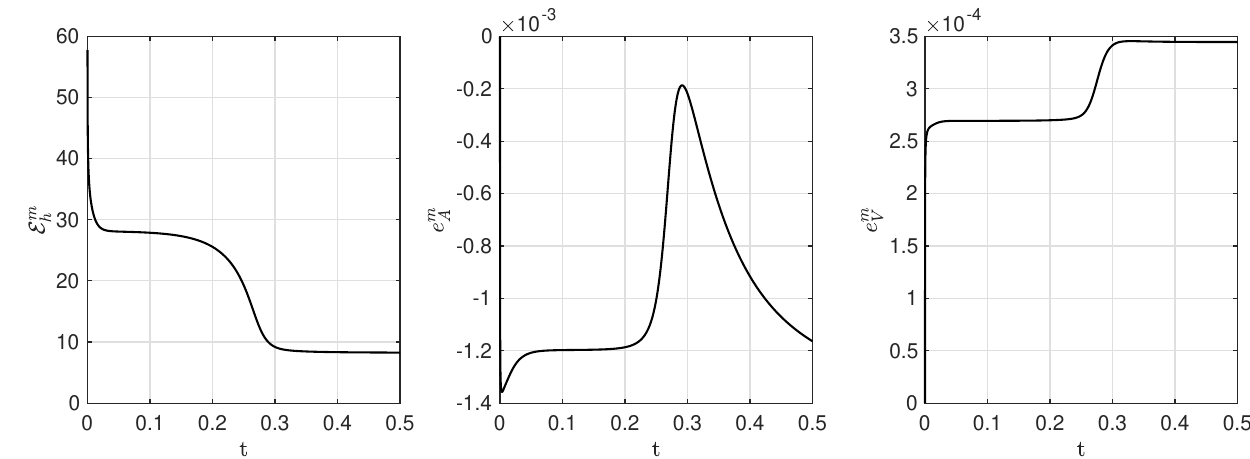}
\caption{{Helfrich flow of a $4\times4\times1$ ellipsoid with $3$-fold radius perturbations.
The upper panels show $\Gamma^m$ at $t=0,0.1,0.2$, and $0.5$. The lower panels
show $\E_h^m$, $e_A^m$, and $e_V^m$. Here $\bkap=-2$,
$\ttau=10^{-4}$, and $T=0.5$.}}
\label{fig:helfrich-ellipsoid3fold}
\end{figure}

\begin{figure}[!htbp]
\centering
\includegraphics[width=0.24\textwidth]{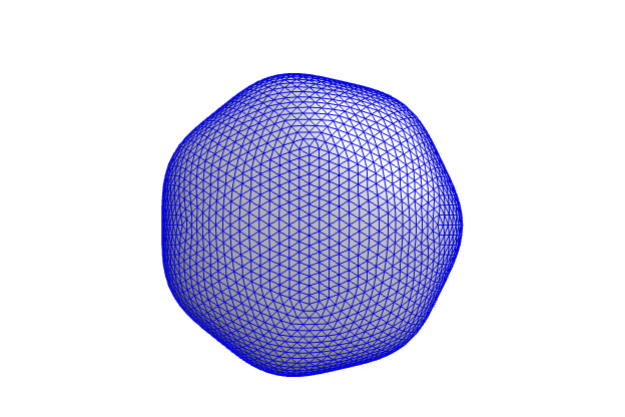}
\includegraphics[width=0.24\textwidth]{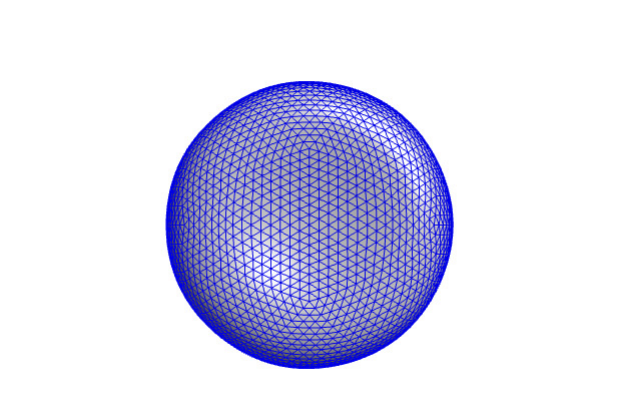}
\includegraphics[width=0.24\textwidth]{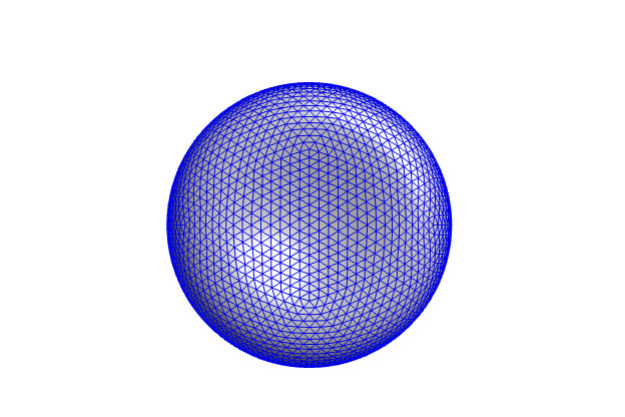}
\includegraphics[width=0.24\textwidth]{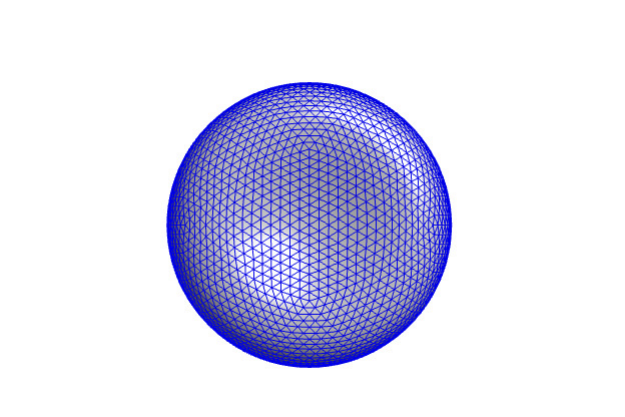}
\includegraphics[width=0.8\textwidth]{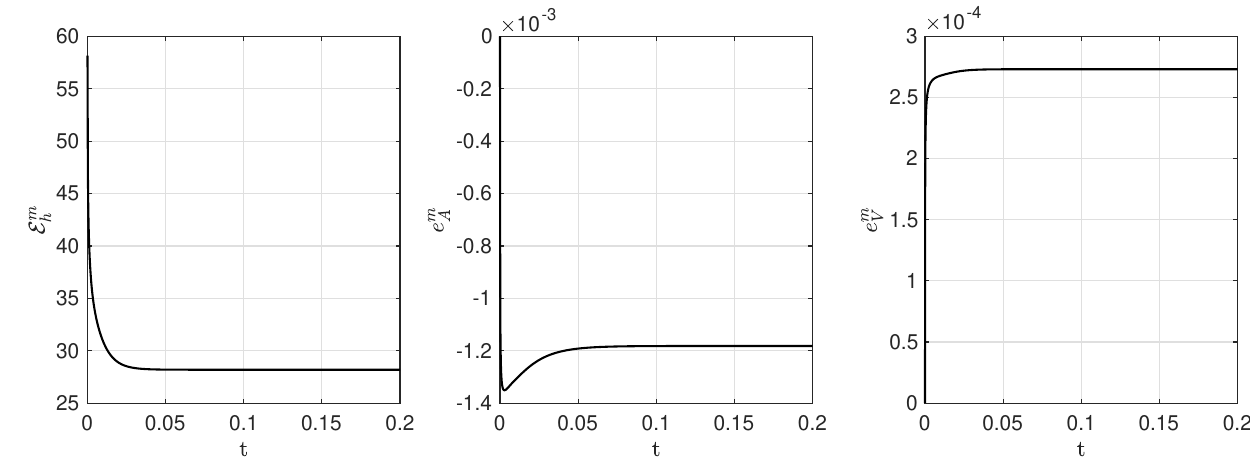}
\caption{{Helfrich flow of a $4\times4\times1$ ellipsoid with $7$-fold radius perturbations.
The upper panels show $\Gamma^m$ at $t=0,0.025,0.05$, and $0.2$. The lower panels
show $\E_h^m$, $e_A^m$, and $e_V^m$. Here $\bkap=-2$,
$\ttau=10^{-4}$, and $T=0.2$.}}
\label{fig:helfrich-ellipsoid7fold}
\end{figure}

\FloatBarrier
\Example{8}{Helfrich flow of a rounded flat plate}

Finally, we consider a rounded flat plate of dimensions
$4.8\times5.2\times1$ and use an initial mesh with $J=5120$ triangles and
$K=2562$ vertices.  We set $\bkap=-2$,
$\ttau=10^{-4}$, and $T=0.25$.  As shown in Fig.~\ref{fig:helfrich-flat-plate}, the slight in-plane asymmetry selects the
orientation of the four emerging lobes.  Similar multilobed morphologies were
reported in {\cite{wintz96starfish, BGN17twofluidic}}.  The time histories further demonstrate
energy decay and accurate preservation of the discrete area and volume.

\begin{figure}[!htbp]
\centering
\includegraphics[width=0.24\textwidth]{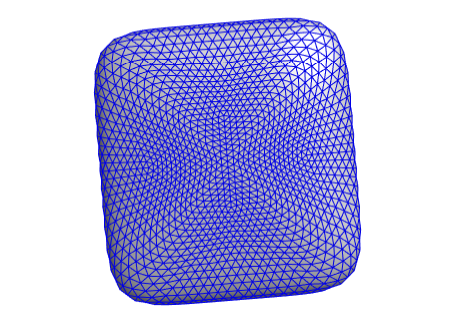}
\includegraphics[width=0.24\textwidth]{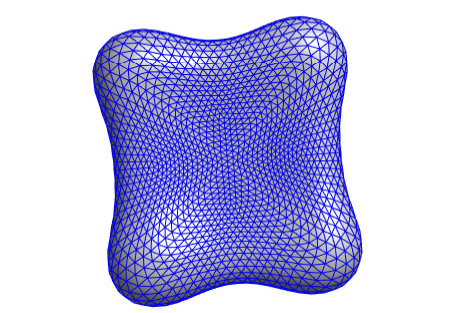}
\includegraphics[width=0.24\textwidth]{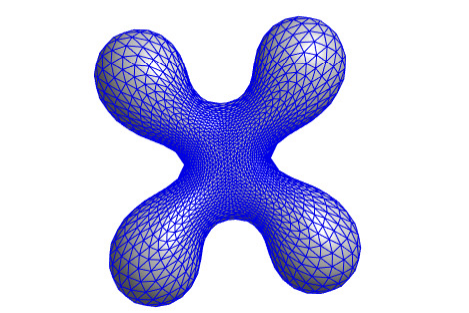}
\includegraphics[width=0.24\textwidth]{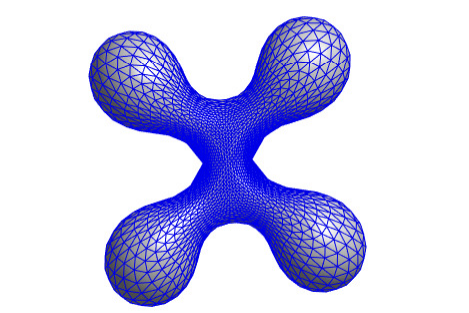}
\includegraphics[width=0.9\textwidth]{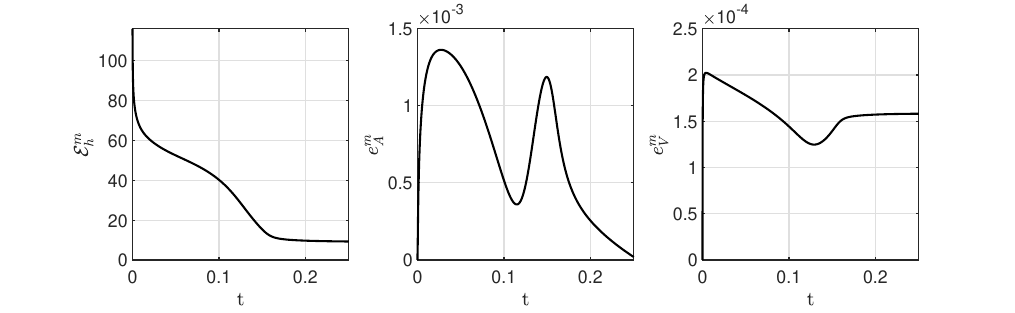}
\caption{Helfrich flow of the rounded
$4.8\times5.2\times1$ flat plate with $\bkap=-2$, $\ttau=10^{-4}$, and
$T=0.25$.  The upper panels show $\Gamma^m$ at $t=0,0.05,0.15$, and $0.25$.
The lower panels show $\E_h^m$, $e_A^m$, and $e_V^m$.}
\label{fig:helfrich-flat-plate}
\end{figure}

\section{Concluding remarks}
\label{sec:con}

We have developed a two-stage constrained Onsager formulation for parametric
approximations of Willmore and Helfrich flows, in which the dissipation, geometric constraints, and tangential mesh redistribution are treated within a unified velocity-based framework.  A weak curvature PDE constraint
first represents the bending-energy variation in curvature-vector form.
Onsager's principle then determines the geometric evolution through
normal-velocity dissipation, while the relaxed-MDR condition independently
selects the tangential motion.  This separation allows the strength of mesh
redistribution to be adjusted without changing the normal gradient-flow
structure.

The framework yields a continuous piecewise linear semidiscretization
that satisfies an exact energy-dissipation identity and exactly preserves area and volume when the corresponding constraints are imposed.  {It also accommodates constant spontaneous curvature and open surfaces with fixed
boundaries under Navier or clamped conditions.}  The proposed linearly
implicit time discretizations lead to mixed linear systems that are uniquely
solvable under geometric nondegeneracy assumptions.  The numerical results
corroborate convergence and energy decay, quantify the preservation of the
geometric constraints, and demonstrate effective mesh control for large
deformations, complex topology, and open-surface evolutions.

\section*{Acknowledgements}

This work was partially supported by the National Natural Science Foundation of China
(No. 12401572, Q.Z.).

\bibliographystyle{abbrv}
\bibliography{bib}

@article{Bank96, 
  title={ { An algorithm for coarsening unstructured meshes}},
  author={Bank, Randolph E and Xu, Jinchao},
  journal={Numer. Math. },
  volume={73},
  number={1},
  pages={1--36},
  year={1996},
  publisher={Springer}
}

@article{BalzaniR12,
AUTHOR = {Balzani, Nadine and Rumpf, Martin},
TITLE = {A nested variational time discretization for parametric
{W}illmore flow},
JOURNAL = {Interfaces Free Bound.},
FJOURNAL = {Interfaces and Free Boundaries. Mathematical Analysis,
Computation and Applications},
VOLUME = {14},
YEAR = {2012},
NUMBER = {4},
PAGES = {431--454},
ISSN = {1463-9963,1463-9971},
MRCLASS = {53C44 (65N30)},
MRNUMBER = {3016467},
DOI = {10.4171/IFB/287},
URL = {https://doi.org/10.4171/IFB/287},
}

@article{Blatt09singular,
  title={A singular example for the {W}illmore flow},
  author={Blatt, Simon},
  journal = {Analysis (Munich)},
  volume = {29},
  number = {4}, 
  pages = {407--430},
  year={2009},
  publisher={Oldenbourg Wissenschaftsverlag GmbH M{\"u}nchen Germany}
}

@article{BGN17twofluidic,
  title={{Finite element approximation for the dynamics of fluidic two-phase biomembranes}},
  author={Barrett, John W and Garcke, Harald and N{\"u}rnberg, Robert},
  journal={ESAIM Math. Model. Numer. Anal.},
  volume={51},
  number={6},
  pages={2319--2366},
  year={2017}
}

@misc{contri26,
  title={Mechanochemical feedback between cell shape and intracellular mechanics revealed by a finite-element framework},
  author={Contri, Alessandro and Francis, Emmet A. and Massing, Andr{\'e} and Rangamani, Padmini},
  howpublished={{bioRxiv} preprint 2026.07.03.736361},
  year={2026},
  doi={10.64898/2026.07.03.736361}
}

@article{BaoL25,
AUTHOR = {Bao, Weizhu and Li, Yifei},
TITLE = {An energy-stable parametric finite element method for the
planar {W}illmore flow},
JOURNAL = {SIAM J. Numer. Anal.},
VOLUME = {63},
YEAR = {2025},
NUMBER = {1},
PAGES = {103--121},
ISSN = {0036-1429,1095-7170},
MRCLASS = {65M60 (35K55 65M12)},
MRNUMBER = {4849434},
DOI = {10.1137/24M1633893},
URL = {https://doi.org/10.1137/24M1633893},
}

@incollection{Barrett20,
  title={Parametric finite element approximations of curvature-driven interface evolutions},
  author={Barrett, John W and Garcke, Harald and N{\"u}rnberg, Robert},
  editor={Bonito, Andrea and Nochetto, Ricardo H.},
  booktitle={Handbook of Numerical Analysis},
  volume={21},
  pages={275--423},
  year={2020},
  publisher={Elsevier},
  address={Amsterdam},
  doi={10.1016/bs.hna.2019.05.002}
}

@article{BGN07,
  title={A parametric finite element method for fourth order geometric evolution equations},
  author={Barrett, John W and Garcke, Harald and N{\"u}rnberg, Robert},
  journal={J. Comput. Phys.},
  volume={222},
  number={1},
  pages={441--467},
  year={2007},
  publisher={Elsevier}
}

@article{BGN08parametric,
  title={On the parametric finite element approximation of evolving hypersurfaces in $\mathbb{R}^3$},
  author={Barrett, John W and Garcke, Harald and N{\"u}rnberg, Robert},
  journal={J. Comput. Phys.},
  volume={227},
  number={9},
  pages={4281--4307},
  year={2008},
  publisher={Elsevier}
}

@article{BGN08willmore,
  title={Parametric approximation of {W}illmore flow and related geometric evolution equations},
  author={Barrett, John W and Garcke, Harald and N{\"u}rnberg, Robert},
  journal={SIAM J. Sci. Comput.},
  volume={31},
  number={1},
  pages={225--253},
  year={2008},
  publisher={SIAM}
}

@article{pwfade,
  title={Computational parametric {W}illmore flow with spontaneous curvature and area difference elasticity effects},
  author={Barrett, John W and Garcke, Harald and N{\"u}rnberg, Robert},
  journal={SIAM J. Numer. Anal.},
  volume={54},
  number={3},
  pages={1732--1762},
  year={2016},
  publisher={SIAM}
}

@article{pwfopen,
  title={Stable variational approximations of boundary value problems for {Willmore} flow with {Gaussian} curvature},
  author={Barrett, John W and Garcke, Harald and N{\"u}rnberg, Robert},
  journal={IMA J. Numer. Anal.},
  volume={37},
  number={4},
  pages={1657--1709},
  year={2017},
  publisher={Oxford University Press}
}

@article{Bretin15phase,
  title={Phase-field approximations of the {Willmore} functional and flow},
  author={Bretin, Elie and Masnou, Simon and Oudet, Edouard},
  journal={Numer. Math.},
  volume={131},
  pages={115--171},
  year={2015},
  publisher={Springer}
}

@article{BONITO2010,
  title={Parametric {FEM} for geometric biomembranes},
  author={Bonito, Andrea and Nochetto, Ricardo H. and Pauletti, M. Sebastian},
  journal={J. Comput. Phys.},
  volume={229},
  number={9},
  pages={3171--3188},
  year={2010},
  doi={10.1016/j.jcp.2009.12.036}
}

@article{GNZ25willmore,
author = {Garcke, Harald and N{\"u}rnberg, Robert and Zhao, Quan},
title = {Stable fully discrete finite element methods with {BGN} tangential motion for {Willmore} flow of planar curves},
JOURNAL = {J. Sci. Comput.},
volume = 105,
number = 2,
pages = {45},
numpages = {23},
year = 2025,
doi = {10.1007/s10915-025-03068-9},
url = {https://doi.org/10.1007/s10915-025-03068-9}
}

@Article{ClarenzDDRR04,
author = {Clarenz, Ulrich and Diewald, Udo and Dziuk, Gerd and Rumpf, Martin 
and Rusu, Raluca},
title = {A finite element method for surface restoration with smooth
boundary conditions},
journal = {Comput. Aided Geom. Design},
fjournal = {Computer Aided Geometric Design},
volume = 21,
year = 2004,
number = 5,
pages = {427--445},
issn = {0167-8396},
coden = {CAGDEX},
mrclass = {65D18 (65D17 65M60)},
mrnumber = {2058390},
DOI = {10.1016/j.cagd.2004.02.004},
URL = {https://doi.org/10.1016/j.cagd.2004.02.004},
}

@Article{Dziuk08, comment = {was: Dziuk07},
AUTHOR = {Dziuk, Gerhard},
TITLE = {Computational parametric {W}illmore flow},
JOURNAL = {Numer. Math.},
FJOURNAL = {Numerische Mathematik},
VOLUME = {111},
YEAR = {2008},
NUMBER = {1},
PAGES = {55--80},
ISSN = {0029-599X},
CODEN = {NUMMA7},
MRCLASS = {35K55 (53C44)},
MRNUMBER = {2448203},
DOI = {10.1007/s00211-008-0179-1},
URL = {https://doi.org/10.1007/s00211-008-0179-1}
}

@Article{DziukKS02,
AUTHOR = {Dziuk, Gerhard and Kuwert, Ernst and Sch{\"a}tzle, Reiner},
TITLE = {Evolution of elastic curves in {${\mathbb R}^n$}: {E}xistence and
computation},
JOURNAL = {SIAM J. Math. Anal.},
FJOURNAL = {SIAM Journal on Mathematical Analysis},
VOLUME = {33},
YEAR = {2002},
NUMBER = {5},
PAGES = {1228--1245},
ISSN = {0036-1410},
MRCLASS = {53C44 (35K55 65M60)},
MRNUMBER = {1897710},
MRREVIEWER = {Anders Linn{\'e}r},
DOI = {10.1137/S0036141001383709},
}

@article{Deckelnick09error,
  title={Error analysis for the elastic flow of parametrized curves},
  author={Deckelnick, Klaus and Dziuk, Gerhard},
  journal={Math. Comp},
  volume={78},
  number={266},
  pages={645--671},
  year={2009}
}

@article{Droske04level,
  title={A level set formulation for {Willmore} flow},
  author={Droske, Marc and Rumpf, Martin},
  journal={Interfaces Free Bound.},
  volume={6},
  number={3},
  pages={361--378},
  year={2004}
}

@article{Duan24new,
  title={New artificial tangential motions for parametric finite element approximation of surface evolution},
  author={Duan, Beiping and Li, Buyang},
  journal={SIAM J. Sci. Comput.},
  volume={46},
  number={1},
  pages={A587--A608},
  year={2024},
  publisher={SIAM}
}

@article{Duan2021high,
  title={High-order fully discrete energy diminishing evolving surface finite element methods for a class of geometric curvature flows},
  author={Duan, Beiping and Li, Buyang and Zhang, Zhimin},
  journal={Ann. Appl. Math.},
  volume={37},
  number={4},
  pages={405--436},
  year={2021}
}

@article{Doi15,
  title={Onsager principle as a tool for approximation},
  author={Doi, Masao},
  journal={Chinese Phys. B},
  volume={24},
  number={2},
  pages={020505},
  year={2015},
  doi={10.1088/1674-1056/24/2/020505}
}

@article{Doi11,
  title={{Onsager}'s variational principle in soft matter},
  author={Doi, Masao},
  journal={J. Phys.: Condens. Matter},
  volume={23},
  number={28},
  pages={284118},
  year={2011},
  doi={10.1088/0953-8984/23/28/284118}
}

@article{QianWS06,
  title={A variational approach to moving contact line hydrodynamics},
  author={Qian, Tiezheng and Wang, Xiao-Ping and Sheng, Ping},
  journal={J. Fluid Mech.},
  volume={564},
  pages={333--360},
  year={2006}
}

@article{ZhaoJWSQB24,
  title={Dynamics of small solid particles on substrates of arbitrary topography},
  author={Zhao, Quan and Jiang, Wei and Wang, Yan and Srolovitz, David J. and Qian, Tiezheng and Bao, Weizhu},
  journal={Acta Mater.},
  volume={281},
  pages={120407},
  year={2024},
  doi={10.1016/j.actamat.2024.120407}
}

@article{XuDD16,
  title={Variational method for liquids moving on a substrate},
  author={Xu, Xianmin and Di, Yana and Doi, Masao},
  journal={Phys. Fluids},
  volume={28},
  number={8},
  pages={087101},
  year={2016}
}

@article{JiangZQSB19,
  title={Application of {Onsager}'s variational principle to the dynamics of a solid toroidal island on a substrate},
  author={Jiang, Wei and Zhao, Quan and Qian, Tiezheng and Srolovitz, David J. and Bao, Weizhu},
  journal={Acta Mater.},
  volume={163},
  pages={154--160},
  year={2019}
}

@article{Du04phase,
  title={A phase field approach in the numerical study of the elastic bending energy for vesicle membranes},
  author={Du, Qiang and Liu, Chun and Wang, Xiaoqiang},
  journal={J. Comput. Phys.},
  volume={198},
  number={2},
  pages={450--468},
  year={2004},
  publisher={Elsevier}
}

@article{FrankenRW13,
AUTHOR = {Franken, Martina and Rumpf, Martin and Wirth, Benedikt},
TITLE = {A phase field based {PDE} constrained optimization approach to
time discrete {W}illmore flow},
JOURNAL = {Int. J. Numer. Anal. Model.},
FJOURNAL = {International Journal of Numerical Analysis and Modeling},
VOLUME = {10},
YEAR = {2013},
NUMBER = {1},
PAGES = {116--138},
ISSN = {1705-5105},
MRCLASS = {65M60 (53A07 53C44)},
MRNUMBER = {3011864},
}

@article{Gao26,
author = {Gao, Guangwei and Garcke, Harald and Li, Buyang and Tang, Rong},
title = {An Energy-Stable Minimal Deformation Rate Scheme for Mean Curvature Flow and Surface Diffusion},
journal = {SIAM J. Sci. Comput. },
volume = {48},
number = {1},
pages = {A103-A131},
year = {2026},
doi = {10.1137/25M1753838},
}

@misc{Gao26dualMDR,
      title={Dual formulations of geometric curvature flows and their discretizations}, 
      author={Guangwei Gao and Buyang Li and Rong Tang},
      year={2026},
      howpublished ={arXiv: 2604.18288},
      archivePrefix={arXiv},
      primaryClass={math.NA},
      url={https://arxiv.org/abs/2604.18288}, 
}

@article{GNZ26,
author = {Garcke, Harald and N\"{u}rnberg, Robert and Zhao, Quan},
title = {An Energy-Stable Parametric Finite Element Method for {Willmore} Flow with Normal-Tangential Velocity Splitting},
journal = {SIAM J. Sci. Comput.},
volume = {48},
number = {3},
pages = {A1235-A1259},
year = {2026},
doi = {10.1137/25M1773878}
}

@article{Hu22evolving,
  title={Evolving finite element methods with an artificial tangential velocity for mean curvature flow and {W}illmore flow},
  author={Hu, Jiashun and Li, Buyang},
  journal={Numer. Math.},
  volume={152},
  number={1},
  pages={127--181},
  year={2022},
  publisher={Springer}
}

@article{Helfrich73elastic,
  title={Elastic properties of lipid bilayers: theory and possible experiments},
  author={Helfrich, Wolfgang},
  journal={Z. Naturforsch. C},
  volume={28},
  number={11-12},
  pages={693--703},
  year={1973},
  publisher={Verlag der Zeitschrift f\"ur Naturforschung}
}

@article{Hsu92,
author = {Lucas Hsu and Rob Kusner and John Sullivan},
title = {{Minimizing the squared mean curvature integral for surfaces in space forms}},
volume = {1},
journal = {Experiment. Math.},
number = {3},
publisher = {A K Peters, Ltd.},
pages = {191--207},
year = {1992},
}

@article{Joshi07,
  title={Energy minimizers for curvature-based surface functionals},
  author={Joshi, Pushkar and S{\'e}quin, Carlo},
  journal={Comput. Aided Design Appl.},
  volume={4},
  number={5},
  pages={607--617},
  year={2007},
  publisher={Taylor \& Francis}
}

@article{KovacsLL21,
AUTHOR = {Kov\'{a}cs, Bal\'{a}zs and Li, Buyang and Lubich, Christian},
TITLE = {A convergent evolving finite element algorithm for {W}illmore
flow of closed surfaces},
JOURNAL = {Numer. Math.},
FJOURNAL = {Numerische Mathematik},
VOLUME = {149},
YEAR = {2021},
NUMBER = {3},
PAGES = {595--643},
ISSN = {0029-599X,0945-3245},
MRCLASS = {65M60 (35R01 65M12 65M15)},
MRNUMBER = {4344594},
DOI = {10.1007/s00211-021-01238-z},
URL = {https://doi.org/10.1007/s00211-021-01238-z},
}

@article{KS02,
title ={Gradient Flow for the {Willmore} Functional}, 
author={Kuwert, Ernst and Sch{\"a}tzle, Reiner},
journal = {Comm. Anal. Geom. },  
volume = {10},
number = {2},
pages = {307--339}, 
year = {2002},
}

@article{Kuwert01willmore,
  title={The {W}illmore flow with small initial energy},
  author={Kuwert, Ernst and Sch{\"a}tzle, Reiner},
  journal={J. Differ. Geom. },
  volume={57},
  number={3},
  pages={409--441},
  year={2001},
  publisher={Lehigh University}
}

@article{Liu12mesoscale,
  title={Mesoscale simulations of curvature-inducing protein partitioning on lipid bilayer membranes in the presence of mean curvature fields},
  author={Liu, Jin and Tourdot, Richard and Ramanan, Vyas and Agrawal, Neeraj J and Radhakrishnan, Ravi},
  journal={Mol. Phys.},
  volume={110},
  number={11-12},
  pages={1127--1137},
  year={2012},
  publisher={Taylor \& Francis}
}

@article{LX25,
author = {Yihe Liu and Xianmin Xu},
title = {A variational discretization method for mean curvature flows by the {Onsager} principle},
volume={6}, 
url={https://www.global-sci.com/csiam-am/article/view/7878}, 
DOI={10.4208/csiam-am.SO-2024-0005}, 
number={1}, 
journal={CSIAM Trans. Appl. Math. }, 
year={2025}, 
month={Feb.}, 
pages={63--95} 
}

@article{LZ26,
      title={A minimizing-movement framework for geometric gradient flows with admissible tangential motion}, 
      author={Xiaoxiao Liu and Quan Zhao},
      year={2026},
      eprint={2606.18177},
      archivePrefix={arXiv},
    journal = {arXiv:2606.18177},
    url={https://arxiv.org/abs/2606.18177}, 
}

@article{DeTurck17,
  title={On approximations of the curve shortening flow and of the mean curvature flow based on the {DeTurck} trick},
  author={M. Elliott, Charles and Fritz, Hans},
  journal={IMA J. Numer. Anal. },
  volume={37},
  number={2},
  pages={543--603},
  year={2017},
  publisher={Oxford University Press}
}

@article{Mayer02numerical,
  title={A numerical scheme for axisymmetric solutions of curvature-driven free boundary problems, with applications to the {W}illmore flow},
  author={Mayer, Uwe F and Simonett, Gieri},
  journal={Interfaces Free Bound.},
  volume={4},
  number={1},
  pages={89--109},
  year={2002}
}

@InProceedings{OlischlagerR09,
author={Olischl{\"a}ger, N. and Rumpf, M.},
editor={Hancock, Edwin R. and Martin, Ralph R. and Sabin, Malcolm A.},
title={Two Step Time Discretization of {W}illmore Flow},
booktitle={Mathematics of Surfaces XIII},
year=2009,
publisher={Springer},
address={Berlin},
pages={278--292},
}

@article{Onsager31a,
  title={Reciprocal relations in irreversible processes. {I}.},
  author={Onsager, Lars},
  journal={Phys. Rev.},
  volume={37},
  number={4},
  pages={405--426},
  year={1931},
  publisher={APS}
}

@article{Onsager31b,
  title={Reciprocal relations in irreversible processes. {II}.},
  author={Onsager, Lars},
  journal={Phys. Rev.},
  volume={38},
  number={12},
  pages={2265--2279},
  year={1931},
  publisher={APS}
}

@article{PAN26,
title = {When surface evolution meets {Fokker-Planck} equation: A novel tangential velocity model for uniform parametrization},
author = {Jiangong Pan and Guozhi Dong and Hailong Guo and Zuoqiang Shi},
journal = {J. Comput. Phys.},
volume = {549},
pages = {114604},
year = {2026},
issn = {0021-9991},
doi = {https://doi.org/10.1016/j.jcp.2025.114604}
}

@article{Canham1970minimum,
  title={The minimum energy of bending as a possible explanation of the biconcave shape of the human red blood cell},
  author={Canham, Peter B},
  journal={J. Theor. Biol.},
  volume={26},
  number={1},
  pages={61--81},
  year={1970},
  publisher={Elsevier}
}

@article{PalmurellaR22,
  title={The parametric approach to the {W}illmore flow},
  author={Palmurella, Francesco and Rivi\`ere, Tristan},
  journal={Adv. Math.},
  volume={400},
  pages={108257},
  year={2022},
  doi={10.1016/j.aim.2022.108257}
}

@article{Rupp23volume,
  title={The volume-preserving {W}illmore flow},
  author={Rupp, Fabian},
  journal={Nonlinear Anal.},
  volume={230},
  pages={113220},
  year={2023},
  doi={10.1016/j.na.2023.113220}
}

@article{Rangamani21local,
  title={Local sensitivity analysis of the “membrane shape equation” derived from the {Helfrich} energy},
  author={Rangamani, Padmini and Behzadan, Ali and Holst, Michael},
  journal={Math. Mech. Solids},
  volume={26},
  number={3},
  pages={356--385},
  year={2021},
  publisher={Sage Publications Sage UK: London, England}
}

@Misc{RumpfSS25preprint,
author = {Martin Rumpf and Josua Sassen and Christoph Smoch},
title = {A hybrid minimizing movement and neural network approach to {W}illmore 
flow},
year = 2025,
howpublished = {arXiv:2502.14656},
url = {https://arxiv.org/abs/2502.14656}
}

@Article{Rusu05, comment = {was: rusu},
author = {Rusu, Raluca E.},
title = {An algorithm for the elastic flow of surfaces},
journal = {Interfaces Free Bound.},
fjournal = {Interfaces and Free Boundaries. Mathematical Modelling,
Analysis and Computation},
volume = 7,
year = 2005,
number = 3,
pages = {229--239},
issn = {1463-9963},
mrclass = {53C44 (65M60)},
mrnumber = {2171130},
DOI = {10.4171/IFB/122},
URL = {https://doi.org/10.4171/IFB/122},
}

@Article{Seifert97,
author = {Udo Seifert},
title = {Configurations of fluid membranes and vesicles},
journal = {Adv. Phys.},
volume = 46,
number = 1,
pages = {13--137},
year = 1997,
doi = {10.1080/00018739700101488},
url = {https://doi.org/10.1080/00018739700101488}
}

@article{Simonett05willmore,
  title={The {W}illmore flow near spheres},
  author={Simonett, Gieri},
  journal = {Differ. Integral Equ.},
  volume = 14,
  number = {8},
  pages = {1005--1014},
  year={2001}
}

@article{Schlierf25spont,
  title={Spontaneous curvature effects of the {Helfrich} flow: singularities and convergence},
  author={Schlierf, Manuel},
  journal={Comm. Partial Diff. Equ.},
  volume = {50}, 
  number={3}, 
  pages={441--476},
  year={2025},
}

@book{Willmore93,
  title={Riemannian geometry},
  author={Willmore, Thomas James},
  year={1993},
  publisher={Oxford University Press}
}

@article{wintz96starfish,
  title={Starfish vesicles},
  author={Wintz, W and D{\"o}bereiner, H-G and Seifert, U},
  journal={Europhys. Lett.},
  volume={33},
  number={5},
  pages={403--408},
  year={1996}
}

\end{document}